\documentclass[preprint]{imsart}

\RequirePackage[numbers]{natbib}

\RequirePackage[colorlinks,citecolor=blue,urlcolor=blue]{hyperref}

\arxiv{arXiv:1612.04740}

\usepackage{amssymb}
\usepackage{amsthm}
\usepackage{color}
\usepackage{mathtools}
\usepackage{tikz}
\usepackage{todonotes}
\usepackage{stmaryrd}

\DeclareMathOperator*{\argmin}{arg\,min}
\DeclareMathOperator{\crit}{crit}
\DeclareMathOperator{\conv}{conv}
\DeclareMathOperator{\diam}{diam}
\DeclareMathOperator{\Tr}{Tr}
\DeclareMathOperator{\e}{e}
\DeclareMathOperator{\hilbert}{\mathcal{H}}
\DeclareMathOperator{\Id}{Id}
\DeclareMathOperator{\pen}{pen}

\DeclareMathOperator{\penlin}{pen_{\ell}}
\DeclareMathOperator{\R}{\mathbb{R}}
\DeclareMathOperator{\X}{\mathcal{X}}

\newcommand{\A}{A}
\newcommand{\abs}[1]{\left\lvert#1\right\rvert}

\newcommand{\card}[1]{\left\lvert#1\right\rvert}

\newcommand{\cmax}{C_{\max}}
\newcommand{\cmin}{C_{\min}}
\newcommand{\condexpec}[2]{\mathbb{E}\left[#1\big|#2\right]}
\newcommand{\defeq}{\vcentcolon =}
\newcommand{\deltainf}{\underline{\Delta}}
\newcommand{\deltasup}{\overline{\Delta}}
\newcommand{\dhat}{D_{\tauhat}}
\newcommand{\distfrob}{\mathrm{d}_{\mathrm{F}}}
\newcommand{\disthaus}{\mathrm{d}_{\mathrm{H}}}
\newcommand{\distinf}{\mathrm{d}_{\infty}}
\newcommand{\dstar}{D_{\taustar}}
\newcommand{\dtau}{D_{\tau}}
\newcommand{\E}{\mathcal{E}}
\newcommand{\emprisk}{\widehat{\mathcal{R}}}
\newcommand{\eqdef}{= \vcentcolon}
\newcommand{\expec}[1]{\mathbb{E}\left[#1\right]}
\newcommand{\floor}[1]{\left\lfloor#1\right\rfloor}
\newcommand{\frobnorm}[1]{\norm{#1}_{\mathrm{F}}}
\newcommand{\gammainf}{\underline{\Gamma}}
\newcommand{\gaussian}{\mathcal{N}}
\newcommand{\hilbertinner}[2]{\left\langle#1,#2\right\rangle_{\hilbert}}
\newcommand{\hilbertnorm}[1]{\left\lVert#1\right\rVert_{\hilbert}}
\newcommand{\indic}[1]{\mathbf{1}_{\{#1\}}}
\newcommand{\inner}[2]{\left\langle#1,#2\right\rangle}
\renewcommand{\L}{L}
\newcommand{\lambdainf}{\underline{\Lambda}}
\newcommand{\lambdasup}{\overline{\Lambda}}
\newcommand{\lstar}{\lambda^{\star}}
\newcommand{\lzero}{\lambda^{\circ}}
\newcommand{\maxeps}{M_n}
\newcommand{\muhat}{\widehat{\mu}}
\newcommand{\mustar}{\mu^{\star}}
\newcommand{\norm}[1]{\left\lVert#1\right\rVert}
\newcommand{\opnorm}[1]{{\left\vert\kern-0.25ex\left\vert\kern-0.25ex\left\vert #1 
    \right\vert\kern-0.25ex\right\vert\kern-0.25ex\right\vert}}
\newcommand{\Q}{Q}
\newcommand{\segmentation}[1]{\left[#1\right]}
\newcommand{\set}[1]{\mathopen{} \left\{ #1 \right\} \mathclose{}}
\newcommand{\T}{\mathcal{T}}
\newcommand{\tauhat}{\widehat{\tau}}
\newcommand{\tauhatmom}{\tauhat_{2}} 

\newcommand{\tauzero}{\tau^{\circ}}
\newcommand{\redtauzero}{\widetilde{\tau}^{\circ}}
\newcommand{\taustar}{\tau^{\star}}
\newcommand{\redtaustar}{\widetilde{\tau}^{\star}}
\newcommand{\redmustar}{\widetilde{\mu}^{\star}}

\newcommand{\vitThmmoment}{v_2}
\newcommand{\vitThmbounded}{v_1}

\numberwithin{equation}{section}

\theoremstyle{plain}
\newtheorem{theorem}{Theorem}[section]
\newtheorem{proposition}{Proposition}[section]
\newtheorem{lemma}{Lemma}[section]
\newtheorem{corollary}{Corollary}[section]

\theoremstyle{definition}
\newtheorem{assumption}{Assumption}

\theoremstyle{remark}
\newtheorem{remark}{Remark}[section]
\newtheorem{example}{Example}[section]

\DeclareRobustCommand{\proba}[1]{\ensuremath{\mathbb{P}\left (#1\right )}} 

\newcommand{\dinf}{\distinf^{(1)}}
\newcommand{\dinfH}{\disthaus^{(1)}}
\newcommand{\dinfb}{\distinf^{(2)}}
\newcommand{\dinfbH}{\disthaus^{(2)}}
\newcommand{\dinfD}{\distinf^{(3)}}

\usepackage[nokeyprefix]{refstyle}

\endlocaldefs

\AtBeginDocument{%
   \def\MR#1{}
}

\begin{document}

\begin{frontmatter}

\begin{aug}
\title{Consistent change-point detection with kernels}
\runtitle{Change-point detection with kernels}

\author{\fnms{Damien} \snm{Garreau}\thanksref{t1}\ead[label=e1]{damien.garreau@ens.fr}}

\address{Centre de recherche Inria de Paris \\
2 rue Simone Iff \\
CS 42112 \\
75589 Paris Cedex 12 \\
\printead{e1}}

\vspace{0.4cm}

\and

\author{\fnms{Sylvain} \snm{Arlot}\ead[label=e2]{sylvain.arlot@math.u-psud.fr}}
\address{Laboratoire de Math\'ematiques d'Orsay \\
Univ. Paris-Sud, CNRS, Universit\'e Paris-Saclay, \\
91405 Orsay, France \\
\printead{e2}}

\thankstext{t1}{Corresponding author}

\runauthor{D. Garreau and S. Arlot}

\end{aug}

\begin{abstract}
In this paper we study the kernel change-point algorithm (KCP) proposed 
by Arlot, Celisse and Harchaoui \citep{Arl_Cel_Har:2012}, 
which aims at locating an unknown number of change-points in 
the distribution of a sequence of independent data taking values in an arbitrary set. 
The change-points are selected by model selection with a penalized 
kernel empirical criterion. 
We provide a non-asymptotic result showing that, 
with high probability, the KCP procedure retrieves the correct number 
of change-points, provided that the constant in the penalty is well-chosen; 
in addition, KCP estimates the change-points location at the optimal rate. 
As a consequence, when using a characteristic kernel, 
KCP detects all kinds of change in the distribution (not only changes in the mean 
or the variance), 
and it is able to do so for complex structured data 
(not necessarily in $\R^d$). 
Most of the analysis is conducted assuming that the kernel is bounded; 
part of the results can be extended when we only assume a finite second-order moment.
\end{abstract}



\received{\smonth{6} \syear{2017}}

\end{frontmatter}

\section{Introduction}

In many situations, some properties of a time series change over time, 
such as the mean, the variance or higher-order moments.
Change-point detection is the long standing question of finding 
both the number and the localization of such changes.
This is an important front-end task in many applications.
For instance, detecting changes occuring in comparative genomic hybridization array data 
(CGH arrays) is crucial to the early diagnosis of cancer~\citep{Lai_Joh_Kuc:2005}.
In finance, some intensively examined time series like the volatility process 
exhibit local homogeneity and it is useful to be able to segment these time series 
both for modeling and forecasting~\citep{Lav_Tey:2006,Spo:2009}.
Change-point detection can also be used to detect changes 
in the activity of a cell~\citep{Rit_Raz_Ber:2002}, 
in the structure of random Markov fields~\citep{Liu_Suz_Rel:2014}, 
or a sequence of images~\citep{Kim_Mar_Per:2009,Abo_Gou_Blo:2015}. 
Generally speaking, it is of interest to the practitioner to segment a time series 
in order to calibrate its model on homogeneous sets of datapoints.

Addressing the change-point problem in practice requires to face 
several important challenges. 
First, the number of changes can not be assumed to be known in advance --- in particular, it can not be assumed to be equal to~$0$ or~$1$ ---, hence a practical change-point procedure 
must be able to infer the number of changes from the data. 
Second, changes do not always occur in the mean or the variance of the data, 
as assumed by most change-point procedures. 
We need to be able to detect changes in other features of the distribution. 
Third, parametric assumptions --- which are often made for building 
or for analyzing change-point procedures --- are 
often unrealistic, so that we need a fully non-parametric approach. 
Fourth, data points in the time series we want to segment 
can be high-dimensional and/or structured. 
If the dimensionality is larger than the number of observations, 
a non-asymptotic analysis is mandatory for theoretical results to be meaningful. 
When data are structured --- for instance, histograms, graphs or strings ---, 
taking their structure into account seems necessary for detecting 
efficiently the change-points. 

We focus only on the \emph{offline} problem in this article, that is, 
when all observations are given at once, 
as opposed to the situation where data come as a continuous stream.
We refer to \citet{Tar_Bas_Nik:2014} for an extensive review of sequential methods, 
which are adapted to the later situation. 
Numerous offline change-point procedures have been proposed 
since the seminal works of \citet{Pag:1955}, \citet{Fis:1958} and \citet{Bel:1961}, 
which are mostly parametric in essence.
We refer to \citet[Chapter 2]{Bro_Dar:2013} for a review of non-parametric offline change-point detection methods.
Among recent works in this direction, we can mention the Wild Binary Segmentation (WBS,~\citep{Fry:2014}) and the non-parametric multiple change-point detection procedure (NMCD,~\citep{Zou_Yin_Fen:2014}).
Some authors also consider the case of high-dimensional data when only 
a few coordinates of the mean change at each change-point \citep[and references therein]{Wan_Sam:2016}, 
or the problem of detecting gradual changes \citep{Vog_Det:2015}; 
this paper does not address these slightly different problems. 

To the best of our knowledge, no offline change-point procedure addressed simultaneously the four challenges 
mentioned above, until the kernel change-point procedure (KCP) 
was proposed by \citet{Arl_Cel_Har:2012}. 
In short, KCP mixes the penalized least-squares approach 
to change-point detection \citep{Com_Roz:2004,Leb:2005} with 
semi-definite positive kernels \citep{Aro:1950}. 
It is not the only procedure that uses positive semi-definite kernels to detect changes in a times series.
Apart from \citet{Har_Cap:2007}, who introduced KCP for a fixed number of change-point, 
and \citet{Arl_Cel_Har:2012} who extended KCP to an unknown number of change-points, 
we are aware of several closely related work.
Maximum Mean Discrepancy \citep[MMD,][]{Gre_Bor_Ras:2006} has been used for building two sample tests; 
a block average version of the MMD, named the $M$-statistic, has lead to an online change-point detection procedure \citep{Li_Xie_Dai:2015}. 
A kernel-based statistic, named kernel Fisher discriminant ratio, has been used by \citet{Har_Mou_Bac:2009} 
for homogeneity testing  and for detecting one change-point. 
\citet{Sha_Tew_Wen:2016} build an analogue of the CUSUM statistic for Hilbert-valued random variables in order to detect a single change in the mean, and could be applied in our setting to the images of the observations in the feature space.
Kernel change detection~\citep{Des_Dav_Don:2005} is an online procedure that uses a kernel to build a dissimilarity measure between the near past and future of a data-point.

On the computational side, the KCP segmentation can be computed efficiently thanks to 
a dynamic programming algorithm \citep{Har_Cap:2007,Arl_Cel_Har:2012}, 
which can be made even faster \citep{Cel_Mor_Mar_Rig:2016}. 
An oracle inequality for KCP is proved by \citet{Arl_Cel_Har:2012}; 
this is not exactly a result on change-point estimation, but a guarantee 
on estimation of the ``mean'' of the time series 
in the RKHS associated with the kernel chosen. 
The good numerical performance of KCP in terms of change-point 
estimation is also demonstrated in several experiments. 

So, a key theoretical question remains open: 
does KCP estimate correctly the number of change-points 
and their locations with a large probability? 
If yes, at which speed does KCP estimate the change-point locations?

\medbreak

This paper answers these questions, showing that KCP has good theoretical properties 
for change-point estimation with independent data, under a boundedness assumption 
(Theorem~\ref{thm.bounded} in Section~\ref{sec:main:Dh}). 
This result is non-asymptotic, hence meaningful for high-dimensional or complex data. 
In the asymptotic setting --- with a fixed true segmentation and more and more data 
points observed within each segment ---, 
Theorem~\ref{thm.bounded} implies that KCP estimates consistently all changes 
in the ``kernel mean'' of the distribution of data, 
at speed $\log(n)/n$ with respect to the sample size $n$.
Since we make no assumptions on the minimal size of the true segments, 
this matches minimax lower bounds~\citep{Bru:2014}. 
We also provide a partial result under a weaker finite variance assumption 
(Theorem~\ref{th:localization-moment} in Section~\ref{sec:main:extension}) 
and explain in Section~\ref{sec:conclu} how our proofs could be extended 
to other settings, including the dependent case. 
These findings are illustrated by numerical simulations in Section~\ref{sec:simulation}.

An important case is when KCP is used with a characteristic kernel \citep{Fuk_etal:2008}, 
such as the Gaussian or the Laplace kernel. 
Then, any change in the distribution of data induces a change in the ``kernel mean''. 
So, Theorem~\ref{thm.bounded} implies that KCP then estimates  
consistently and at the minimax rate {\em all changes\/} in the distribution of the data, 
without any parametric assumption and without prior knowledge about the number of changes. 

Our results also are interesting regarding to the theoretical understanding 
of least-squares change-point procedures. 
Indeed, when KCP is used with the linear kernel, 
it reduces to previously known penalized least-squares change-point 
procedures \cite[for instance]{Yao:1988,Com_Roz:2004,Leb:2005}. 
There are basically two kinds of results on such procedures in the change-point literature: 
(i) asymptotic statements on change-point estimation 
\citep{Yao:1988,Yao_Au:1989,Bai_Per:1998,Lav_Mou:2000} 
and (ii) non-asymptotic oracle inequalities \citep{Com_Roz:2004,Leb:2005,Arl_Cel_Har:2012}, 
which are based upon concentration inequalities and model selection theory \citep{Bir_Mas:2001} 
but do not directly provide guarantees on the estimated change-point locations. 
Our results and their proofs show how to reconciliate the two approaches 
when we are interested in change-point locations, 
which is already new for the case of the linear kernel, 
and also holds for a general kernel. 

\section{Kernel change-point detection}

This section describes the general change-point problem and the 
kernel change-point procedure~\citep{Arl_Cel_Har:2012}. 

\subsection{Change-point problem}
\label{sec:setting}

Set~$2\leq n <+\infty$ and consider $X_1,\ldots,X_n$ independent $\X$-valued random variables, 
where~$\X$ is an arbitrary (measurable) space.
The goal of change-point detection is to detect abrupt changes in the distribution of the~$X_i$s.
For any $D\in\set{1,\ldots,n}$ and any integers $0=\tau_0 <\tau_1 <\cdots <\tau_D=n$, 
we define the {\em segmentation\/} $\tau\defeq\bigl[\tau_0,\ldots,\tau_D\bigr]$ of $\set{1,\ldots,n}$ 
as the collection of segments 
$\lambda_{\ell}=\set{\tau_{\ell-1}+1,\ldots,\tau_{\ell}}$, $\ell\in\set{1,\ldots,D}$.
We call {\em change-points\/} the right-end of the segments, 
that is the~$\tau_{\ell}$, $\ell\in\set{1,\ldots,D}$.
We denote by~$\T_n^D$ the set of segmentations with~$D$ segments and 
$\T_n\defeq\bigcup_{D=1}^n \T_n^D$ the set of all segmentations of~$\set{1,\ldots,n}$.
For any~$\tau\in\T_n$, we write~$\dtau$ for the number of segments of~$\tau$.
Figure~\ref{fig:segmentation-example} provides a visual example.

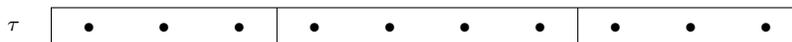
\begin{figure}[ht]
\centering
\begin{tikzpicture}
\def\offset{0}
\draw (0,0+\offset) rectangle (10,0.5+\offset) ;
\draw (3,0.5+\offset) -- (3,0+\offset) ;
\draw (7,0.5+\offset) -- (7,0+\offset) ;
\draw (0.5,0.23\offset) node {$\bullet$} ;
\draw (1.5,0.23+\offset) node {$\bullet$} ;
\draw (2.5,0.23+\offset) node {$\bullet$} ;
\draw (3.5,0.23+\offset) node {$\bullet$} ;
\draw (4.5,0.23+\offset) node {$\bullet$} ;
\draw (5.5,0.23+\offset) node {$\bullet$} ;
\draw (6.5,0.23+\offset) node {$\bullet$} ;
\draw (7.5,0.23+\offset) node {$\bullet$} ;
\draw (8.5,0.23+\offset) node {$\bullet$} ;
\draw (9.5,0.23+\offset) node {$\bullet$} ;
\draw (-0.5,0.25+\offset) node {$\tau$} ;
\end{tikzpicture}
\caption{We often represent the segmentations as above. 
The bullet points stand for the elements of~$\set{1,\ldots,n}$. 
Here, $n=10$, $D_{\tau}=3$, $\tau_0 = 0$, $\tau_1=3$, $\tau_2=7$ and $\tau_3 = 10$.}
\label{fig:segmentation-example}
\end{figure}

An important example to have in mind is the following. 
\begin{example}[Asymptotic setting]
\label{ex.asymptotic-setting}
Let $K \geq 1$, 
$0 = b_0 < b_1 < \cdots < b_K < b_{K+1} = 1$ 
and $P_1,\ldots,P_{K+1}$  some probability distributions on~$\X$ 
be fixed. 
Then, for any~$n$ and $i \in \{1, \ldots, n\}$, 
we set $t_i\defeq i/n$ and the distribution of $X_i$ is 
$P_{j(i)}$ where $j(i)$ is such that $t_i \in [b_j, b_{j+1})$. 
In other words, we have a fixed segmentation of $[0,1]$, 
given by the $b_j$, 
a fixed distribution over each segment, 
given by the $P_j$, 
and we observe independent realizations from the distributions 
at discrete times $t_1, \ldots, t_n$. 
The corresponding true change-points in $\{0 , \ldots, n\}$ 
are the $\floor{n b_j}$, $j=1, \ldots, K$. 
For~$n$ large enough, 
there are $K+1$ segments. 
Figure~\ref{fig:asymptotic} shows an example.
Let us emphasize that in this setting, $n$ going to infinity 
does not mean that new observations are observed over time. 
Recall that we consider the change-point problem {\em a posteriori\/}: 
a larger~$n$ means that we have been able to observe the phenomenon of interest 
with a finer time discretization. 
Also note that this asymptotic setting is restrictive in the sense that segments size asymptotically are of order $n$; 
we do not make this assumption in our analysis, which also covers asymptotic settings where some segments have a smaller size. 
\end{example}
\begin{figure}[!h]
\includegraphics[scale=0.5]{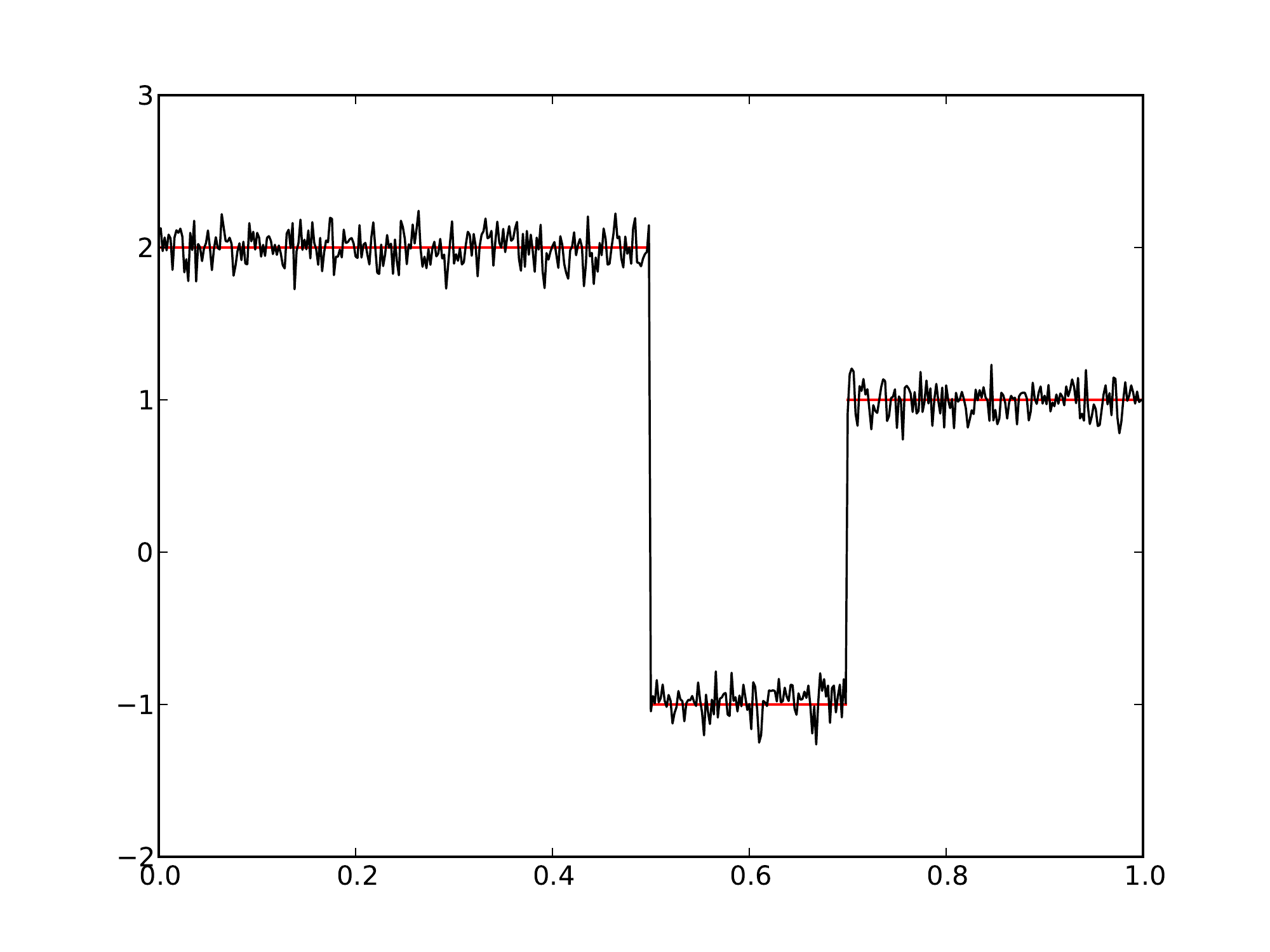}
\caption{\label{fig:asymptotic} 
Illustration of the asymptotic setting (Example~\ref{ex.asymptotic-setting}) 
in the case of changes in the mean of the $X_i$. 
Here, $\X=\R$, $X_i = f(t_i) + \varepsilon_i$ with $\varepsilon_1, \ldots, \varepsilon_n$ 
i.i.d. and centered, and $f:[0,1]\to\R$ is a (fixed) piecewise constant function 
(shown in red). 
The goal is to recover the number of abrupt changes of $f$ 
(here, $2$) 
and their locations ($b_1 = 0.5$ and $b_2 = 0.7$). 
Note that other kinds of changes in the distribution of the $X_i$ 
can be considered, 
see Section~\ref{sec:simulation}.  
}
\end{figure}

\subsection{Kernel change-point procedure (KCP)}
\label{sec:KCP}

Let $k:\X\times \X\to\R$ be a positive semidefinite kernel, that is, 
a measurable function such that the matrix $\left(k(x_i,x_j)\right)_{1\leq i,j\leq m}$ 
is positive semidefinite for any $m\geq 1$ and $x_1,\ldots,x_m\in\X$~\citep{Sch_Smo:2002}.
Classical examples of kernels are given by \cite[section~3.2]{Arl_Cel_Har:2012}, among which: 
\begin{itemize}
\item the {\em linear kernel\/}: 
$k^{\mathrm{lin}}(x,y) = \inner{x}{y}_{\R^p}$ 
for $x,y \in \X = \R^p$.  

\item the {\em polynomial kernel\/} of order $d \geq 1$: 
$k^{\mathrm{poly}}_d(x,y) = \bigl( \inner{x}{y}_{\R^p} + 1)^d$ 
for $x,y \in \X = \R^p$.  

\item the {\em Gaussian kernel\/} with bandwidth $h>0$:  
$k^{\mathrm{G}}_{h}(x,y) = \exp[ - \lVert x-y \rVert^2 / (2 h^2) ]$ 
for $x,y \in \X = \R^p$. 

\item the {\em Laplace kernel\/} with bandwidth $h>0$: 
$k^{\mathrm{L}}_{h}(x,y) = \exp [ - \lVert x-y \rVert / (2 h^2) ]$ 
for $x,y \in \X = \R^p$. 

\item the $\chi^2$-kernel: 
$k_{\chi^2} (x,y) = \exp\left( -\frac{1}{2}\sum_{i=1}^p \frac{( x_i - y_i )^2}{x_i+y_i} \right)$ 
for $x,y \in \X$ the $p$-dimensional simplex. 

\end{itemize}

As done by \citet{Har_Cap:2007} and \citet{Arl_Cel_Har:2012}, 
for a given segmentation~$\tau\in\T_n^D$, 
we assess the adequation of~$\tau$ with the \emph{kernel least-squares criterion} 
\begin{eqnarray}
\label{eq:def-empirical-risk}
\emprisk_n(\tau)
&\defeq& \frac{1}{n}\sum_{i=1}^n k(X_i,X_i) \\
&- &\frac{1}{n} \sum_{\ell=1}^D\left[\frac{1}{\tau_{\ell}-\tau_{\ell-1}} 
\sum_{i=\tau_{\ell-1}+1}^{\tau_{\ell}} \sum_{j=\tau_{\ell-1}+1}^{\tau_{\ell}}k(X_i,X_j)\right] 
\, .\notag
\end{eqnarray}
Elementary algebra shows that, when $\X=\R^p$ and $k = k^{\mathrm{lin}}$, 
$\emprisk_n$ is the usual least-squares criterion. 
Minimizing this criterion over the set of all segmentations always 
outputs the segmentation with~$n$ segments reduced to a point, 
that is $\segmentation{0,\ldots,n}$; this is a well-known overfitting phenomenon. 
To counteract this, a classical idea \cite[for instance]{Lav:2005} 
is to minimize a penalized criterion 
$\crit(\tau)\defeq \emprisk_n(\tau) + \pen(\tau)$, where $\pen :\T_n\to\R_+$ is called the penalty. 
Formally, the kernel change-point procedure (KCP) of \citet{Arl_Cel_Har:2012} 
selects the segmentation 
\begin{equation}
\label{eq:original-problem}
\tauhat \in\argmin_{\tau\in\T_n}\bigl\{ \crit(\tau)\bigr\} 
\qquad \text{where} \qquad 
\crit(\tau) = \emprisk_n(\tau)+\pen(\tau) 
\, .
\end{equation}
In this paper, we focus on the classical choice of 
a penalty proportional to the number of segments, similarly to AIC, BIC and~$C_p$ criteria. 
Namely, we consider
\begin{equation}
\label{eq:def-penalty-alt}
\pen(\tau) = \penlin(\tau)\defeq \frac{CM^2\dtau}{n} 
\, ,
\end{equation}
where $C$ is a positive constant and $M$ is specified in Assumption~\ref{assump:bounded-kernel} later on.
As mentioned in the Introduction, slightly different penalty shapes can be considered, 
as suggested by \citet{Arl_Cel_Har:2012}. 
Our results could be extended to the penalty of \citet{Arl_Cel_Har:2012}, 
but we choose to consider the linear penalty \eqref*{eq:def-penalty-alt} 
only for simplicity.

\subsection{The reproducing kernel Hilbert space}

Let~$\hilbert$ be the reproducing kernel Hilbert space (RKHS) associated 
with~$k$~\citep{Aro:1950}, 
together with the canonical feature map $\Phi :\X\to\hilbert$ 
\[\begin{array}{ccccc}
\Phi & : & \X & \to & \hilbert \\
 & & x & \mapsto & \Phi(x)\defeq k(\cdot,x) \, . 
\end{array}\]
We write~$\hilbertinner{\cdot}{\cdot}$ (resp.~$\hilbertnorm{\cdot}$) for the inner product 
(resp. the norm) of~$\hilbert$.
For any $i\in\set{1,\ldots,n}$, define $Y_i\defeq \Phi(X_i)\in\hilbert$. 
In the case where $k=k^{\mathrm{lin}}$, then $Y_i=\inner{\cdot}{X_i}_{\R^p}$ and the empirical risk 
$\emprisk_n$ reduces to the least-squares criterion
\[
\emprisk_n(\tau) 
= \frac{1}{n}\sum_{\ell =1}^{\dtau} \sum_{i=\tau_{\ell -1}+1}^{\tau_{\ell}} 
\left(X_i-\overline{X}_{\ell}\right)^2
\, ,
\]
where $\overline{X}_{\ell}$ is the empirical mean of the~$X_i$ 
over the segment $\set{\tau_{\ell-1}+1,\ldots,\tau_{\ell}}$.
It is well-known that penalized least-squares procedures detect changes in the mean of the observations~$X_i$, see \citet{Yao:1988}.
Hence the kernelized version of this least-squares procedure, KCP, should detect changes in the ``mean'' of the $Y_i = \Phi(X_i)$, 
which are a nonlinear transformation of the $X_i$.

More precisely, assume that $\hilbert$ is separable and that 
\[ 
\forall i\in\set{1,\ldots,n} , \qquad 
\expec{\sqrt{k(X_i,X_i)}}< +\infty
\, . 
\]
Then~$\mustar_i$, the Bochner integral of~$Y_i$, is well-defined~\citep{Led_Tal:2013}.
The condition above is satisfied in our setting 
(when either Assumption~\ref{assump:bounded-kernel} or Assumption~\ref{assump:bounded-variance} holds true, see Section~\ref{sec.KCP.assumptions}), 
and~$\hilbert$ is separable in most cases~\citep{Die_Bac:2014}. 
The Bochner integral commutes with continuous linear operators, hence the following property holds, which will be of common use:
\[
\forall g\in\hilbert,\quad \hilbertinner{\mustar_i}{g} 
= \mathbb{E} \bigl[ g(X_i) \bigr] 
= \mathbb{E} \bigl[ \hilbertinner{Y_i}{g} \bigr]
\, .
\]

\medbreak

We now define the ``true segmentation'' $\taustar \in \T_n$ by 
\begin{equation}
\label{def.taustar}
\begin{split}
\mustar_1 = \cdots = \mustar_{\taustar_1},
\quad 
\mustar_{\taustar_1 + 1} = \cdots = \mustar_{\taustar_2}, 
\quad \cdots \quad 
\mustar_{\taustar_{\dstar - 1} + 1} = \cdots =\mustar_n 
\\
\text{and} \quad 
\forall i \in \{ 1, \ldots, \dstar - 1\} , \quad \mustar_{\taustar_i}\neq\mustar_{\taustar_{i+1}} 
\end{split}
\end{equation}
with $1\leq \taustar_1 < \cdots < \taustar_{\dstar - 1} \leq n$. 
We call the~$\taustar_i$s the {\em true\/} change-points. 
It should be clear that it is always possible to define~$\taustar$.

A kernel is said to be characteristic if the mapping 
$P\mapsto \mathbb{E}_{X\sim P}\left[\Phi(X)\right]$ is injective, 
for~$P$ belonging to the set of Borel probability measures on~$\X$~\citep{Sri_Fuk_Gre:2009}.
In simpler terms, when $k$ is a characteristic kernel, $X_i$ and~$X_{i+1}$ 
have the same distribution if and only if 
$\mustar_i = \mustar_{i+1}$, 
and~$\taustar$ indeed corresponds to the set of changes in the 
distribution of the $X_i$. 
For instance, all strictly positive definite kernels are characteristic, including the Gaussian kernel, see \citet{Sri_Fuk_Gre:2009}.
Therefore, in the setting of Example~\ref{ex.asymptotic-setting}, for $n$ large enough, 
$\dstar = K+1$ and $\taustar_{\ell} = \floor{n b_{\ell}}$ for 
$\ell = 1, \ldots, K$. 

For a general kernel, some changes of~$P_{X_i}$, the distribution of $X_i$, 
might not appear in~$\taustar$. 
For instance, with the linear kernel,~$\taustar$ only corresponds to changes 
of the mean of the~$X_i$. 
In most cases, a characteristic kernel is known 
and we can choose to use KCP with a characteristic kernel; 
then, as we prove in the following, KCP eventually detects any change in the distribution 
of the observations. 
But one can also choose a non-characteristic kernel on purpose, 
hence focusing only on some changes in the distribution of the~$X_i$.
For instance, the polynomial kernel of order $d$ is not characteristic 
and leads to the detection of changes in the first $d$ moments of the distribution; 
with the linear kernel, KCP detects changes in the mean of the~$X_i$. 

From now on, we focus on the problem of detecting the changes of~$\taustar$ only, 
whether the kernel is characteristic or not.

\subsection{Rewriting the empirical risk}
\label{sec:rewriting-emp-risk}

It is convenient to see the images of the observations by the feature map as an element of~$\hilbert^n$.
To this extent, we define $Y\defeq (Y_1,\ldots,Y_n)$, as well as 
$\mustar\defeq (\mustar_1,\ldots,\mustar_n)\in\hilbert^n$ 
and $\varepsilon\defeq Y-\mustar\in\hilbert^n$.
We identify the elements of~$\hilbert^n$ with the set of applications 
$\{1,\ldots,n\} \to \hilbert$, naturally embedded with 
the inner product and norm given by 
\[
\forall x,y\in\hilbert^n,\qquad 
\inner{x}{y} 
\defeq \sum_{i=j}^n \hilbertinner{x_j}{y_j} 
\qquad \text{and}\qquad 
\norm{x}^2\defeq \sum_{j=1}^n \hilbertnorm{x_j}^2 
\, .\]
We now rewrite the empirical risk 
as a function of~$\tau$ and~$Y$.
For any segmentation~$\tau\in\T_n$, define~$F_{\tau}$ the set of applications  
$\{1,\ldots,n\} \to \hilbert$ that are constant over the segments of~$\tau$.
We see~$F_{\tau}$ as a subspace of~$\hilbert^n$ as a vector space.
Take $f\in\hilbert^n$, we define $\Pi_{\tau} f$ the orthogonal projection of~$f$ 
onto~$F_{\tau}$ with respect to~$\norm{\cdot}$: 
\[
\Pi_{\tau} f \in\argmin_{g\in F_{\tau}}\norm{f-g}
\, .
\]
It is shown by \citet{Arl_Cel_Har:2012} that for any $f\in\hilbert^n$ and any $\ell \in \{1, \ldots, \dtau\}$, 
\begin{equation}
\label{eq:computation-projection}
\forall i \in \{ \tau_{\ell-1} + 1 , \ldots,  \tau_{\ell} \}, 
\qquad 
(\Pi_{\tau} f)_i 
= \frac{1}{\abs{\tau_{\ell}-\tau_{\ell-1}}}\sum_{j=\tau_{\ell-1}+1}^{\tau_{\ell}}f_j
\, .
\end{equation}

We are now able to write the empirical risk as
\begin{equation}
\label{eq:empirical-risk-alt}
\emprisk_n(\tau) = \frac{1}{n}\norm{Y-\muhat_{\tau}}^2,
\end{equation}
where $\muhat_{\tau} = \Pi_{\tau}Y$, following \citep{Har_Cap:2007,Arl_Cel_Har:2012}.

\subsection{Assumptions}
\label{sec.KCP.assumptions}

A key ingredient of our analysis is the concentration of~$\varepsilon$. 
Intuitively, the performance of KCP is better 
when~$\varepsilon$ concentrates strongly around its mean, 
since without noise we are just given the task to segment a piecewise-constant signal.
It is thus natural to make assumptions on~$\varepsilon$ 
in order to obtain concentration results.
We actually formulate assumptions on the kernel~$k$, 
which translate automatically onto~$\varepsilon$.

As done by \citet{Arl_Cel_Har:2012}, the main hypothesis used in our analysis is the following. 
\begin{assumption}
\label{assump:bounded-kernel}
A positive constant~$M$ exists such that
\[
\forall i \in \{1, \ldots, n\} , \qquad 
k(X_i,X_i)\leq M^2 < +\infty 
\qquad \text{a.s.} 
\]
\end{assumption}
If Assumption~\ref{assump:bounded-kernel} holds true, 
\[
\forall i \in \{1, \ldots, n\} , \qquad 
\hilbertnorm{Y_i} = \sqrt{k(X_i,X_i)} \leq M 
\qquad \text{a.s.} 
\]
and \citet{Arl_Cel_Har:2012} show that  
$\hilbertnorm{\varepsilon_i} \leq 2M$ almost surely. 

Assumption~\ref{assump:bounded-kernel} is always satisfied for a large class of commonly used kernels, 
such as the Gaussian, Laplace and $\chi^2$ kernels.

Note that Assumption~\ref{assump:bounded-kernel} is weaker than 
assuming $k$ to be bounded --- that is, $k(x,x)\leq M$ for any $x\in\X$, 
which is equivalent to $k(x,x')\leq M$ for any $x,x'\in\X$ since~$k$ is positive definite.
For instance, if $\X = \R^p$ and the data $X_i$ are bounded almost surely, 
Assumption~\ref{assump:bounded-kernel} holds true for the linear kernel  
and all polynomial kernels, which are not bounded on $\R^p$. 

In the setting of Example~\ref{ex.asymptotic-setting}, 
Assumption~\ref{assump:bounded-kernel} holds true when 
\[ 
\forall j \in \{1, \ldots, K \}, \qquad 
k(x,x) \leq M^2 \qquad \text{for $P_j$-a.e. } x \in \X 
\, . 
\]

\medbreak

It is sometimes possible to weaken Assumption~\ref{assump:bounded-kernel} into a finite variance assumption. 
\begin{assumption}
\label{assump:bounded-variance}
A positive constant~$V < +\infty$ exists such that
\[\max_{1\leq i\leq n} \expec{\hilbertnorm{\varepsilon_i}^2}\leq V.\]
\end{assumption}
Since $v_i\defeq \mathbb{E} [ \hilbertnorm{\varepsilon_i}^2 ] 
=\expec{k(X_i,X_i)}-\hilbertnorm{\mustar_i}^2$, Assumption~\ref{assump:bounded-variance} holds true when 
\[
\forall i\in\set{1,\ldots,n}, \qquad 
\mathbb{E} \bigl[ k(X_i,X_i) \bigr] \leq V
\, . 
\]
As a consequence, Assumption~\ref{assump:bounded-kernel} implies Assumption~\ref{assump:bounded-variance} with~$V=M^2$. 
Note that Assumption~\ref{assump:bounded-variance} is satisfied for 
the polynomial kernel of order $d$ 
provided that
\[\forall i\in\set{1,\ldots,n},\quad \expec{\norm{X_i}^{2d}}<+\infty.\]

In the setting of Example~\ref{ex.asymptotic-setting}, 
Assumption~\ref{assump:bounded-variance} holds true with 
\[
V = \max_{ 1 \leq \ell \leq K+1} \mathbb{E}_{X \sim P_{\ell}} \bigl[ k(X, X) \bigr] 
\, , 
\]
provided this maximum is finite.

\section{Theoretical guarantees for KCP}
\label{sec:main} 
We are now able to state our main results. 
In Section~\ref{sec:main:Dh}, we state the main result of the paper, 
Theorem~\ref{thm.bounded}, which provides simple conditions under which 
KCP recovers the correct number of segments 
and localizes the true change-points with high probability, 
under the bounded kernel Assumption~\ref{assump:bounded-kernel}. 
Then, Section~\ref{sec:main:metric} details a few classical losses between segmentations 
which can be considered in addition to the one used in Theorem~\ref{thm.bounded}. 
Corollary~\ref{cor:Frobenius.bounded} formulates a result on $\tauhat$ 
in terms of the Frobenius loss. 
Finally, Section~\ref{sec:main:extension} states a partial result on KCP 
--- requiring the number of change-points $\dstar$ to be known --- 
under the weaker Assumption~\ref{assump:bounded-variance}.

\subsection{Main result}
\label{sec:main:Dh} 
We first need to define some quantities. 
The size of the smallest jump of $\mustar$ in~$\hilbert$ is
\begin{equation}
\label{eq:def-deltainf}
\deltainf 
\defeq \min_{i \,/\, \mustar_i\neq \mustar_{i+1}} 
\hilbertnorm{\mustar_i - \mustar_{i+1}} \, .
\end{equation}
Intuitively, the higher~$\deltainf$ is, the easier it is to detect the smallest jump with our procedure. 
The quantity $\hilbertnorm{\mustar_i - \mustar_{i+1}}$ is often called the (population) maximum mean discrepancy 
\citep[MMD,][]{Gre_Bor_Ras:2006} between the distributions of $X_i$ and $X_{i+1}$. 
In the scalar setting (with the linear kernel), 
the ratio $\deltainf/\sigma$ (where $\sigma^2$ is the variance of the noise) is called the \textit{signal-to-noise ratio}~\citep{Bas_Nik:1993} 
and is often used as a measure of the magnitude of a change in the signal. 
In Example~\ref{ex.asymptotic-setting}, 
\[
\deltainf 
= \min_{1 \leq j \leq K} \hilbertnorm{\mustar_{P_j} - \mustar_{P_{j+1}}} 
\]
where $\mustar_{P_j}$ denotes the (Bochner) expectation of $\Phi(X)$ when $X \sim P_j$.

For any $\tau\in\T_n$, we denote the (normalized) sizes of 
its smallest and of its largest segment by
\begin{equation}
\label{eq:def-lambdainf}
\lambdainf_{\tau}\defeq \frac{1}{n}\min_{1\leq \ell \leq \dtau}\card{\tau_{\ell}-\tau_{\ell-1}} 
\qquad \text{and}\qquad 
\lambdasup_{\tau}\defeq \frac{1}{n}\max_{1\leq\ell\leq \dtau}\card{\tau_{\ell}-\tau_{\ell-1}} 
\, .
\end{equation}
It should be clear that the smaller $\lambdainf_{\taustar}$ is, the harder it is to detect 
the segment that achieves the minimum in \eqref{eq:def-lambdainf}.
For instance, in the particular case of Example~\ref{ex.asymptotic-setting}, 
\[ 
\lambdainf_{\taustar} \xrightarrow[n \to +\infty]{} 
\min_{0 \leq j \leq K} \lvert b_{j+1} - b_j \rvert 
\qquad \text{and} \qquad 
\lambdasup_{\taustar} \xrightarrow[n \to +\infty]{} 
\max_{0 \leq j \leq K} \lvert b_{j+1} - b_j \rvert 
\, . 
\]

For any~$\tau^1$ and~$\tau^2\in\T_n$, we define 
\begin{align*}
\dinf(\tau^1,\tau^2) 
&\defeq \max_{1\leq i\leq D_{\tau^1} - 1} \biggl\{ \min_{1\leq j\leq D_{\tau^2}-1} 
\abs{\tau_i^1 - \tau_j^2} \biggr\} 
\, ,
\end{align*}
which is a loss function (a measure of dissimilarity) 
between the segmentations $\tau^1$ and $\tau^2$. 
Note that $\dinf$ is not a distance; 
other possible losses between segmentations and their relationship with $\dinf$ 
are discussed in Section~\ref{sec:main:metric}. 

\begin{theorem}
\label{thm.bounded}
Suppose that Assumption~\ref{assump:bounded-kernel} holds true.
For any $y>0$, an event $\Omega$ of probability at least $1-\e^{-y}$ 
exists on which the following holds true. 
For any $C>0$, 
let $\tauhat$ be defined as in \eqref[name=Eq.~]{eq:original-problem} 
with~$\pen$ defined by \eqref[name=Eq.~]{eq:def-penalty-alt}. 
Set 
\[
\cmin \defeq \frac{74}{3}(\dstar+1)(y+\log n + 1)\quad\text{and}\quad\cmax \defeq \dfrac{\deltainf^2}{M^2} 
\frac{\lambdainf_{\taustar} }{6 \dstar} n 
\, .
\]
Then, if 
\begin{equation}
\label{eq:hyp-main-result-synth}
\cmin < C < \cmax 
\, , 
\end{equation}
on $\Omega$, we have 
\[
D_{\tauhat} = \dstar 
\qquad \text{and} \qquad 
\frac{1}{n}\dinf\bigl(\taustar,\tauhat\bigr)
\leq  \vitThmbounded(y) 
\defeq 
\frac{148 \dstar M^2}{ \deltainf^2} \cdot \frac{y+\log n + 1}{n} 
\, .
\]
\end{theorem}
We delay the proof of Theorem~\ref{thm.bounded} to Section~\ref{sec:proof-main-result}. 
Some remarks follow.

\medbreak

Theorem~\ref{thm.bounded} is a non-asymptotic result: 
it is valid for any $n \geq 1$ and there is nothing hidden 
in $\mathrm{o}(1)$ remainder terms. 
The latter point is crucial for complex data --- for instance, 
$\X = \R^p$ with $p > n$ --- since in this case, 
assuming $\X$ fixed while $n \rightarrow +\infty$ is not realistic. 

Nevertheless, it is useful to write down what Theorem~\ref{thm.bounded} 
becomes in the asymptotic setting of Example~\ref{ex.asymptotic-setting}. 
As previously noticed, $\dstar$, $\lambdainf_{\taustar}$, $\deltainf^2$ 
and $M^2$ then converge to positive constants as $n \rightarrow + \infty$. 
Therefore, $\cmin$ is of order $\log(n)$, 
$\cmax$ is of order $n$ and we always have $\cmin < \cmax$ for $n$ large enough. 
The upper bound on $C$ matches classical asymptotic conditions 
for variable selection \citep{Sha:1997}. 
The necessity of taking $C$  of order at least $\log(n)$ is 
shown by \citet{Bir_Mas:2006} in a variable selection setting, 
which includes change-point detection as a particular example; 
\citet{Bir_Mas:2006,Abr_etal:2006} provide several arguments 
for the optimality of taking a constant $C$ of order $\log(n)$. 
When~$C$ satisfies \eqref*{eq:hyp-main-result-synth}, 
the result of Theorem~\ref{thm.bounded} implies that 
$\proba{\dhat=\dstar}\to 1$. 
For the linear kernel in $\R^d$, this is a well-known result 
when the distribution of the $X_i$ changes only through its mean. 
The first result dates back to \citet[Section~2]{Yao:1988} for a Gaussian noise, 
later extended by \citet{Liu_Wu_Zid:1997} and \citet[Section~3.1]{Bai_Per:1998} 
under mixingale hypothesis on the error, 
and \citet{Lav_Mou:2000} under very mild assumptions satisfied for 
a large family of zero-mean processes 
\citep[for the precise statement of the hypothesis, see][Section~2.1]{Lav_Mou:2000}. 
Theorem~\ref{thm.bounded} also shows that the normalized estimated change-points of $\tauhat$ 
converge towards the normalized true change-points at speed at least $\log(n)/n$. 

Up to a logarithmic factor, this speed matches the minimax lower bound $n^{-1}$ 
which has been obtained previously for various change-point procedures 
\citep[for instance]{Kor:1988,Boy_etal:2006,Kor_Tsy:2012} 
including least-squares \citep{Lav_Mou:2000}, 
assuming that $\lambdainf_{\taustar} \geq \kappa >0$. 
When $\dstar \geq 3$ and the assumption on $\lambdainf_{\taustar}$ is removed 
---that is, segments of length much smaller than $n$ are allowed, 
which is compatible with Theorem~\ref{thm.bounded} since it is non-asymptotic---, 
\citet[Theorem~6]{Bru:2014} shows a minimax lower bound of order $\log(n)/n$. 
Therefore, in this setting, KCP achieves the minimax rate. 
We do not know whether KCP remains minimax optimal (without the $\log$ factor) 
under the assumption $\lambdainf_{\taustar} \geq \kappa >0$. 

Note finally that KCP also performs well for finite samples, 
according to the simulation experiments of \citet{Arl_Cel_Har:2012}. 

\medbreak

Theorem~\ref{thm.bounded} emphasizes the key role of $\deltainf^2/M^2$, 
which can be seen as a generalization of the signal-to-noise ratio, 
for the change-point detection performance of KCP. 
The larger is this ratio, the easier it is to have \eqref[name=Eq.~]{eq:hyp-main-result-synth} 
satisfied and the smaller is $\vitThmbounded(y) $. 
This suggests to choose $k$ (theoretically at least) by maximizing 
$\deltainf^2/M^2$, as we discuss in Section~\ref{sec:conclu}. 
Note that $\deltainf^2/M^2$ is invariant by a rescaling of $k$, 
hence the result of Theorem~\ref{thm.bounded} is unchanged when $k$ is rescaled.

\medbreak

The hypothesis in \eqref[name=Eq.~]{eq:hyp-main-result-synth} is actually three-fold. 
First, we use that $C > \cmin$ to get $D_{\tauhat} \leq \dstar$. 
We have to assume $C$ large enough since a too small penalty 
leads to selecting (with KCP or any other penalized least-squares procedure) 
the segmentation with~$n$ segments, that is~$\dhat=n$. 
Second, $C < \cmax$ is used to get $D_{\tauhat} \geq \dstar$. 
Such an assumption is required since taking a penalty function too large 
in \eqref[name=Eq.~]{eq:original-problem} would result in selecting 
the segmentation with only one segment, that is, $\dhat=1$.
Third, $ \cmax $ has to be greater than 
$\cmin$ for providing a non-empty interval of possible values for~$C$.
This inequality is also used in the proof of the upper bound on 
$\dinf\bigl(\taustar,\tauhat\bigr) $ when we already know that 
$D_{\tauhat} = \dstar$. 
In Example~\ref{ex.asymptotic-setting}, the $\cmin < \cmax$ hypothesis translates into $\lambdainf_{\taustar} \succ \log(n)/n$.
That is, the size of the smallest segment has to be of order $\log n / n$.
This is known to be a necessary condition to obtain the minimax rate in multiple change-point detection \cite[section~2]{Bru:2014}.

Theorem~\ref{thm.bounded} helps choosing~$C$, 
which is a key parameter of KCP, as in any penalized model selection procedure. 
However, in practice, we do not recommend to directly use 
\eqref{eq:hyp-main-result-synth} for choosing~$C$ for two reasons: 
$\cmin,\cmax$ depend on unknown quantities $\dstar, \lambdainf_{\taustar}, \deltainf$, 
and the exact values of the constants in $\cmin,\cmax$ might be pessimistic 
compared to what we can observe from simulation experiments. 
We rather suggest to use a data-driven method for choosing $C$, 
see Section~\ref{sec:conclu}.

\medbreak

If we know~$\dstar$, we can replace $\tauhat$ by 
\[
\tauhat(\dstar) \in \argmin_{ \tau \in \T_n^{\dstar} } \bigl\{\emprisk_n(\tau) \bigr\}
\, . 
\]
Then, assuming that $\lambdainf_{\taustar} > v_1(y)$ 
--- which is weaker than assuming $\cmin < \cmax$ ---, the proof of 
Theorem~\ref{thm.bounded} shows that, on $\Omega$, we have 
\[
\frac{1}{n}\dinf\bigl(\taustar,\tauhat (\dstar) \bigr)
\leq  \vitThmbounded(y) 
\, . 
\]

\subsection{Loss functions between segmentations}
\label{sec:main:metric} 

Theorem~\ref{thm.bounded} shows that~$\tauhat$ is close to~$\taustar$ 
in terms of $\dinf$. 
Several other loss functions (measures of dissimilarity) 
can be defined between segmentations \citep{Hub_Ara:1985}. 
We here consider a few of them, which are often used or natural 
for the change-point problem. 

\medbreak

Let us first consider losses related to the Hausdorff distance. 
For any~$\tau^1$ and~$\tau^2\in\T_n$, we define
\begin{align*}
\dinf(\tau^1,\tau^2) 
&\defeq \max_{1\leq i\leq D_{\tau^1} - 1}\biggl\{ \min_{1\leq j\leq D_{\tau^2} - 1} 
\abs{\tau_i^1 - \tau_j^2} \biggr\} 
\\
\dinfb(\tau^1,\tau^2) 
&\defeq \max_{1\leq i\leq D_{\tau^1} - 1} \biggl\{ \min_{0\leq j\leq D_{\tau^2}} 
\abs{\tau_i^1 - \tau_j^2} \biggr\} 
\\
\disthaus^{(i)}(\tau^1,\tau^2) 
&\defeq \max\bigl\{\distinf^{(i)}(\tau^1,\tau^2),\distinf^{(i)}(\tau^2,\tau^1)\bigr\} 
\qquad \text{for}\; i\in\bigl\{1,2\bigr\}
\, .
\end{align*}
Whenever $ D_{\tau^1}=D_{\tau^2}$, we define
\[
\dinfD(\tau^1,\tau^2) 
\defeq \max_{1\leq i\leq D_{\tau^1} - 1}\abs{\tau_i^1 - \tau_i^2}
.\]
Note that~$\dinfD$ is symmetric thus there is no need to define~$\disthaus^{(3)}$.
One could also define~$\disthaus^{(1)}$ as the {\em Hausdorff distance\/} between the subsets 
$\{ \tau_1^1,\dots,\tau_{D_{\tau^1}-1}^1 \}$ and 
$\{ \tau_1^2,\dots,\tau_{D_{\tau^2}-1}^2 \}$ 
with respect to the distance $\delta(x,y)=\abs{x-y}$ on $\R$.
These definitions are illustrated by Figure~\ref{fig:metrics-example}.

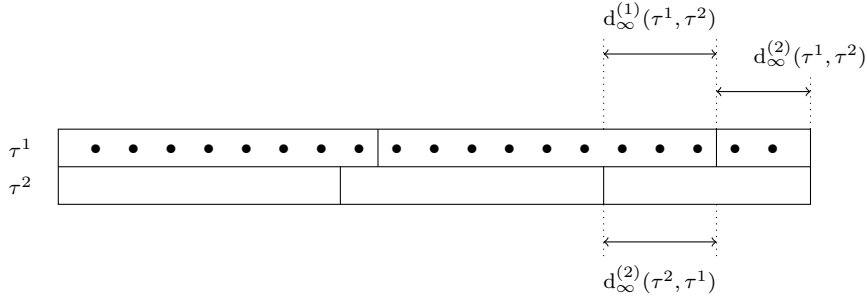
\begin{figure}[ht]
\centering
\begin{tikzpicture}
\def\offset{0}
\draw (0,0+\offset) rectangle (10,1+\offset) ;
\draw (0,0.5+\offset) -- (10,0.5+\offset) ;
\draw (4.25,0.5+\offset) -- (4.25,1+\offset) ;
\draw (8.75,0.5+\offset) -- (8.75,1+\offset) ;
\draw (0.5,0.73+\offset) node {$\bullet$} ;
\draw (1,0.73+\offset) node {$\bullet$} ;
\draw (1.5,0.73+\offset) node {$\bullet$} ;
\draw (2,0.73+\offset) node {$\bullet$} ;
\draw (2.5,0.73+\offset) node {$\bullet$} ;
\draw (3,0.73+\offset) node {$\bullet$} ;
\draw (3.5,0.73+\offset) node {$\bullet$} ;
\draw (4,0.73+\offset) node {$\bullet$} ;
\draw (4.5,0.73+\offset) node {$\bullet$} ;
\draw (5,0.73+\offset) node {$\bullet$} ;
\draw (5.5,0.73+\offset) node {$\bullet$} ;
\draw (6,0.73+\offset) node {$\bullet$} ;
\draw (6.5,0.73+\offset) node {$\bullet$} ;
\draw (7,0.73+\offset) node {$\bullet$} ;
\draw (7.5,0.73+\offset) node {$\bullet$} ;
\draw (8,0.73+\offset) node {$\bullet$} ;
\draw (8.5,0.73+\offset) node {$\bullet$} ;
\draw (9,0.73+\offset) node {$\bullet$} ;
\draw (9.5,0.73+\offset) node {$\bullet$} ;
\draw (3.75,0+\offset) -- (3.75,0.5+\offset) ;
\draw (7.25,0+\offset) -- (7.25,0.5+\offset) ;
\draw (-0.5,0.75+\offset) node {$\tau^1$} ;
\draw (-0.5,0.25+\offset) node {$\tau^2$} ;
\draw[<->] (7.25,-0.5) -- (8.75,-0.5) ;
\draw (8,-1) node {$\dinfb(\tau^2,\tau^1)$} ;
\draw[dotted] (7.25,0) -- (7.25,-0.75) ;
\draw[dotted] (8.75,0) -- (8.75,-0.75) ;
\draw[dotted] (7.25,1) -- (7.25,2.25) ;
\draw[dotted] (8.75,1) -- (8.75,2.25) ;
\draw[<->] (7.25,2) -- (8.75,2) ;
\draw (8,2.5) node {$\dinf(\tau^1,\tau^2)$} ;
\draw[dotted] (10,1) -- (10,1.75) ;
\draw[<->] (8.75,1.5) -- (10,1.5) ;
\draw (10,2) node {$\dinfb(\tau^1,\tau^2)$} ;
\end{tikzpicture}
\caption{Illustration of the definition of 
$\distinf^{(i)}$, with $n=19$, $\tau^1=\bigl[0,8,17,19\bigr]$ 
and $\tau^2=\bigl[0,7,14,19\bigr]$. 
In this example, $D_{\tau^1}=D_{\tau^2}=3$.
We can compute  
$\dinf(\tau^1,\tau^2)= \dinf(\tau^2,\tau^1)
= \dinfb(\tau^2,\tau^1) =  \dinfD(\tau^1,\tau^2)=3$ 
and $\dinfb(\tau^1,\tau^2)=2$.}
\label{fig:metrics-example}
\end{figure}

Interestingly, all these loss functions coincide  whenever $n^{-1} \dinf(\tau^1,\tau^2)$ is small enough.
The following lemma makes this claim rigorous.

\begin{lemma}
\label{lemma:equality-distances}
We have the following two properties. 
\begin{itemize}
\item[(i)]
For any $\tau^1, \tau^2 \in \T_n$ such that 
\[
\frac{1}{n}\dinf(\tau^1,\tau^2) 
< \frac{1}{2} \min\bigl\{\lambdainf_{\tau^1},\lambdainf_{\tau^2}\bigr\} 
\, , \]
we have 
$D_{\tau^1}=D_{\tau^2}$ and
\[
\dinf(\tau^1,\tau^2)
=\dinfb(\tau^1,\tau^2)
=\dinfD(\tau^1,\tau^2)
=\dinfH(\tau^1,\tau^2)
=\dinfbH(\tau^1,\tau^2)
.\]

\item[(ii)]
For any $\tau^1, \tau^2 \in \T_n$ such that 
\[ 
D_{\tau^1}=D_{\tau^2} 
\qquad \text{and} \qquad 
\frac{1}{n}\dinf(\tau^1,\tau^2) < \frac{\lambdainf_{\tau^1}}{2} 
\, , 
\]
we have 
\[
\dinf(\tau^1,\tau^2)
=\dinf(\tau^2,\tau^1)
=\dinfH(\tau^1,\tau^2)
\, .
\]
\end{itemize}
\end{lemma}
Lemma~\ref{lemma:equality-distances} is proved 
in Section~\ref{sec:proof-equality-distances}. 
As a direct application of Lemma~\ref{lemma:equality-distances} we see that 
the statement of Theorem~\ref{thm.bounded} holds true 
with $\dinf$ replaced by {\em any\/} of 
the loss functions that we defined above, at least for $n$ large enough. 

\medbreak

Another loss between segmentations is the {\em Frobenius\/} loss \citep{Laj_Arl_Bac:2014}, 
which is defined as follows. 
For any $\tau^1, \tau^2 \in \T_n$, 
\[
\distfrob(\tau^1,\tau^2)
\defeq \frobnorm{\Pi_{\tau^1}-\Pi_{\tau^2}}
,\]
where $\Pi_{\tau}$ is the orthogonal projection onto~$F_{\tau}$, 
as defined in Section~\ref{sec:rewriting-emp-risk}, 
and $\frobnorm{\cdot}$ denotes the Frobenius norm of a matrix: 
\[
\forall A\in\R^{N\times M},\quad 
\frobnorm{A}^2
\defeq \sum_{i=1}^N\sum_{j=1}^M A_{ij}^2
\,. 
\]
A closed-form formula for $\distfrob$ can be derived from 
the matrix representation of~$\Pi_{\tau}$ 
that is given by \eqref*{eq:computation-projection}: 
for any $i,j \in \{1 , \ldots, n \}$, 
\[
(\Pi_{\tau})_{i,j} 
= 
\begin{cases}
\frac{1}{\abs{\lambda}} \qquad & 
\text{if $i$ and $j$ belong to the same segment $\lambda$ of $\tau$} 
\\
0 \qquad &\text{otherwise.}
\end{cases}
\]

An interesting feature of the Frobenius loss is that it is smaller than one 
only when $\tau^1$ and $\tau^2$ have the same number of segments, 
whereas Hausdorff distances can be small with very diffferent numbers of segments. 
Indeed, we prove in Section~\ref{sec:proof-prop-equivalence} that 
\begin{equation} 
\label{eq.distfrob.ineq}
\lvert D_{\tau^1} - D_{\tau^2} \rvert 
\leq \distfrob(\tau^1,\tau^2)^2 
\leq D_{\tau^1} + D_{\tau^2} 
\, . 
\end{equation}

The next proposition shows that there is an equivalence 
(up to constants) between 
the Hausdorff and Frobenius losses between segmentations, 
provided that they are close enough.
\begin{proposition}
\label{prop:equivalence}
Suppose that $D_{\tau^1} = D_{\tau^2}$ and 
$\frac{1}{n}\dinf(\tau^1,\tau^2) <\lambdainf_{\tau^1}/2$, then
\[ 
\left(\distfrob(\tau^1,\tau^2)\right)^2 
\leq \frac{12 D_{\tau^1} }{\lambdainf_{\tau^1}} \frac{1}{n}\dinf(\tau^1,\tau^2)
\, . 
\]
If in addition $\frac{1}{n}\dinf(\tau^1,\tau^2) <\lambdainf_{\tau^1}/3$, then 
\[
\frac{2 }{3 \lambdasup_{\tau^1}} \frac{1}{n}\dinf(\tau^1,\tau^2)   
\leq \left(\distfrob(\tau^1,\tau^2)\right)^2
\, . 
\]
\end{proposition}
Proposition~\ref{prop:equivalence} 
was first stated and proved by \cite[Theorem~B.2]{Laj_Arl_Bac:2014}.
We prove it in Section~\ref{sec:proof-prop-equivalence} for completeness. 

\medbreak

As a corollary of Theorem~\ref{thm.bounded} 
and Proposition~\ref{prop:equivalence}, we get the following 
guarantee on the Frobenius loss between $\taustar$ and 
the segmentation $\tauhat$ estimated by KCP. 

\begin{corollary}
\label{cor:Frobenius.bounded}
Under the assumptions of Theorem~\ref{thm.bounded}, 
on the event $\Omega$ defined by Theorem~\ref{thm.bounded}, 
for any $\tauhat$ satisfying \eqref*{eq:original-problem} 
with~$\pen$ defined by \eqref*{eq:def-penalty-alt}, 
we have: 
\[
\distfrob(\taustar,\tauhat) 
\leq 
\frac{43\dstar}{\sqrt{\lambdainf_{\taustar}}} \cdot \frac{M}{\deltainf} \sqrt{\frac{y+\log n + 1}{n}}
\, . 
\]
\end{corollary}
Note that Corollary~\ref{cor:Frobenius.bounded} gives a better result 
(at least for large~$n$) than the obvious bound
\[
\distfrob(\taustar,\tauhat) \leq \dstar +\dhat - 2
\, .
\]
\begin{proof}
On the event~$\Omega$, we have $\frac{1}{n}\dinf(\taustar,\tauhat) <\lambdainf_{\taustar}/(\dstar + 1)$ and $\dstar=\dhat$. 
Therefore, according to Proposition~\ref{prop:equivalence},
\[
\bigl( \distfrob(\taustar,\tauhat) \bigr)^2 
\leq \frac{12\dstar}{\lambdainf_{\taustar}} \frac{1}{n}\dinf(\taustar,\tauhat) 
\leq \frac{1776\dstar^2 (y+\log n + 1)}{n\lambdainf_{\taustar}}\cdot \frac{M^2}{\deltainf^2}
\, .
\]
\end{proof}

\medbreak

Up to this point, we assessed the quality of the segmentation~$\tau$ by considering the proximity of~$\tau$ with~$\taustar$.
Another natural idea is to measure the distance between~$\mustar$ and~$\mustar_{\tau}$ in~$\hilbert^n$.
It is closely related to the oracle inequality proved by \citet{Arl_Cel_Har:2012}, 
which implies an upper bound on $\norm{\mustar - \muhat_{\tauhat}}^2$. 
We can also observe that there is a simple relationship between $\norm{\mustar-\mustar_{\tau}}^2$ and the Frobenius distance between~$\tau$ and~$\taustar$.
Indeed,
\begin{equation}
\label{eq.ineq.approx-dfrob} 
\norm{\mustar - \mustar_{\tau}}^2 
= \norm{ (\Pi_{\taustar} - \Pi_{\tau} ) \mustar }^2 
\leq \lVert \Pi_{\taustar} - \Pi_{\tau} \rVert_2^2 \norm{ \mustar }^2
\leq 
\bigl( \distfrob(\taustar,\tauhat) \bigr)^2 \norm{ \mustar }^2 
\, . 
\end{equation}
\Eqref{eq.thm.bounded.maj-approx} in the proof of Theorem~\ref{thm.bounded} 
shows that on $\Omega$, under the assumptions of Theorem~\ref{thm.bounded}, 
\[
\norm{ \mustar - \mustar_{\tauhat} }^2 \leq 74 \bigl( y + \log(n) + 1 \bigr) \dstar M^2 
\]
which is slightly better (but similar) to what Corollary~\ref{cor:Frobenius.bounded}, 
\eqref{eq.ineq.approx-dfrob} 
and the bound $\norm{\mustar}^2 \leq M^2 n$ imply. 

\subsection{Extension to the finite variance case}
\label{sec:main:extension}

Theorem~\ref{thm.bounded} is valid under a boundedness assumption 
(Assumption~\ref{assump:bounded-kernel}). 
What happens under the weaker Assumption~\ref{assump:bounded-variance}? 
As a first step, we provide a result for 
\begin{equation}
\label{eq:alternative-problem}
\tauhat(\dstar, \delta_n)
\in \argmin_{\tau\in\T_n^{\dstar} \,/\, \lambdainf_{\tau} \geq \delta_n } 
\bigl\{\emprisk_n(\tau)\bigr\}
\end{equation}
for some $\delta_n >0$. 
In other words, we restrict our search to segmentations $\tau$ 
of the correct size --- hence $\dstar$ must be known \textit{a priori} --- and 
having no segment with less than $n \delta_n$ observations. 
We discuss how to relax this restriction right after the statement of 
Theorem~\ref{th:localization-moment}. 
Note that the dynamic programming algorithm of \citet{Har_Cap:2007} 
can be used for computing $\tauhat(\dstar, \delta_n)$ efficiently.

Similarly to $\deltainf$, we define 
$\deltasup\defeq \max_i\hilbertnorm{\mustar_i - \mustar_{i+1}}$.

\begin{theorem}
\label{th:localization-moment}
Suppose that Assumption~\ref{assump:bounded-variance} holds true.
For any $\delta_n, y > 0$, define\textup{:}  
\[ 
\vitThmmoment(y,\delta_n) 
:= 
24 (\dstar)^2  \frac{ \deltasup \sqrt{V} }{ \deltainf^2}  \frac{ y }{ \sqrt{n} }
+ 8 \dstar \frac{ V }{ \deltainf^2 } \frac{ y^2 }{ n \delta_n }
\, .  
\]
For any $y>0$, an event $\Omega_2$ exists such that 
\[ 
\proba{ \Omega_2 } \geq 1 - \frac{1}{y^2} 
\]
and, on $\Omega_2$, we have the following\textup{:} 
for any $\delta_n \in (0 , \lambdainf_{\taustar}]$ 
and any $\tauhat(\dstar,\delta_n)$ satisfying \eqref[name=Eq.~]{eq:alternative-problem}, 
if $\vitThmmoment(y,\delta_n)  \leq \lambdainf_{\taustar}$, 
\begin{equation}
\label{eq:localization-moment}
\frac{1}{n}\dinf \bigl( \taustar,\tauhat(\dstar,\delta_n) \bigr) 
\leq \vitThmmoment(y,\delta_n)  
\, .
\end{equation}
\end{theorem}
We postpone the proof of Theorem~\ref{th:localization-moment} to Section~\ref{sec:proof-th-localization-moment}. 
Let us make a few remarks.

As for Theorem~\ref{thm.bounded}, our result is non-asymptotic. 
However, it is interesting to write it down in 
the setting of Example~\ref{ex.asymptotic-setting}. 
If~$n$ goes to infinity, then the assumption $\lambdainf_{\taustar}\geq\delta_n$ 
is satisfied whenever $\delta_n\to 0$. 
If we furthermore require that $n\delta_n\to\infty$, 
then \eqref[name=Eq.~]{eq:localization-moment} implies that 
\[
\frac{1}{n}\dinf \bigl( \taustar,\tauhat(\dstar,\delta_n) \bigr)
\xrightarrow[n \rightarrow +\infty]{\mathbb{P}} 0
\, ,
\]
by taking a well-chosen $y$ of order $\sqrt{n}+\sqrt{n\delta_n}$. 
In the particular case of the linear kernel, this result is known under various hypothesis 
\cite[for instance]{Lav_Mou:2000}; it is new for a general kernel.

More precisely, if we take $\delta_n = n^{-1/2}$, 
Theorem~\ref{th:localization-moment} implies that 
\[
\frac{1}{n}\dinf\left (\taustar,\tauhat(\dstar,n^{-1/2})\right )
\]
goes to zero at least as fast as $\ell_n / \sqrt{n}$, 
where $(\ell_n)_{n \geq 1}$ is any sequence tending to infinity, 
for instance $\ell_n = \log(n)$. 
This speed seems suboptimal compared to previous results 
\cite[for instance]{Lav_Mou:2000} 
--- which do not consider the case of a general kernel ---, 
but we have not been able to prove tight enough deviation bounds 
for getting the localization rate $\log(n)/n$ under Assumption~\ref{assump:bounded-variance}. 

How does Theorem~\ref{th:localization-moment} compares to 
Theorem~\ref{thm.bounded}? 
First, as noticed by Remark~\ref{rk.thm.bounded.deltan} 
in Section~\ref{sec:proof-main-result}, the result of 
Theorem~\ref{thm.bounded} also holds true for 
$\tauhat(\dstar,\delta_n)$ as long as $\lambdainf_{\taustar} \geq \delta_n$. 
Second, $\vitThmbounded(y) $ is usually smaller than $\vitThmmoment(y,\delta_n)$ 
--- its order of magnitude is smaller when $n \rightarrow +\infty$ ---, 
and the lower bound on the probability of $\Omega$ is better than the one for $\Omega_2$. 
There is no surprise here: the stronger Assumption~\ref{assump:bounded-kernel} 
helps us proving a stronger result for $\tauhat(\dstar,\delta_n)$. 
Nevertheless, these only are upper bounds, so we do not know whether 
the performance of $\tauhat(\dstar,\delta_n)$ 
actually changes much depending on the noise assumption. 
For instance, as already noticed, we do not believe that the localization speed 
$\log(n)/n$ requires a boundedness assumption; 
in particular cases at least, it has been obtained for unbounded data 
\citep{Lav_Mou:2000,Boy_etal:2006}. 

The dependency in $k$ of the speed of convergence of $\tauhat(\dstar,\delta_n)$ 
is slightly less clear than in Theorem~\ref{thm.bounded}. 
The signal-to-noise ratio appears through $\deltainf^2/V$, 
as expected, but the size $\deltasup$ of the largest true jump also 
appears in $\vitThmmoment$. 
At the very least, it is clear that~$\deltainf^2/V$ should not be too small. 

As noted by \citet{Lav_Mou:2000}, 
it may be possible to get rid of the minimal segment length~$\delta_n$, 
either by imposing stronger conditions on~$\varepsilon$ --- which 
are not met in our setting --- or by constraining the values of~$\muhat$ 
to lie in a compact subset $\Theta\subset\hilbert^{\dstar+1}$.

\section{Numerical simulations}
\label{sec:simulation}

One consequence of our main result, Theorem~\ref{thm.bounded}, is that for a bounded kernel, 
the KCP procedure is consistent in the asymptotic setting presented in Example~\ref{ex.asymptotic-setting}.
We now illustrate this fact by a simulation study.

\begin{figure}[ht!]
\makebox[\textwidth][c]{\includegraphics[width=1.2\textwidth]{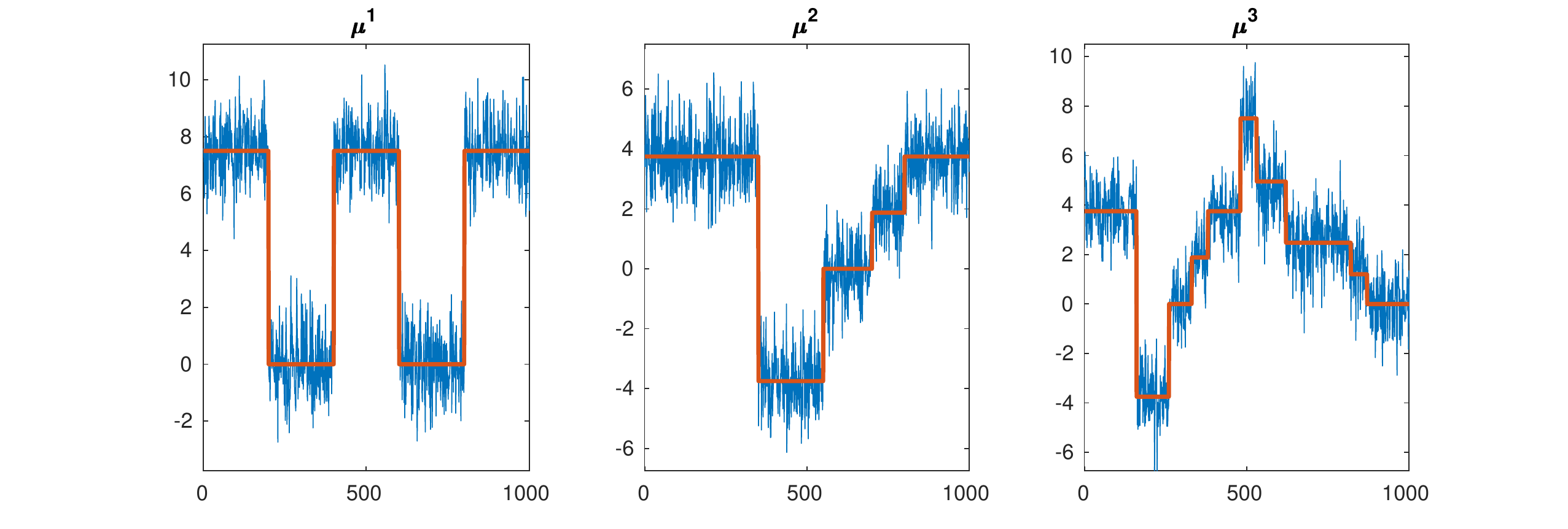}}
\caption{In red, the three piecewise constant functions used in the simulations of Section~\ref{sec:gaussian-kernel}. 
In blue, a noisy version of these functions. 
Both $\mu^1$ and $\mu^2$ have $4$ jumps; 
$\mu^3$ has $9$ jumps.}
\label{fig:reg-functions}
\end{figure}

\paragraph{Detecting changes in the mean with the Gaussian kernel}
\label{sec:gaussian-kernel}
Let us consider the archetypic change-point detection problem 
---finding changes in the mean of a sequence of independent random variables--- 
and show how these changes are localized more precisely when more data are available. 

We define three functions $\mu^m:[0,1]\to \R$, $1\leq m\leq 3$, previously used by \citet{Arl_Cel:2011}, 
which cover a variety of situations (see Fig.~\ref{fig:reg-functions}). 
For each $m \in \{1,2,3\}$ and several values of $n$ between $10^2$ and $10^3$, 
we repeat $10^3$ times the following:
\begin{itemize}
\item
Sample $n$ independent Gaussian random variables $g_i\sim\gaussian(0,1)$;

\item
Set $X_i = \mu^m(i/n) + g_i$ ---Fig.~\ref{fig:reg-functions} shows one sample for each $m \in \{1,2,3\}$; 

\item
Perform KCP with Gaussian kernel and linear penalty on $X_1, \ldots, X_n$;   
the penalty constant is chosen as indicated in Section~\ref{sec:conclu}, the bandwidth is set to $0.1$, and the maximum number of change-points is set to $30$;

\item
Compute $d_H^{(2)}(\taustar, \tauhat_n)$.
\end{itemize}

The results are collected in Fig.~\ref{fig:convergence-results}, 
where each graph corresponds to a regression function $\mu^m$. 
We represent in logarithmic scale the mean distance between the true segmentation and the estimated segmentation for each value of~$n$. 
The error bars are $\pm \widehat{\sigma} / \sqrt{N}$, where $\widehat{\sigma}$ is the empirical standard deviation over $N=10^3$ repetitions.
We want to emphasize that, though these experiments illustrate our main result Theorem~3.1, 
they are carried out in a slightly different setting 
since the penalty constant~$C$ is not chosen according to \eqref{eq:hyp-main-result-synth}, 
but using the dimension jump heuristic \citep{Bau_Mau_Mic:2012}.

\begin{figure}[ht!]
\makebox[\textwidth][c]{\includegraphics[width=1.15\textwidth]{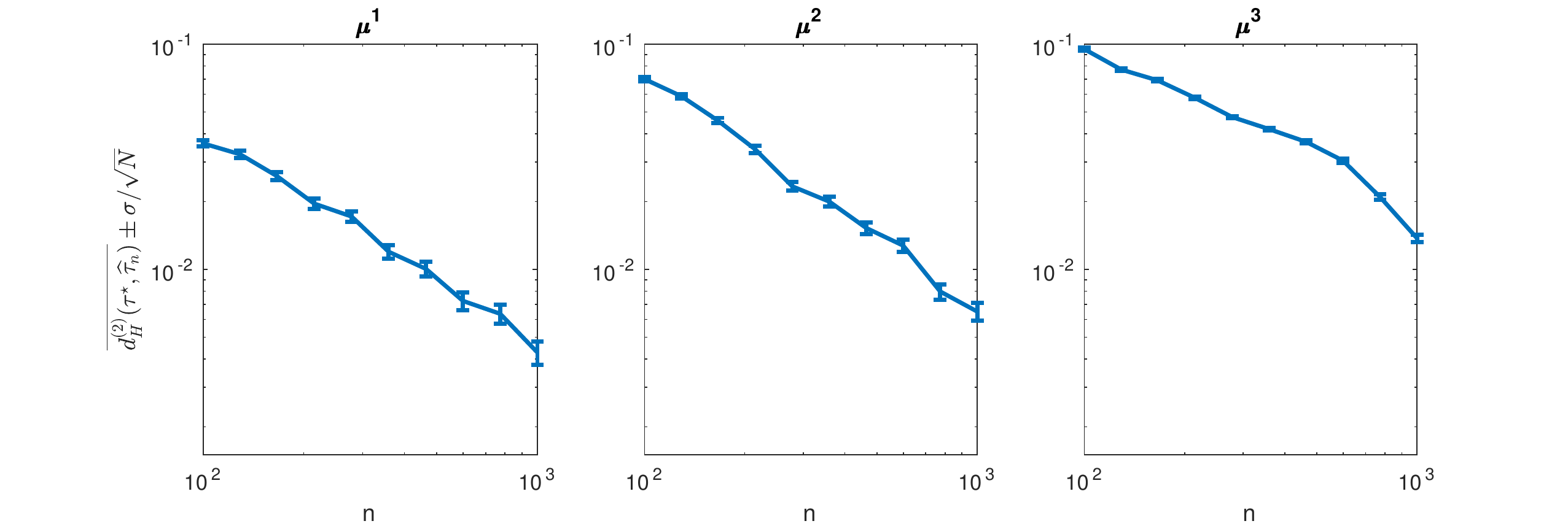}}
\caption{Convergence of $\frac{1}{n}\disthaus^{(2)}(\taustar,\tauhat_n)$ towards $0$ when the number of data points~$n$ is increasing.  
A linear regression between $\log n$ and $\frac{1}{n}\disthaus^{(2)}(\taustar,\tauhat_n)$ 
for $n \geq 300$ 
yields slope estimates $-0.97$, $-1.04$ and $-1.00$, respectively. }
\label{fig:convergence-results}
\end{figure}

The three segmentation problems considered here are quite different in nature, 
but all lead to a linear convergence rate 
(slopes close to~$-1$ on the graphs of Figure~\ref{fig:convergence-results}) 
with different constants 
(different values for the intercept on the graphs of Figure~\ref{fig:convergence-results}). 
Recall that Theorem~\ref{thm.bounded} 
combined with Lemma~\ref{lemma:equality-distances} states that, with high probability, 
\[
\frac{1}{n}\disthaus^{(2)}(\taustar,\tauhat_n) \lesssim \widetilde{v}_1 = \frac{\dstar M^2}{\deltainf^2}\cdot \frac{\log n}{n}
\, .
\]
Hence, whenever $\dstar$, $\deltainf$ and $M$ are fixed, $\frac{1}{n}\disthaus^{(2)}(\taustar,\tauhat_n)$ converges to $0$ at rate at least $\log n / n$ when the number of data points increases.
In our experimental setting, these quantities are fixed, and the observed convergence rate matches our theoretical upper bound. 
The performance of KCP still depends on the regression function $\mu^m$ experimentally, 
by a constant multiplicative factor, 
like the theoretical bound $\widetilde{v}_1$. 

\paragraph{Detecting changes in the number of modes}
\label{sec:other-changes}
Let us now consider data $X_1,\ldots,X_n\in\R$ 
whose distribution vary only through the number of modes.
Can we accurately detect such changes with the KCP procedure? 
The data are generated according to the following process for several $n$:
\begin{itemize}
\item
Set $\taustar_1=\floor{n/3}$ and $\taustar_2=\floor{2n/3}$;

\item 
Draw $X_1,\ldots,X_{\taustar_1},X_{\taustar_2+1},\ldots,X_n$ according to a standard Gaussian distribution, and $X_{\taustar_1+1},\ldots,X_{\taustar_2}$ according to a $(1/2,1/2)$-mixture of Gaussian distributions $\gaussian(\delta, 1-\delta^2)$ and $\gaussian(-\delta, 1-\delta^2)$, 
with $\delta = 0.999$; 
the $X_i$ are independent. 
\end{itemize}
%
%
We test KCP with various kernels assuming that the number of change-points ($\dstar = 3$) is known; 
this simplification avoids possible artifacts linked to the choice of the penalty constant. 
Results are shown on Figure~\ref{fig:kernel-choice}. 
The $X_i$ all have zero mean and unit variance, 
hence a classical penalized least-squares procedure ---KCP with the linear kernel--- 
is expected to detect poorly the changes in the distribution of the $X_i$, 
as confirmed by Figure~\ref{fig:kernel-choice} 
(for instance, according to the right panel, it is not consistent). 
On the contrary, a Gaussian kernel with well-chosen bandwidth yields much better performance 
according to the middle and right panels of Figure~\ref{fig:kernel-choice} 
(with a rate of order $1/n$). 
\begin{figure}[ht!]
\begin{center}
\hspace*{-0.5cm}
\begin{minipage}[c]{0.35\textwidth}
\includegraphics[width=\linewidth]{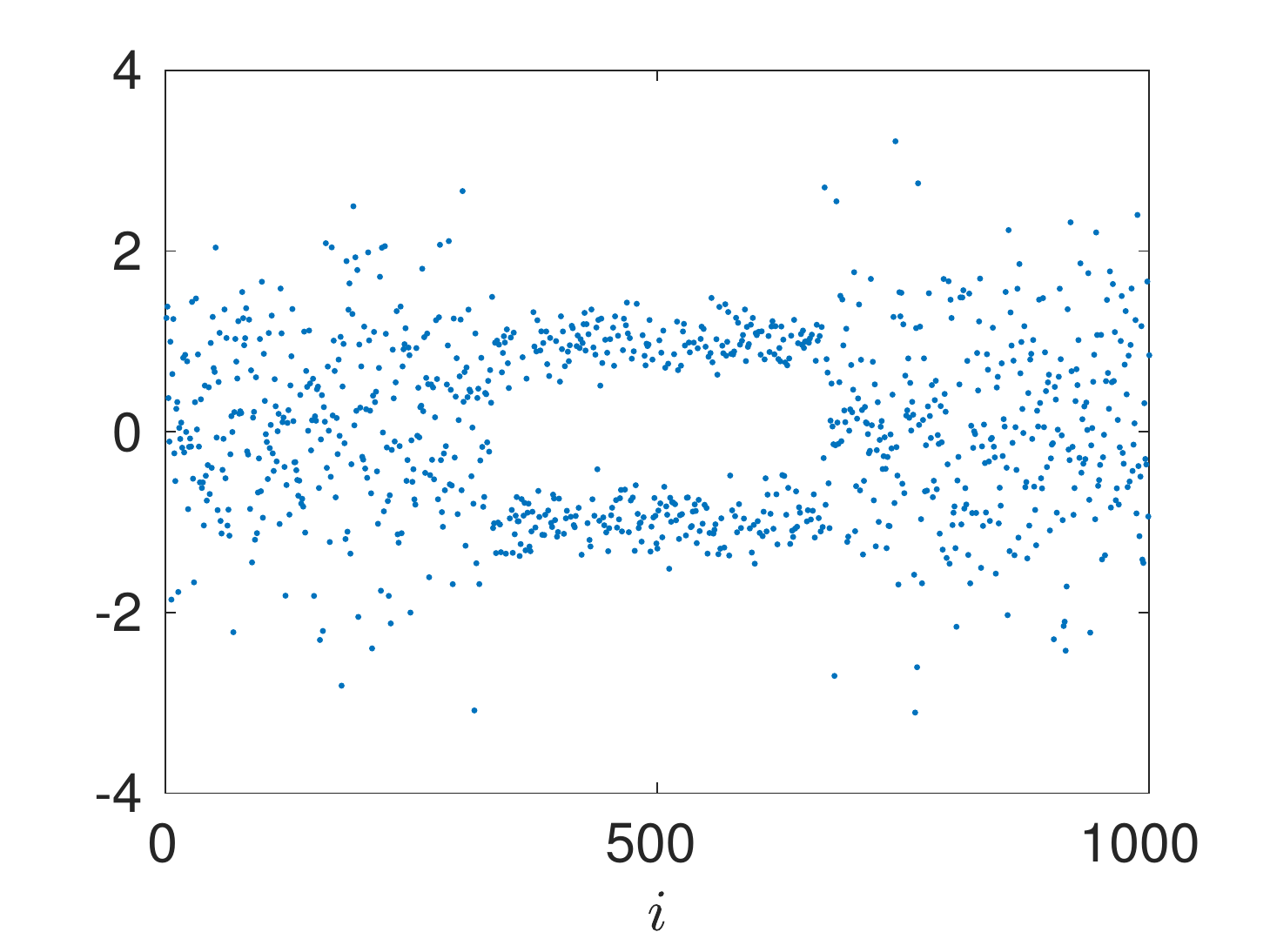}
\end{minipage} \hspace*{-0.25cm}
\begin{minipage}[c]{0.35\textwidth}
\includegraphics[width=\linewidth]{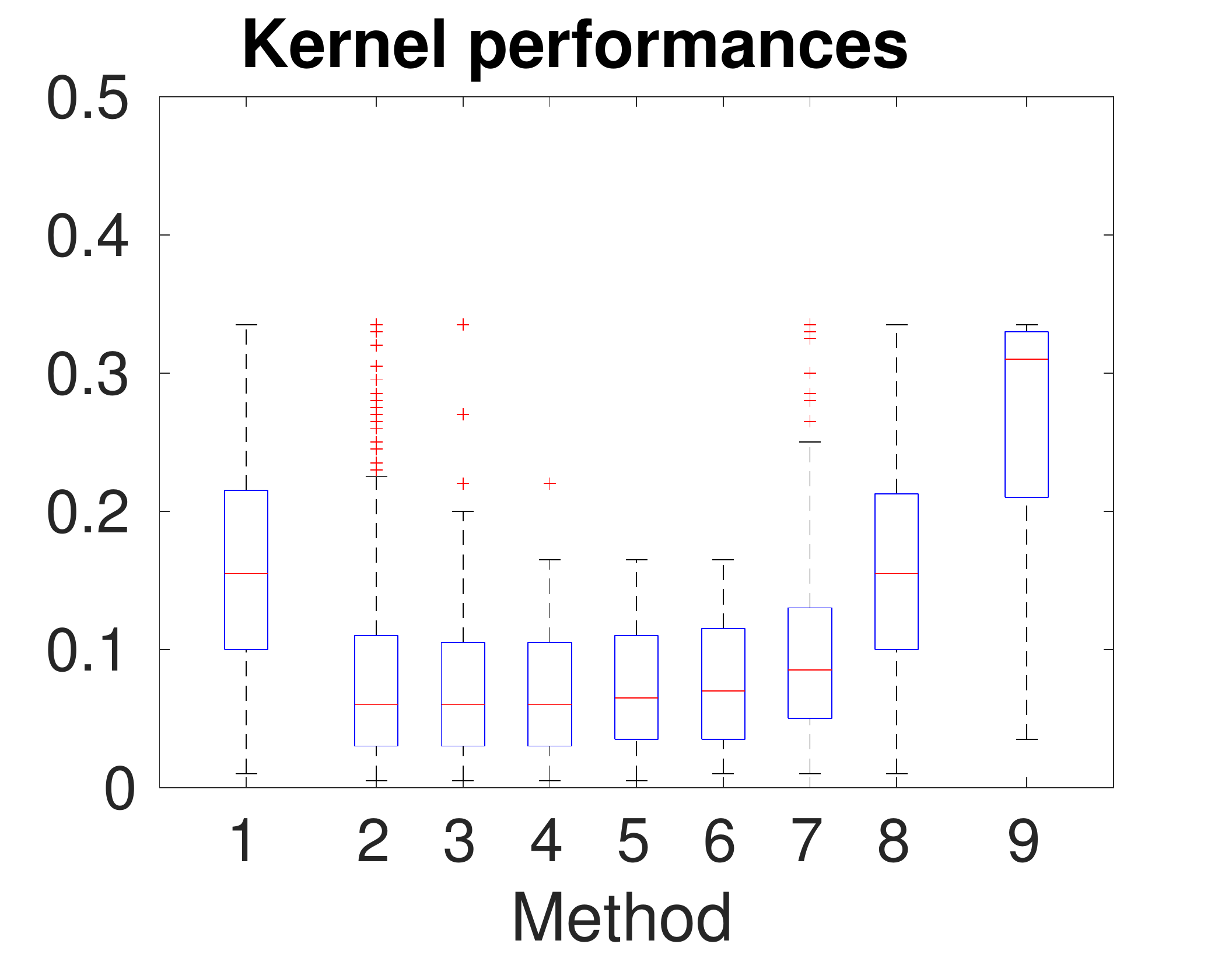}
\end{minipage} \hspace*{-0.4cm}
\begin{minipage}[c]{0.35\textwidth}
\includegraphics[width=\linewidth]{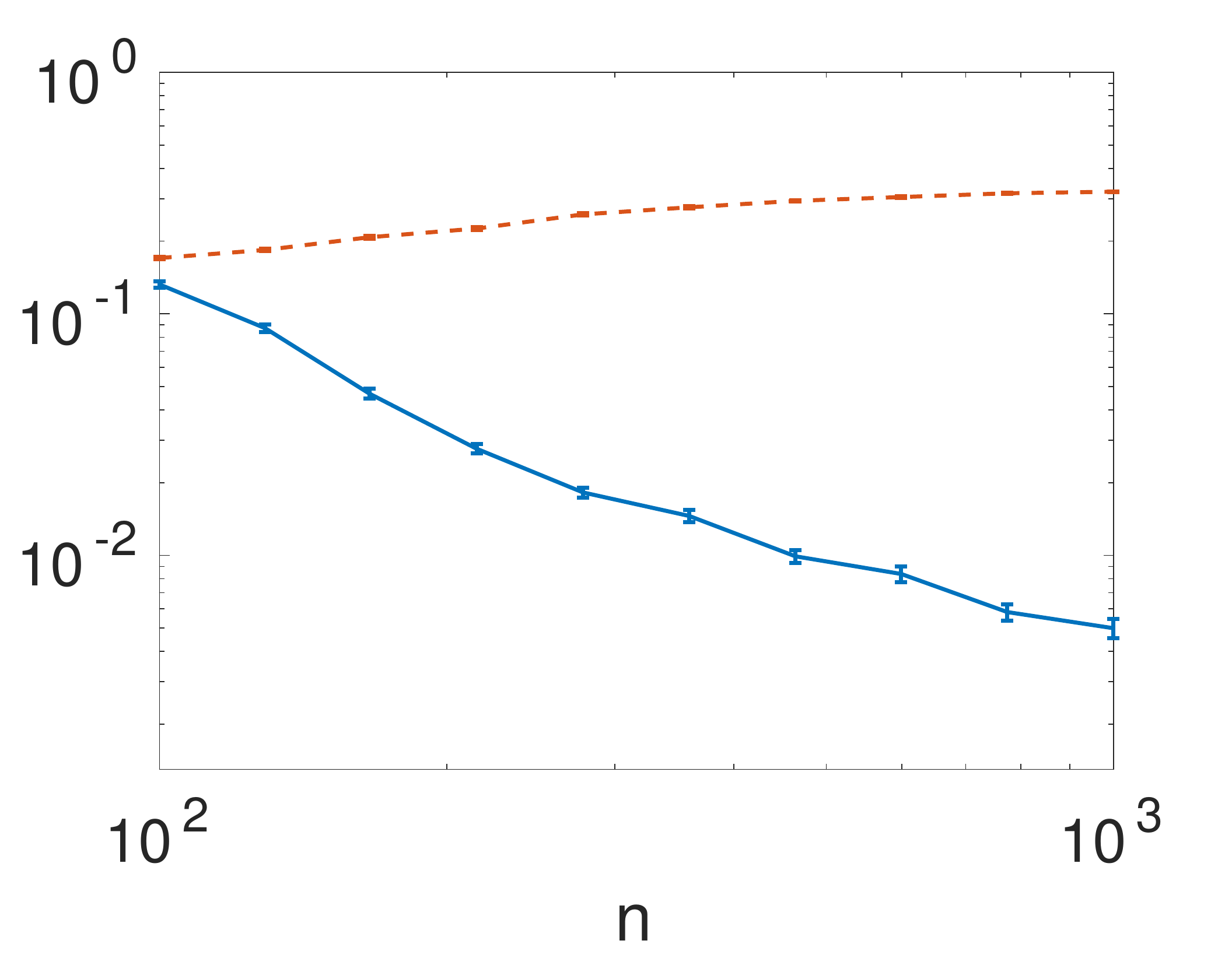}
\end{minipage}
\caption{\textbf{Left\textup{:}} One sample $X_1, \ldots, X_n$ for $n=10^3$. 
\textbf{Middle\textup{:}} Performance of KCP with various kernels \textup{(}$n=200$\textup{)}. 
Methods $1$ to $8$\textup{:} Gaussian kernel with bandwidth set \emph{via} the median heuristic \textup{(}method 1\textup{)}, 
or fixed equal to $0.001,0.005,0.01,0.05,0.1,0.5,1$ \textup{(}methods $2, \ldots, 9$, respectively\textup{)}. 
Method~9\textup{:} linear kernel. 
\textbf{Right\textup{:}} Estimated values of 
$n^{-1} \disthaus^{(2)}(\taustar, \tauhat_n)$ vs. $n$ in log scale, 
for KCP with a Gaussian kernel with bandwidth $0.01$ \textup{(}blue solid line\textup{;} estimated slope $-1.05$\textup{)} 
and with the linear kernel \textup{(}red dashed line\textup{;} estimated slope $0.16$\textup{)}.}
\label{fig:kernel-choice}
\end{center}
\vskip -0.2in
\end{figure}

\section{Discussion}
\label{sec:conclu}
Before proving our main results, let us discuss some of their consequences 
regarding the KCP procedure. 

\paragraph{Fully non-parametric consistent change-point detection}

We have proved that for any kernel satisfying 
some reasonably mild hypotheses, the KCP procedure outputs a segmentation closeby 
the true segmentation with high probability.

An important particular example is the ``asymptotic setting'' 
of Example~\ref{ex.asymptotic-setting}, where we have a fixed 
true segmentation $\taustar$ and fixed distributions $P_1, \ldots, P_{K+1}$ 
from which more and more points are sampled. 
How fast can KCP recover $\taustar$, without any prior information 
on the number of segments $\dstar$ or on the distributions $P_1, \ldots, P_{K+1}$? 

Let us take a bounded characteristic kernel 
--- for instance the Gaussian or the Laplace kernel 
if $\X = \R^d$ ---, 
so that Assumption~\ref{assump:bounded-kernel} holds true. 
Then, Theorem~\ref{thm.bounded} shows that KCP detects consistently 
all changes in the distribution of the $X_i$, 
and localizes them at speed $\log(n)/n$. 
This speed also depends on the adequation between 
the kernel $k$ and the differences between the $P_j$,  
through the ratio $\deltainf^2/M^2$.  
Obtaining such a fully non-parametric result for multiple change-points 
with a general set $\X$ ---we only need to know 
a bounded characteristic kernel on $\X$--- 
has never been obtained before. 
To the best of our knowledge, 
non-parametric consistency results 
for the detection of arbitrary changes in the distribution of the data 
have only been obtained for real-valued data 
\citep{Zou_Yin_Fen:2014} 
or for the case of a single change-point 
\citep{Car:1988,Bro_Dar:2013}. 

\medbreak

\paragraph{Choice of $k$}

An important question remains: how to choose the kernel $k$? 
In Theorem~\ref{thm.bounded}, $k$ only appears through 
the ``signal-to-noise ratio'' $\deltainf^2/M^2$, 
leading to better theoretical guarantees when this signal-to-noise ratio is larger: 
a larger value for $\cmax$ and a smaller bound $\vitThmbounded$ 
on $\dinf\bigl(\taustar,\tauhat\bigr)$. 
Therefore, a simple strategy for choosing the kernel is 
to pick~$k$ that maximizes $\deltainf^2/M^2$, 
at least among a family of kernels, for instance Gaussian kernels. 
This first idea requires to know the distributions of the $X_i$, 
or at least to have prior information on them. 
Interestingly, when the change-points locations are known, 
$\deltainf^2$ corresponds to the maximum mean discrepancy \citep[MMD,][]{Gre_Bor_Ras:2006} 
between the distributions of the $X_i$ over contiguous segments. 
In this particular setting, 
it is feasible to estimate and to maximize~$\deltainf^2$ with respect to the kernel~$k$, 
as done by \citet{Gre_Sej_Str:2012}.
An interesting future development would be to build an estimator of~$\deltainf^2$ 
without knowing the change-point locations 
and to maximize this estimator with respect to the kernel~$k$. 
We refer to \citet[section~7.2]{Arl_Cel_Har:2012} for a complementary discussion 
about the choice of $k$ for KCP.

\paragraph{Choice of $C$}

Another important parameter of the KCP procedure is the constant $C$ 
that appears in the penalty function. 
As mentioned below Theorem~\ref{thm.bounded}, our theoretical guarantees 
provide some guidelines for choosing $C$, but these are not sufficient 
to choose precisely $C$ in practice. 
We recommend to follow the advice of \cite[section~6.2]{Arl_Cel_Har:2012} on this point, 
which is to choose $C$ from data with the ``slope heuristic'' \citep{Bau_Mau_Mic:2012}.

\medbreak

\paragraph{Modularity of the proofs and possible extensions}
Finally, we would like to emphasize what we believe to be an important contribution 
of this paper. 
The structure of the proofs of Theorems~\ref{thm.bounded} and~\ref{th:localization-moment} 
--- which follow the same strategy --- is modular, 
so that one can easily adapt it to different sets of assumptions. 

Our proof strategy is not fully new, since it is similar to the one of almost all previous 
papers analyzing the consistency of least-squares change-point detection procedures.  
In particular, we adapted some ideas of the proofs of \citet{Lav_Mou:2000} 
to the Hilbert space setting. 
Nevertheless, these papers formulate their main results in asymptotic terms, 
which can be seen as a limitation --- especially when $n$ is small or $\X$ is of 
large dimension. 
Another approach is the one of \citet{Leb:2005,Com_Roz:2004,Arl_Cel_Har:2012} 
where non-asymptotic oracle inequalities 
--- using concentration inequalities and 
following the model selection results of \citet{Bir_Mas:2001} --- 
are provided as theoretical guarantees on some penalized least-squares change-point procedures. 
Up to now, these two approaches seemed difficult to combine. 
The proofs of Theorems~\ref{thm.bounded} and~\ref{th:localization-moment} 
show how they can be reconciled, 
which allows us to mix their strengths. 

Indeed, the assumptions on the distributions of the $X_i$ 
---Assumptions~\ref{assump:bounded-kernel} and~\ref{assump:bounded-variance}--- 
are only used for proving bounds on two quantities 
---a linear term $L_{\tau}$ and a quadratic term $Q_{\tau}$---, uniformly over 
$\tau \in \T_n$. 
Under Assumption~\ref{assump:bounded-kernel}, this is done 
thanks to concentration inequalities 
(Lemmas~\ref{lemma:concentration-linear} 
and~\ref{lemma:concentration-quadratic}) 
which have been proved first by 
\citet{Arl_Cel_Har:2012} in order to get an oracle inequality. 
Under Assumption~\ref{assump:bounded-variance}, this is done 
by generalizing the method of \citet{Lav_Mou:2000} to Hilbert-space valued data, 
through two deterministic bounds 
(Lemmas \ref{lemma:majo-V} and~\ref{lemma:majo-linear-part}) 
and a deviation inequality for 
\[ 
\maxeps\defeq \max_{1\leq k\leq n} \hilbertnorm{\sum_{j=1}^k\varepsilon_j}
\]
(Lemma~\ref{lemma:kolmogorov-moment}). 
The rest of the proofs does not use anything about the distribution 
of $X_1, \ldots, X_n$. 

As a consequence, if one can generalize these bounds 
to another setting, 
a straightforward consequence is that a result similar to 
Theorem~\ref{thm.bounded} or~\ref{th:localization-moment} 
holds true for the KCP procedure in this new setting. 
In particular, this could be used for dealing with 
the case of dependent data $X_1, \ldots, X_n$. 
We could also consider an intermediate assumption between 
Assumption~\ref{assump:bounded-variance} 
and Assumption~\ref{assump:bounded-kernel}, of the form: 
\[
\max_{1 \leq i \leq n} \mathbb{E} \bigl[ k(X_i, X_i)^{\alpha} \bigr] \leq B_{\alpha} < +\infty 
\]
for some $\alpha \in (1,+\infty)$.

\section{Proofs}
\label{sec:proofs}

Let us start by describing our general strategy for proving 
our main results. 
Our goal is to build a large probability event on which 
any 
$\tauhat\in\argmin_{\tau\in\T_n}\crit(\tau)$ 
belongs to some subset~$\E$ of~$\T_n$. 
For proving this, we use the key fact that 
$\crit(\taustar) \geq \crit(\tauhat)$, 
together with a lower bound on $\crit(\tau)$ 
holding simultaneously for all $\tau \in \T_n$---hence for $\tau=\tauhat$. 

In order to get such a lower bound on the empirical penalized criterion, 
we start by decomposing it in Section~\ref{sec:decomposition} 
into terms that are simpler to control individually: 
two random terms --- a linear function of $\varepsilon$ 
and a quadratic function of $\varepsilon$ ---, 
and two deterministic terms --- the approximation error and the penalty.
Then, we control these terms thanks to deterministic bounds 
(Section~\ref{sec:deterministic}) 
and deviation/concentration inequalities (Section~\ref{sec:concentration}). 
Finally, we prove Theorem~\ref{thm.bounded} in Section~\ref{sec:proof-main-result} 
and Theorem~\ref{th:localization-moment} 
in Section~\ref{sec:proof-th-localization-moment}.

\subsection{Decomposition of the empirical risk}\label{sec:decomposition}

The first step in the proofs of Theorems~\ref{thm.bounded} 
and~\ref{th:localization-moment} is 
to decompose the empirical risk~\eqref*{eq:empirical-risk-alt}.

\begin{lemma}
\label{lemma:empirical-decomp}
Let $\tau\in\T_n$ be a segmentation. 
Define $\mustar_{\tau} = \Pi_{\tau} \mustar$. 
Then we can write
\begin{equation}
\label{eq:empirical-decomp}
n\emprisk_n(\tau)=\norm{Y - \muhat_{\tau}}^2 
= \norm{\mustar-\mustar_{\tau}}^2 + 2\inner{\mustar-\mustar_{\tau}}{\varepsilon} 
- \norm{\Pi_{\tau}\varepsilon}^2 + \norm{\varepsilon}^2
\, .
\end{equation}
\end{lemma}

\begin{proof}
First, recall that $\muhat_{\tau} = \Pi_{\tau}Y$ and that $Y=\mustar+ \varepsilon$, hence
\begin{align*}
\norm{Y-\muhat_{\tau}}^2 &= \norm{Y-\Pi_{\tau}Y}^2 \\
&= \norm{\mustar+\varepsilon - \Pi_{\tau}(\mustar + \varepsilon)}^2 \\
&= \norm{\mustar - \Pi_{\tau}\mustar}^2 + \norm{\varepsilon-\Pi_{\tau}\varepsilon}^2 
+ 2\inner{\mustar-\Pi_{\tau}\mustar}{\varepsilon - \Pi_{\tau}\varepsilon}.
\end{align*}
Since~$\Pi_{\tau}$ is an orthogonal projection,
\begin{align*}
\norm{Y-\muhat_{\tau}}^2 
&= \norm{\mustar-\mustar_{\tau}}^2 + \norm{\varepsilon}^2 
- 2\inner{\varepsilon}{\Pi_{\tau}\varepsilon} 
+ \norm{\Pi_{\tau}\varepsilon}^2 + 2\inner{(\Id-\Pi_{\tau})\mustar}{\varepsilon} 
\\
&= \norm{\mustar-\mustar_{\tau}}^2 + \norm{\varepsilon}^2 
- \norm{\Pi_{\tau}\varepsilon}^2 
+ 2\inner{(\Id-\Pi_{\tau})\mustar}{\varepsilon}
\, .
\end{align*}
\end{proof}

Since each term of \eqref[name=Eq.~]{eq:empirical-decomp} behaves differently and 
is controlled via different techniques depending on the result to be proven, 
we name each of these terms: 
\begin{equation}
\label{eq:def-L-Q-A}
\L_{\tau} \defeq \inner{\mustar-\mustar_{\tau}}{\varepsilon}
\, , \qquad 
\Q_{\tau}\defeq \norm{\Pi_{\tau}\varepsilon}^2
\qquad\text{and}\qquad 
\A_{\tau}\defeq \norm{\mustar-\mustar_{\tau}}^2 
\, .
\end{equation}
It should be clear that $\L$ stands for ``linear'', $\Q$ stands for ``quadratic'' 
and $\A$ stands for ``approximation error''.
We also define
\begin{equation}
\label{eq:def-psi}
\psi_{\tau}\defeq 2\L_{\tau}-Q_{\tau}+A_{\tau}
\, . 
\end{equation}
Therefore a reformulation of Lemma~\ref{lemma:empirical-decomp} is
\[n\emprisk_n(\tau) = \psi_{\tau} +\norm{\varepsilon}^2 \, .\]
Notice that $\L_{\taustar}=\A_{\taustar}=0$ and $\Q_{\taustar}\geq 0$, hence $\psi_{\taustar}\leq 0$.
Also note that~$\psi$,~$\L$ and~$\Q$ are random quantities depending on~$\varepsilon$.

\subsection{Deterministic bounds}\label{sec:deterministic}

In this section, we provide some deterministic bounds that are used in 
the proofs of Theorems \ref{thm.bounded} and~\ref{th:localization-moment}.

\subsubsection{Approximation error~$\A_{\tau}$} 
We begin by the following result, which is the reason for 
the $\lambdainf_{\taustar} \deltainf^2 $ term 
in Theorem~\ref{thm.bounded}.

\begin{lemma}
\label{lemma:approx-error-minoration:optimal-bound} 
Let $\tau \in \T_n$ be a segmentation such that $D\defeq \dtau <\dstar$. 
Then
\begin{equation}
\label{eq:approx-error-minoration:optimal-bound}
\frac{1}{n}\A_{\tau} 
= \frac{1}{n}\norm{\mustar - \mustar_{\tau}}^2 
\geq \frac{1}{2} \lambdainf_{\taustar}  \deltainf^2 
\, .
\end{equation}
\end{lemma}
The proof of Lemma~\ref{lemma:approx-error-minoration:optimal-bound} 
can be found in Section~\ref{sec:proof-approx-error}. 

\begin{remark}
Lemma~\ref{lemma:approx-error-minoration:optimal-bound} is tight. 
Indeed, consider the simple case $\dtau =1$ and $\dstar = 2$. 
Assume that~$n=2m$ is an even number, and let~$\taustar_1 = m$. 
It follows from definitions \eqref*{eq:def-deltainf} 
and~\eqref*{eq:def-lambdainf} that, in this case,
\[
\deltainf = \hilbertnorm{\mustar_{1}-\mustar_{n}}
\qquad\text{and}\qquad 
\lambdainf_{\taustar} = \frac{1}{2} 
\, .
\]
According to \eqref[name=Eq.~]{eq:computation-projection}, 
$\left(\mustar_{\tau}\right)_i 
= \frac{1}{2}\left(\mustar_{1}+\mustar_{n}\right)$, 
which yields 
\[
\frac{1}{n}\A_{\tau} 
=\frac{1}{4}\hilbertnorm{\mustar_{1}-\mustar_{n}}^2
= \frac{1}{2} \lambdainf_{\taustar} \deltainf^2 
\, .
\] 
Thus, in this particular class of examples, equality holds in 
\eqref*{eq:approx-error-minoration:optimal-bound}.
\end{remark}

We next state an analogous result, valid for any $\tau \in \T_n$, 
which plays a key role in the proofs of Theorems \ref{thm.bounded} 
and~\ref{th:localization-moment}.

\begin{lemma}
\label{lemma:approx-error-minoration-2}
For any~$\tau\in\T_n$, 
\begin{equation}
\label{eq:approx-error-minoration-2}
\frac{1}{n}\A_{\tau} 
\geq \dfrac{1}{2} \min\biggl\{\lambdainf_{\taustar},\frac{1}{n}\dinf(\taustar,\tau)\biggr\}\deltainf^2 
\, .
\end{equation}
\end{lemma}
Lemma~\ref{lemma:approx-error-minoration-2} is proved 
in Section~\ref{sec:proof-approx-error-2}. 

\subsubsection{Linear term~$\L_{\tau}$ and quadratic term~$\Q_{\tau}$} 
The proof of Theorem~\ref{th:localization-moment} 
relies on some deterministic bounds on $L_{\tau}$ and $Q_{\tau}$. 
We start with a preliminary lemma. 

\begin{lemma}
\label{lemma:partial-sum}
For any $\varepsilon_1, \ldots, \varepsilon_n \in \hilbert$, 
\begin{equation}
\label{eq:partial-sum-1}
\frac{1}{2}\max_{1\leq a < b \leq n} \hilbertnorm{\sum_{j=a}^{b}\varepsilon_j} \leq \max_{1\leq k\leq n} \hilbertnorm{\sum_{j=1}^k\varepsilon_j} \eqdef \maxeps 
\, .
\end{equation}
\end{lemma}

\begin{proof}
For every $a < b$, we have: 
\begin{align*}
\hilbertnorm{\sum_{j=a}^{b}\varepsilon_j}
= 
\hilbertnorm{\sum_{j=1}^{b}\varepsilon_j - \sum_{j=1}^{a-1}\varepsilon_j}
\leq 
\hilbertnorm{\sum_{j=1}^{b}\varepsilon_j} + \hilbertnorm{\sum_{j=1}^{a-1}\varepsilon_j}
\leq 2 \maxeps \, . 
\end{align*}
\end{proof}

The following result is a deterministic bound on~$\Q_{\tau}$ in terms of~$\maxeps$.

\begin{lemma}
\label{lemma:majo-V}
Let~$\tau\in\T_n$ be a segmentation.
Then
\[\Q_{\tau}\leq \frac{4D_{\tau}\maxeps^2}{n\lambdainf_{\tau}}.\]
\end{lemma}

\begin{proof}
By \eqref[name=Eq.~]{eq:computation-projection},
\begin{align*}
\Q_{\tau} 
&= \sum_{\ell=1}^{D_{\tau}} \frac{1}{\abs{\tau_{\ell}-\tau_{\ell-1}}} \hilbertnorm{\sum_{j=\tau_{\ell-1}+1}^{\tau_{\ell}}\varepsilon_j}^2 
\\ 
&\leq D_{\tau}\max_{1\leq\ell\leq\dtau}\left\{\frac{1}{\abs{\tau_{\ell}-\tau_{\ell-1}}} \hilbertnorm{\sum_{j=\tau_{\ell-1}+1}^{\tau_{\ell}} \varepsilon_j}^2\right\} 
\\ &
\leq \frac{D_{\tau}}{n\lambdainf_{\tau}}\max_{1\leq\ell\leq\dtau} \hilbertnorm{\sum_{j=\tau_{\ell-1}+1}^{\tau_{\ell}} \varepsilon_j}^2 
\leq \frac{4D_{\tau}}{n\lambdainf_{\tau}}\maxeps^2
\, ,
\end{align*}
where we used Lemma~\ref{lemma:partial-sum} for the last inequality.
\end{proof}

The following result is a deterministic bound on~$\L_{\tau}$.

\begin{lemma}
\label{lemma:majo-linear-part}
For any $\tau\in\T_n$,
\[
\abs{\L_{\tau}} \leq 6 \dstar \max\left\{\dstar,D_{\tau}\right\}\deltasup\maxeps
\, .\]
\end{lemma}
Lemma~\ref{lemma:majo-linear-part} is proved in 
Section~\ref{sec:proof-lemma-linear}. 

\subsection{Concentration}\label{sec:concentration}

In this subsection, we present concentration results on~$\Q_{\tau}$,~$\L_{\tau}$,  
and deviation bounds for~$\maxeps$ --- which will imply deviation bounds 
on $\Q_{\tau}$ and $L_{\tau}$ by Lemmas~\ref{lemma:majo-V} and~\ref{lemma:majo-linear-part}). 
For any $j \in \{1, \ldots, n\}$, $\tau \in \T_n$ and $\ell \in \{1, \ldots, D_{\tau}\}$, 
we define 
\[ 
v_j \defeq \expec{\hilbertnorm{\varepsilon_j}^2} 
\qquad 
v_{\tau , \ell} \defeq \frac{1}{\tau_{\ell}-\tau_{\ell-1}}\sum_{j=\tau_{\ell-1}+1}^{\tau_{\ell}}v_j
\qquad \text{and} \qquad 
v_{\tau}\defeq \sum_{\ell=1}^D v_{\tau , \ell}
\, . 
\]

\paragraph{Concentration under Assumption~\ref{assump:bounded-kernel}} 
The first result takes care of the linear term~$\L_{\tau}$ 
when Assumption~\ref{assump:bounded-kernel} is satisfied. 
\begin{lemma}[Prop.~3 of~\citet{Arl_Cel_Har:2012}]
\label{lemma:concentration-linear}
Suppose that Assumption~\ref{assump:bounded-kernel} holds true.
Then for any $x>0$, with probability at least $1-2\e^{-x}$, 
for any $\theta>0$, 
\[
\abs{L_{\tau}}
\leq \theta A_{\tau} + \left( \frac{4}{3}+\frac{1}{2\theta}\right) M^2 x 
\, .
\] 
\end{lemma}

The next result deals with the quadratic term~$\Q_{\tau}$ 
when Assumption~\ref{assump:bounded-kernel} is satisfied. 
\begin{lemma}[Prop.~1 of~\citet{Arl_Cel_Har:2012}]
\label{lemma:concentration-quadratic}
Suppose that Assumption~\ref{assump:bounded-kernel} holds true.
Then for any $x>0$, with probability at least $1-\e^{-x}$,
\[Q_{\tau} - v_{\tau} \leq \left(x+2\sqrt{2xD_{\tau}}\right)\frac{14M^2}{3}.\]
\end{lemma}

We merge Lemmas~\ref{lemma:concentration-linear} 
and~\ref{lemma:concentration-quadratic} for convenience. 
\begin{lemma}
\label{lemma:concentration-synt-alt}
Suppose that Assumption~\ref{assump:bounded-kernel} holds true.
Take any~$\lambda >1$ and~$\tau \in \T_n$ be a segmentation.
Then, there exists an event~$\Omega^{(0)}_{\tau,\lambda}$ 
of probability greater than $1-3\e^{-\lambda\dtau}$ on which\textup{:} 
\[
\psi_{\tau} 
\geq \frac{1}{3}\A_{\tau}-\frac{74}{3}\lambda \dtau M^2
\, .
\]
\end{lemma}

\begin{proof}
According to Lemma~\ref{lemma:concentration-linear} with $\theta=1/3$ 
and $x = \lambda \dtau$, 
there exists an event~$\Omega_{\tau,\lambda}^{(1)}$ on which 
$\abs{\L_{\tau}}\leq \frac{1}{3}\A_{\tau} + \frac{17}{6}\lambda\dtau M^2$, 
with $\proba{\Omega_{\tau,\lambda}^{(1)}}\geq 1-2\e^{-\lambda\dtau}$.
Lemma~\ref{lemma:concentration-quadratic} with $x = \lambda \dtau$ 
gives~$\Omega_{\tau,\lambda}^{(2)}$ on which 
$\Q_{\tau}-v_{\tau} \leq \frac{14}{3}\left (\lambda+2\sqrt{2\lambda}\right ) \dtau M^2$, 
with $\proba{\Omega_{\tau,\lambda}^{(2)}}\geq 1-\e^{- \lambda \dtau}$.
Then, $\Omega^{(0)}_{\tau,\lambda} \defeq\Omega_{\tau,\lambda}^{(1)}\cap\Omega_{\tau,\lambda}^{(2)}$ has a probability larger than 
$1- 3 \e^{- \lambda \dtau}$ by the union bound. 
Since for any $1\leq \ell\leq \dtau$, $v_{\tau , \ell}\leq M^2$, we have $v_{\tau}=\sum_{\ell=1}^{\dtau}v_{\tau , \ell}\leq D_{\tau} M^2$. 
Hence, by definition \eqref*{eq:def-psi} of~$\psi_{\tau}$ 
and using that $\lambda \geq 1$, 
on the event~$\Omega^{(0)}_{\tau,\lambda}$, we have: 
\begin{align*}
\psi_{\tau} 
&\geq \dfrac{1}{3}\A_{\tau} 
- \left(\frac{31}{3}\lambda + \frac{28}{3}\sqrt{2}\sqrt{\lambda}+1\right ) \dtau M^2 
\\
&\geq \dfrac{1}{3}\A_{\tau} 
- \lambda \left(\frac{31}{3}+ \frac{28}{3}\sqrt{2}+1\right ) \dtau M^2
\, . 
\end{align*}
\end{proof}

\begin{remark}
It is also possible to obtain an upper bound for~$\psi_{\tau}$: 
by Lemma~\ref{lemma:concentration-linear}, 
for every $\lambda \geq 0$, 
on the event $\Omega_{\tau,\lambda}^{(2)} \subset \Omega_{\tau,\lambda}^{(0)}$, 
\[
\psi_{\tau} \leq \frac{5}{3}\A_{\tau} + \frac{17}{3}\lambda\dtau M^2
\, .
\]
However, we do not need this result thereafter.
\end{remark}

\paragraph{Concentration under Assumption~\ref{assump:bounded-variance}} 
Lemma~\ref{lemma:majo-V} and~\ref{lemma:majo-linear-part} directly translate 
upper bounds on~$\maxeps$ into controls of~$\L_{\tau}$ and~$\Q_{\tau}$.
Under Assumption~\ref{assump:bounded-variance}, this is achieved via the following lemma, 
a Kolmogorov-like inequality for the noise in the RKHS.
This result is a straightforward generalization of the inequality obtained 
by \citet{Kol:1928} into the Hilbert setting.
A more precise result (for real random variables only) can be found in~\citep{Haj_Ren:1955}, 
of which we follow the proof. 
The scheme of \citet{Haj_Ren:1955} adapts well in our setting even though we do not need the full result.

\begin{lemma}
\label{lemma:kolmogorov-moment}
If Assumption~\ref{assump:bounded-variance} holds true, 
then, for any~$x >0$,
\begin{equation}
\label{eq:kolmogorov-moment}
\proba{\maxeps \geq x}\leq \dfrac{1}{x^2}\sum_{j=1}^n v_j \, .
\end{equation}
\end{lemma}

We prove Lemma~\ref{lemma:kolmogorov-moment} in Section~\ref{sec:proof:lemma:kolmogorov-moment}. 

\begin{remark}
We can reformulate Lemma~\ref{lemma:kolmogorov-moment} as follows. 
For any $y >0$, there exists an event of probability at least $1 - y^{-2}$ 
on which $M_n < y \sqrt{\sum_{i=j}^n v_j} \leq y\sqrt{nV}$.
Equivalently, for any $z \geq 0$, there exists an event of probability 
at least $1-\e^{-z}$ such that 
$M_n < \e^{z/2}\sqrt{\sum_{i=j}^n v_j} \leq \e^{z/2} \sqrt{nV}$.
\end{remark}

\subsection{Proof of Theorem~\ref{thm.bounded}}
\label{sec:proof-main-result}

We follow the strategy described at the beginning of Section~\ref{sec:proofs}.

\paragraph{Definition of $\Omega$} 
Let us define 
$\Omega\defeq \bigcap_{\tau\in\T_n}\Omega^{(0)}_{\tau,\lambda}$ 
with $\lambda = y+\log n + 1 > 1$, 
where we recall that $\Omega^{(0)}_{\tau,\lambda}$ 
is defined in Lemma~\ref{lemma:concentration-synt-alt}. 
By the union bound, and since the~$\Omega^{(0)}_{\tau,\lambda}$ 
have probability greater than $1-3\e^{-\lambda\dtau}$,
\[\proba{\Omega}\geq 1-3\sum_{\tau\in\T_n}\e^{-\lambda\dtau}
\, .\]
The inequality $\proba{\Omega} \geq 1-\e^{-y}$ follows since 
\begin{align*}
\sum_{\tau\in\T_n}\e^{-\lambda\dtau} 
=\sum_{d=1}^{n}\binom{n-1}{d-1}\e^{-\lambda d} 
&= \e^{-\lambda}\left(1+\e^{-\lambda}\right)^{n-1} 
\\ & 
\leq \e^{-\lambda} \exp\left( (n-1) \e^{-\lambda}\right) \\
&= \frac{\e^{-y}}{n \e} \exp\left( \frac{n-1}{n} \e^{-1-y}\right) \\
&\leq  \e^{-y} \frac{\exp(\e^{-1})}{n \e} 
\leq  0.27 \e^{-y} 
\, , 
\end{align*}
where the last inequality uses that $n \geq 2$. 
From now on we work exclusively on~$\Omega$.

\paragraph{Key argument}
We now make the simple (but crucial) observation that 
$\crit(\taustar) \geq \crit(\tauhat)$, 
hence 
\[ 
n \pen(\tauhat) + \psi_{\tauhat} 
\leq n \pen(\taustar) + \psi_{\taustar} 
\leq n \pen(\taustar) 
= C D_{\taustar}  M^2 
\, . 
\]
Since we work on $\Omega$, 
by definition of $\Omega^{(0)}_{\tau,\lambda}$ 
in Lemma~\ref{lemma:concentration-synt-alt}, 
for any $\tau \in \T_n$, we have: 
\begin{equation} 
\notag 
\psi_{\tau} 
\geq \frac{1}{3}\A_{\tau}-\frac{74}{3}\lambda\dtau M^2 
\, . 
\end{equation}
Therefore, we get: 
\begin{equation}
\label{eq.thm.bounded.pr.key}
C D_{\taustar} M^2 
\geq 
\frac{1}{3}\A_{\tauhat} + \left( C -\frac{74}{3}\lambda \right) D_{\tauhat} M^2 
\, . 
\end{equation}

\paragraph{Proof that $\dhat \leq \dstar$} 
Since $C>74 \lambda/3$ 
(by the lower bound in assumption~\eqref*{eq:hyp-main-result-synth}), 
$M^2>0$ and $\A_{\tauhat} \geq 0$, 
\eqref[name=Eq.~]{eq.thm.bounded.pr.key} implies that 
\[
D_{\tauhat} 
\leq 
\frac{C}{C -\frac{74}{3}\lambda} D_{\taustar} 
\, . 
\]
The lower bound in assumption~\eqref*{eq:hyp-main-result-synth} 
ensures that 
\[
\frac{C}{C -\frac{74}{3}\lambda} < \frac{D_{\taustar} + 1}{D_{\taustar} }
\]
hence $\dhat \leq \dstar$ on $\Omega$.

\paragraph{Proof that $D_{\tauhat} \geq \dstar$}
Since $C>74 \lambda/3$ 
(by the lower bound in assumption~\eqref*{eq:hyp-main-result-synth}),    
\eqref[name=Eq.~]{eq.thm.bounded.pr.key} implies that $\A_{\tauhat} \leq 3 C D_{\taustar} M^2$.
A direct consequence of~\eqref*{eq:hyp-main-result-synth} 
is that $\A_{\tauhat} < \frac{1}{2}n\lambdainf_{\taustar} \deltainf^2$, hence $D_{\tauhat} \geq \dstar$ by Lemma~\ref{lemma:approx-error-minoration:optimal-bound}.

\paragraph{Loss between $\tauhat$ and $\taustar$}
We have proved that $\dhat = \dstar$ on $\Omega$, 
therefore, \eqref[name=Eq.~]{eq.thm.bounded.pr.key} 
can be rewritten 
\begin{equation}
\label{eq.thm.bounded.maj-approx}
\A_{\tauhat} \leq 74 \lambda \dstar M^2 
\, . 
\end{equation}
By Lemma~\ref{lemma:approx-error-minoration-2} and the definition of $\lambda$, we get 
\begin{equation}
\label{eq.thm.bounded.pr.1}
\min\biggl\{\lambdainf_{\taustar},\frac{1}{n}\dinf(\taustar,\tauhat)\biggr\} 
\leq 
\frac{148 \dstar M^2}{ \deltainf^2} \cdot \frac{y+\log n + 1}{n} 
= \vitThmbounded(y) 
\, . 
\end{equation}
Remark that assumption~\eqref*{eq:hyp-main-result-synth} implies that 
\[
\dfrac{\deltainf^2}{M^2} 
\frac{\lambdainf_{\taustar} }{6 \dstar} n 
> \frac{74}{3}(\dstar+1)(y+\log n + 1)
\]
hence 
\[
\lambdainf_{\taustar}
> (\dstar+1) \frac{148 \dstar M^2}{ \deltainf^2} \cdot \frac{y+\log n + 1}{n}  
> \vitThmbounded(y) 
\, . 
\]
Therefore, \eqref[name=Eq.~]{eq.thm.bounded.pr.1} 
can be simplified into 
\[
\frac{1}{n}\dinf(\taustar,\tauhat)  
\leq 
 \vitThmbounded(y) 
\, . 
\qed
\]

\begin{remark}
\label{rk.thm.bounded.deltan}
The proof of Theorem~\ref{thm.bounded} 
generalizes to $\tauhat$ defined by 
\[ 
\tauhat \in \argmin_{\tau\in\T_n \,/\, \lambdainf_{\tau} \geq \delta_n} 
\bigl\{ \crit(\tau)\bigr\}
\]
instead of~\eqref*{eq:original-problem},  
for any $\delta_n \geq 0$ 
such that $\lambdainf_{\taustar} \geq \delta_n$. 
Indeed, this assumption allows to write $\crit(\taustar) \geq \crit(\tauhat)$ 
in the key argument, 
and the rest of the proof can stay unchanged (with the same event $\Omega$). 
More generally, any constraint can be added in the argmin defining $\tauhat$, 
provided that $\taustar$ satisfies this constraint. 
\end{remark}

\subsection{Proof of Theorem~\ref{th:localization-moment}}
\label{sec:proof-th-localization-moment}

We follow the strategy described at the beginning of Section~\ref{sec:proofs}. 
Throughout the proof, we write $\tauhatmom$ as a shortcut for $\tauhat(\dstar,\delta_n)$. 

\paragraph{Key argument} 
By definition \eqref*{eq:alternative-problem} of $\tauhatmom = \tauhat(\dstar,\delta_n)$, 
since we assume $\lambdainf_{\taustar} \geq \delta_n$, 
\[ 
\emprisk_n(\taustar) \geq \emprisk_n ( \tauhatmom )
\]
hence 
\[
0 
\geq \psi_{\taustar} 
\geq \psi_{\tauhatmom} 
= A_{\tauhatmom} + 2 L_{\tauhatmom} - Q_{\tauhatmom}
\, . 
\]
By Lemma~\ref{lemma:majo-V}, Lemma~\ref{lemma:majo-linear-part} 
and the facts that $D_{\tauhatmom} = \dstar$ 
and $\lambdainf_{\tauhatmom} \geq \delta_n$, 
we get 
\[ 
0 
\geq \psi_{\tauhatmom}
\geq 
A_{\tauhatmom} 
- 12 \dstar^2 \deltasup \maxeps
- \frac{4 \dstar \maxeps^2}{n \delta_n} 
\]
hence, using Lemma~\ref{lemma:approx-error-minoration-2}, 
\begin{equation} 
\label{eq.th:localization-moment.pr.key}
\min\biggl\{\lambdainf_{\taustar},\frac{1}{n}\dinf(\taustar,\tauhatmom)\biggr\}
\leq \frac{24 \dstar^2 \deltasup }{\deltainf^2 } \frac{\maxeps}{n}
+ \frac{8 \dstar }{\deltainf^2 } \frac{\maxeps^2}{n^2 \delta_n } 
\, . 
\end{equation}

\paragraph{Definition of $\Omega_2$} 
We define 
\[ 
\Omega_2 
\defeq \{ M_n \leq y \sqrt{n V} \} 
\, . 
\]
By Lemma~\ref{lemma:kolmogorov-moment}, 
under  Assumption~\ref{assump:bounded-variance}, 
$\proba{ \Omega_2 } \geq 1 - y^{-2}$.

\paragraph{Conclusion} 
By definition of $\Omega_2$, 
\eqref[name=Eq.~]{eq.th:localization-moment.pr.key} implies that on $\Omega_2$: 
\[ 
\min\biggl\{\lambdainf_{\taustar},\frac{1}{n}\dinf(\taustar,\tauhatmom)\biggr\}
\leq 
24 (\dstar)^2  \frac{ \deltasup \sqrt{V} }{ \deltainf^2}  \frac{ y }{ \sqrt{n} }
+ 8 \dstar \frac{ V }{ \deltainf^2 } \frac{ y^2 }{ n \delta_n }
= \vitThmmoment(y,\delta_n) 
\, . 
\]
Since we assume $\vitThmmoment(y,\delta_n)  < \lambdainf_{\taustar}$, 
the result follows. 
\qed

\appendix
\section{Additional notation}

In this appendix are collected a large part of the technical details of the proofs that precede.
Some additional notation used solely in the appendix are introduced below.

We denote by $\lstar_1,\ldots,\lstar_{\dstar}$ the segments of~$\taustar$, 
that is, 
\[
\lstar_i = \set{\taustar_{i-1}+1,\dots,\taustar_i}
\, .
\]
For any segment $\lambda$ of $\tau \in \T_n$, 
we denote by $\mustar_{\lambda}$ the value of~$\mustar_{\tau}$ 
on $\lambda$, 
which does not depend on $\tau$ and is given by \eqref{eq:computation-projection}: 
\begin{equation}
\label{eq:computation-projection.reformulee}
\mustar_{\lambda} 
= \frac{1}{\card{\lambda}} \sum_{j \in \lambda} \mustar_j
\, . 
\end{equation}

\section{Proofs}

\subsection{Proof of Lemma~\ref{lemma:equality-distances}}\label{sec:proof-equality-distances}

\paragraph{Proof of (i)}
We set $D^i\defeq D_{\tau^i}$ for $i\in\{1,2\}$.
Let us show first that $\dinfb(\tau^1,\tau^2) = \dinf(\tau^1,\tau^2)$.
Take any $i\in\bigl\{1,\dots, D^1 - 1\bigr\}$, by the definition of~$\lambdainf_{\tau^1}$,
\[
\abs{\tau_i^1 - \tau_{D^2}^2}=\abs{\tau_i^1 - n} \geq n\lambdainf_{\tau^1} > n\lambdainf_{\tau^1} / 2 \geq n\min\bigl\{\lambdainf_{\tau^1},\lambdainf_{\tau^2}\bigr\}/2 
\, ,
\]
which is greater than $\dinf(\tau^1,\tau^2)$ by assumption.
In the same fashion we can prove that $\abs{\tau_i^1 - \tau_0^2} > \dinf(\tau^1,\tau^2)$.
Hence, for any~$i \in \{1, \ldots, D^1 - 1 \}$,
\[\min_{0\leq j\leq D^2}\abs{\tau_i^1 - \tau_j^2} = \min_{1\leq j\leq D^2 - 1}\abs{\tau_i^1 - \tau_j^2},\]
which proves that $\dinfb(\tau^1,\tau^2) = \dinf(\tau^1,\tau^2)$.

Next, we prove that $D^1=D^2$ and $\dinfD(\tau^1,\tau^2) = \dinf(\tau^1,\tau^2)$.
Define $\phi : \set{1,\dots,D^1-1}\to\set{1,\dots,D^2-1}$ such that
\[
\bigl\{\phi(i)\bigr\} = \argmin_{1\leq j\leq D^2-1}\abs{\tau_i^1-\tau_j^2}
\]
for all $i  \in \{1, \ldots, D^1 - 1 \}$.
This mapping is well-defined: indeed, suppose that~$j,k\in\set{1,\dots,D^2-1}$ both realize the minimum for some $i\in\set{1,\dots,D^1-1}$.
Since we assumed $\frac{1}{n}\dinf(\tau^1,\tau^2)<\min\bigl\{\lambdainf_{\tau^1},\lambdainf_{\tau^2}\bigr\}/2$,
\[\abs{\tau_i^1-\tau_j^2} = \abs{\tau_i^1-\tau_k^2} \leq \dinf(\tau^1,\tau^2) < n\min\bigl\{\lambdainf_{\tau^1},\lambdainf_{\tau^2}\bigr\}/2.\]
By the triangle inequality,
\[\abs{\tau_j^2-\tau_k^2}<n\min\bigl\{\lambdainf_{\tau^1},\lambdainf_{\tau^2}\bigr\}\leq n\lambdainf_{\tau^2},\]
hence $j=k$. 
Next, we show that~$\phi$ is increasing.
Take $i,j\in\set{1,\dots,D^1-1}$ such that $i<j$.
Recall that~$\tau^k_{\cdot}$ is increasing ($k=1,2$).
Then
\begin{align*}
\tau_{\phi(i)}^2 - \tau_{\phi(j)}^2 &= \tau_{\phi(i)}^2-\tau_i^1 + \tau_i^1-\tau_j^1+\tau_j^1-\tau_{\phi(j)}^2 \\
&= \tau_{\phi(i)}^2-\tau_i^1 - \abs{\tau_i^1-\tau_j^1}+\tau_j^1-\tau_{\phi(j)}^2 \\
&\leq \abs{\tau_{\phi(i)}^2-\tau_i^1} - \abs{\tau_i^1-\tau_j^1}+\abs{\tau_j^1-\tau_{\phi(j)}^2} \\
& \leq 2\dinf(\tau^1,\tau^2) -\abs{\tau_i^1-\tau_j^1}\\
&< n\min\bigl\{\lambdainf_{\tau^1},\lambdainf_{\tau^2}\bigr\} - n\lambdainf_{\tau^1} \leq 0
\, .
\end{align*}
Hence $\phi(i) < \phi(j)$, so $\phi$ is increasing. 
As a consequence, $\phi$ is injective and we get $D^1 \leq D^2$. 
The same argument, exchanging $\tau^1$ and $\tau^2$, shows that $D^2 \leq D^1$. 
Therefore, $D^1 = D^2$ and $\phi$ is an increasing permutation of $\set{1,\dots,D^1-1}$, hence it is the identity.
As a consequence, $\dinfD(\tau^1,\tau^2) = \dinf(\tau^1,\tau^2)$.

Finally, since~$\dinfD$ is symmetric, $\distinf^{(i)}(\tau^1,\tau^2)=\disthaus^{(i)}(\tau^1,\tau^2)$ for any $i\in\set{1,2,3}$.

\paragraph{Proof of (ii)} 
Since $D_{\tau^1}=D_{\tau^2}$, we can set $D=D_{\tau^1}=D_{\tau^2}$. 
Next, define $\phi(i)\defeq \argmin_{1\leq j\leq D-1}\abs{\tau_i^1-\tau_j^2}$ and 
$C_{\phi}(i)\defeq \card{\phi(i)}$ for all $i  \in \{1, \ldots, D - 1 \}$.
Clearly,~$C_{\phi} (i)\geq 1$ for any~$i$.
Let us show that we actually have $C_{\phi} (i)=1$.

Take $i$ and $j$ distincts elements of $\set{1,\ldots,D-1}$, 
and suppose that $\phi(i)\cap\phi(j)$ is non-empty.
Let $k$ be any element of $\phi(i)\cap\phi(j)$. 
By the triangle inequality and the definition of~$\dinf$,
\[
n\lambdainf_{\tau^1} 
\leq \abs{\tau_i^1-\tau_{j}^1} 
\leq \abs{\tau_i^1-\tau_k^2} + \abs{\tau_k^2-\tau_{j}^1} 
\leq 2 \dinf (\tau^1, \tau^2) 
< n\lambdainf_{\tau^1}
\, .
\]
Hence, the $\phi(i)$ are disjoint and we can write
$\sum_{i=1}^{D-1} C_{\phi} (i) = D-1$, which clearly implies that $C_{\phi} (i)=1$.

From now on, we identify~$\phi(i)$ with its unique element. 
Let us show that~$\phi$ is increasing similarily to what we have done 
for proving (i). 
Take $i,j\in\set{1,\dots,D-1}$ such that $i<j$.
We showed that
\[
\tau_{\phi(i)}^2 - \tau_{\phi(j)}^2 \leq 2\dinf(\tau^1,\tau^2) -\abs{\tau_i^1-\tau_j^1},
\]
thus according to the definition of~$\lambdainf_{\tau^1}$, and our assumption,
\[\tau_{\phi(i)}^2 - \tau_{\phi(j)}^2< n\lambdainf_{\tau^1} - n\lambdainf_{\tau^1} \leq 0.\]
Hence $\phi(i) < \phi(j)$: $\phi$ is increasing. 
As a consequence, 
\[
\dinf(\tau^1,\tau^2) = \dinf(\tau^2,\tau^1)=\dinfH(\tau^1,\tau^2)
\, .
\]
\qed

\subsection{The Frobenius loss}
\label{sec:proof-prop-equivalence}

\subsubsection{A formula for $\distfrob^2$} 
We start by proving a general formula for $\distfrob$, 
which is stated by \citet{Laj_Arl_Bac:2014}, we prove it here for completeness: 
\begin{equation}
\label{eq.dF.formule-exacte}
\forall \tau^1,\tau^2 \in \T_n, \qquad 
\distfrob(\tau^1,\tau^2)^2 
=  D_{\tau^1}+D_{\tau^2} 
- 2 \sum_{k=1}^{D_{\tau^1}} \sum_{\ell=1}^{D_{\tau^2}} 
\frac{\abs{\lambda^1_k\cap\lambda^2_{\ell}}^2}{\abs{\lambda^1_k}\times\abs{\lambda^2_{\ell}}} 
\, .
\end{equation}
Indeed, by definition, we have
\begin{gather*}
\distfrob(\tau^1,\tau^2)^2 
= 
\Tr\bigl( (\Pi_{\tau^1} - \Pi_{\tau^2})^2 \bigr)
= 
\underbrace{\Tr(\Pi_{\tau^1})}_{=D_{\tau^1}} + \underbrace{\Tr(\Pi_{\tau^2})}_{=D_{\tau^2}} 
- 2 \Tr( \Pi_{\tau^1} \Pi_{\tau^2} )
\\ 
\text{and}\-
\Tr( \Pi_{\tau^1} \Pi_{\tau^2} )
= 
\sum_{i=1}^n \sum_{j=1}^n \frac{\indic{\lambda_1(i)=\lambda_1(j) \text{ and } \lambda_2(i)=\lambda_2(j)}}{\abs{\lambda_1(i)} \abs{\lambda_2(i)}}
= 
\sum_{k=1}^{D_{\tau^1}} \sum_{\ell=1}^{D_{\tau^2}} 
\frac{\abs{\lambda^1_k\cap\lambda^2_{\ell}}^2}{\abs{\lambda^1_k}\times\abs{\lambda^2_{\ell}}}
\, , 
\end{gather*}
where we denoted by $\lambda_k(i)$ the segment of $\tau^k$ to which $i \in \{1, \ldots, n\}$ belongs.

\subsubsection{Proof of Eq.~\eqref{eq.distfrob.ineq}}
Eq.~\eqref{eq.distfrob.ineq} is stated by \citet{Laj_Arl_Bac:2014}. 
The upper bound is a straightforward consequence of Eq.~\eqref*{eq.dF.formule-exacte}. 
We prove the lower bound here for completeness. 
We remark that 
\begin{align*}
\sum_{k=1}^{D_{\tau^1}} \sum_{\ell=1}^{D_{\tau^2}} 
\frac{\abs{\lambda^1_k\cap\lambda^2_{\ell}}^2}{\abs{\lambda^1_k}\times\abs{\lambda^2_{\ell}}}
\leq 
\sum_{k=1}^{D_{\tau^1}} \sum_{\ell=1}^{D_{\tau^2}} 
\frac{\abs{\lambda^1_k\cap\lambda^2_{\ell}}}{\abs{\lambda^1_k}}
= 
D_{\tau^1} 
\, , 
\end{align*}
hence 
Eq.~\eqref*{eq.dF.formule-exacte} shows that 
\[
\distfrob(\tau^1,\tau^2)^2 
\geq D_{\tau^2} - D_{\tau^1}
\, . 
\]
The lower bound follows since $\tau^1$ and $\tau^2$ play symmetric roles. 
\qed

\subsubsection{Proof of Proposition~\ref{prop:equivalence}}
Throughout the proof, we write $D = D_{\tau^1} = D_{\tau^2}$, 
$\epsilon = n^{-1} \dinf(\tau^1, \tau^2)$ 
and we denote by 
$(\lambda^1_k)_{1 \leq k \leq D}$ and $(\lambda^2_k)_{1 \leq k \leq D}$ 
the segments of $\tau^1$ and $\tau^2$, respectively.

\paragraph{Preliminary remark} 
Since we assume that $D_{\tau^1} = D_{\tau^2}$ and 
$\frac{1}{n}\dinf(\tau^1,\tau^2)=\epsilon <\lambdainf_{\tau^1}/2$, 
point (ii) in Lemma~\ref{lemma:equality-distances} 
shows that $\dinf(\tau^1,\tau^2) = \dinfH(\tau^1,\tau^2) 
= d_{\infty}^{(3)} (\tau^1,\tau^2)$. 
In other words, for every $k \in \{1, \ldots, D - 1\}$, we have 
$\abs{\tau^1_k - \tau^2_k} \leq n \epsilon$, 
and some $k_0 \in \{1, \ldots, D-1 \}$ exists such that 
$\abs{\tau^1_{k_0} - \tau^2_{k_0}} = n \epsilon$.  
As a consequence, for every $k \in \{1, \ldots, D - 1\}$, 
\begin{equation} 
\label{prop:equivalence:eq.pr.1}
\abs{\card{\lambda_k^1}-\card{\lambda_k^2}} \leq 2n\epsilon
\qquad \text{and} \qquad 
\abs{ \lambda^1_k \cap \lambda^2_k } \geq \abs{ \lambda^1_k } - 2 n \epsilon 
\, . 
\end{equation}

\paragraph{Upper bound for $\distfrob(\tau^1,\tau^2)^2$} 
We focus on the sum appearing in the right-hand side of Eq.~\eqref*{eq.dF.formule-exacte}. 
Using Eq.~\eqref*{prop:equivalence:eq.pr.1}, we get: 
\begin{align*}
\sum_{k=1}^{D} \sum_{\ell=1}^{D} 
\frac{\abs{\lambda^1_k\cap\lambda^2_{\ell}}^2}{\abs{\lambda^1_k}\times\abs{\lambda^2_{\ell}}} 
&\geq 
\sum_{k=1}^D \frac{\abs{\lambda^1_k\cap\lambda^2_k}^2}{\abs{\lambda^1_k}\times\abs{\lambda^2_k}} 
\\
&\geq 
\sum_{k=1}^D \left[ \frac{ \left(\abs{\lambda^1_k}- 2n\epsilon \right)^2 }{ \abs{\lambda^1_k}\times \left(\abs{\lambda^1_k} + 2n\epsilon\right) } \right] 
= \sum_{k=1}^{D} \frac{\left(1-\frac{2n\epsilon}{\abs{\lambda^1_k}}\right)^2}{1+\frac{2n\epsilon}{\abs{\lambda^1_k}}} \\
&\geq \sum_{k=1}^D \left(1-\frac{6n\epsilon}{\abs{\lambda^1_k}}\right) 
\geq D-\frac{6\epsilon D}{\lambdainf_{\tau^1}}
\, ,
\end{align*}
since for any $x \geq 0$, $\frac{(1-x)^2}{1+x} \geq 1-3x$. 
The upper bound follows, using Eq.~\eqref*{eq.dF.formule-exacte}.

\paragraph{Lower bound for $\distfrob(\tau^1,\tau^2)^2$} 
As shown in the preliminary remark, 
there exists some $k_0 \in \{1, \ldots, D-1 \}$ such that 
$\abs{\tau^1_{k_0} - \tau^2_{k_0}} = n \epsilon$. 
First consider the case where $\tau^1_{k_0} < \tau^2_{k_0}$. 
Then, by definition of $\distfrob$ and $\Pi_{\tau}$, we have: 
\begin{align*}
\distfrob(\tau^1,\tau^2)^2
&\defeq \sum_{1 \leq i , j \leq n} (\Pi_{\tau^1}-\Pi_{\tau^2})_{i,j}^2
\\
&\geq \sum_{i \in \lambda^1_{k_0+1} \cap \lambda^2_{k_0}}  
\, \sum_{j \in \lambda^1_{k_0+1} \cap \lambda^2_{k_0+1}} \frac{1}{\abs{ \lambda^1_{k_0+1} }^2} \\
&\phantom{blabla}+ \sum_{i \in \lambda^1_{k_0+1} \cap \lambda^2_{k_0+1}} 
\, \sum_{j \in \lambda^1_{k_0+1} \cap \lambda^2_{k_0}} \frac{1}{\abs{ \lambda^1_{k_0+1} }^2}
\\
&= \frac{ 2 \abs{\lambda^1_{k_0+1} \cap \lambda^2_{k_0}} 
\cdot \abs{\lambda^1_{k_0+1} \cap \lambda^2_{k_0+1}} 
}{ \abs{ \lambda^1_{k_0+1} }^2 }
\, . 
\end{align*}
Now, remark that 
$\abs{\lambda^1_{k_0+1} \cap \lambda^2_{k_0}} = n \epsilon$, 
by the preliminary remark and our assumption $\tau^2_{k_0} > \tau^1_{k_0}$. 
Using also Eq.~\eqref*{prop:equivalence:eq.pr.1}, we get: 
\begin{align*}
\distfrob(\tau^1,\tau^2)^2
\geq \frac{ 2 n \epsilon \bigl( \abs{\lambda^1_{k_0+1}} - 2 n \epsilon \bigr)  
}{ \abs{ \lambda^1_{k_0+1} }^2 }
\geq \frac{ 2 n \epsilon }{ 3 \lambdasup_{\tau^1} }
\, , 
\end{align*}
since 
$\abs{\lambda^1_{k_0+1}} - 2 n \epsilon \geq \abs{\lambda^1_{k_0+1}}/3$ 
and $\abs{\lambda^1_{k_0+1}} \leq \lambdasup_{\tau^1}$.  
When $\tau^1_{k_0} > \tau^2_{k_0}$, we apply the same reasoning, 
restricting the sum over $i,j$ in the definition of $\distfrob$ to 
$i \in \lambda^1_{k_0} \cap \lambda^2_{k_0}$ 
and $j \in \lambda^1_{k_0} \cap \lambda^2_{k_0+1}$ 
(plus its symmetric). 
We obtain the same lower bound, which concludes the proof. 
\qed

\subsection{Lower bounds on the approximation error}

This section provides the proofs of 
Lemmas \ref{lemma:approx-error-minoration:optimal-bound} 
and~\ref{lemma:approx-error-minoration-2}.

\subsubsection{Preliminary lemma}
We start with a lemma useful in the two proofs. 

\begin{lemma}
\label{le.approx-lower}
If a segment $\lambda \subset \{ 1, \ldots, n\}$ intersects only 
two segments of $\taustar$, $\lstar_i$ and~$\lstar_{i+1}$, 
then we have\textup{:} 
\begin{align} 
\label{eq:lemma:approx-error-minoration:claim-1}
\sum_{j\in\lambda} \hilbertnorm{ \mustar_j - \mustar_{\lambda} }^2 
&= \frac{\card{\lambda\cap\lstar_i}\cdot \card{\lambda\cap\lstar_{i+1}}}{ 
\card{\lambda\cap\lstar_i}+\card{\lambda\cap\lstar_{i+1}}} 
\hilbertnorm{\mustar_{\lstar_{i+1}} - \mustar_{\lstar_i}}^2 
\\ 
\label{eq:approx-error-minoration:remark}
\geq &
\left(\frac{\card{\lambda\cap\lstar_i}}{\card{\lstar_i}} 
\wedge\frac{\card{\lambda\cap\lstar_{i+1}}}{\card{\lstar_{i+1}}} \right)
\cdot 
\frac{\card{\lstar_i}\cdot\card{\lstar_{i+1}}}{ 
\card{\lstar_i}+\card{\lstar_{i+1}}} 
\cdot\hilbertnorm{\mustar_{\lstar_{i+1}} - \mustar_{\lstar_i}}^2
\, .
\end{align}
\end{lemma}

\begin{proof}
We first prove Eq.~\eqref*{eq:lemma:approx-error-minoration:claim-1}. 
Since~$\lambda$ only intersects~$\lstar_i$ and~$\lstar_{i+1}$, we have: 
\begin{align} \notag 
\sum_{j\in\lambda} \hilbertnorm{ \mustar_j - \mustar_{\lambda} }^2  
&= \sum_{j\in \lambda\cap\lstar_i} \hilbertnorm{\mustar_j - \mustar_{\lambda}}^2 
+ \sum_{j\in \lambda\cap\lstar_{i+1}} \hilbertnorm{\mustar_j - \mustar_{\lambda}}^2
\\
&= \card{\lambda\cap \lstar_i}\cdot \hilbertnorm{\mustar_{\lstar_i} - \mustar_{\lambda}}^2 
+ \card{\lambda\cap \lstar_{i+1}}\cdot \hilbertnorm{\mustar_{\lstar_{i+1}} - \mustar_{\lambda}}^2 
\, .
\label{claim.minor-approx-err.1.pr:eq-1}
\end{align}
Since $\mustar_{\lambda}$ is given by Eq.~\eqref*{eq:computation-projection.reformulee}, 
we obtain 
\begin{align*}
\hilbertnorm{\mustar_{\lstar_i} - \mustar_{\lambda}}^2 
= \hilbertnorm{\dfrac{1}{\card{\lambda}}\sum_{j\in\lambda}\left (\mustar_{\lstar_i}-\mustar_j\right )}^2 
&= \hilbertnorm{\dfrac{1}{\card{\lambda}}\sum_{j\in\lambda\cap\lstar_{i+1}}\left (\mustar_{\lstar_i}-\mustar_{\lstar_{i+1}}\right )}^2 
\\
&= \frac{\card{\lambda\cap\lstar_{i+1}}^2}{\card{\lambda}^2} 
\hilbertnorm{ \mustar_{\lstar_{i+1}} - \mustar_{\lstar_i} }^2
\, .
\end{align*}
The same computation on~$\lambda\cap\lstar_{i+1}$ yields
\[
\hilbertnorm{\mustar_{\lstar_{i+1}} - \mustar_{\lambda}}^2 
= \frac{ \card{\lambda\cap\lstar_i}^2 }{\card{\lambda}^2} \hilbertnorm{ \mustar_{\lstar_{i+1}} - \mustar_{\lstar_i} }^2
\, .
\]
Therefore, 
Eq.~\eqref*{claim.minor-approx-err.1.pr:eq-1} 
and the fact that 
$\card{\lambda}=\card{\lambda\cap\lstar_i} + \card{\lambda\cap\lstar_{i+1}}$ 
yield 
Eq.~\eqref*{eq:lemma:approx-error-minoration:claim-1}.  

Now, we remark that for any $a,b,c,d >0$,
\[
\frac{abcd}{ab+cd} = \frac{1}{\frac{ab}{\max(a,c)}+\frac{cd}{\max(a,c)}}\times \min(a,c)\times bd \geq \min(a,c)\frac{bd}{b+d}
\, .
\]
Taking $a=\card{\lambda\cap\lstar_i} / \card{\lstar_i}$, $b=\card{\lstar_i}$, 
$c=\card{\lambda \cap\lstar_{i+1}} /\card{\lstar_{i+1}}$ and $d=\card{\lstar_{i+1}}$, 
we get Eq.~\eqref*{eq:approx-error-minoration:remark}. 
\end{proof}

\subsubsection{Proof of Lemma~\ref{lemma:approx-error-minoration:optimal-bound}}
\label{sec:proof-approx-error}

In fact, we prove a slightly stronger statement.
We show that, for any $n\geq 2$, 
for any $\dstar \in \{2, \ldots, n\}$, 
for any $D \in \{1, \ldots, \dstar - 1\}$ and any $\tau\in\T_n^D$,
\begin{equation}
\label{eq:approx-error-minoration:stronger}
\norm{\mustar-\mustar_{\tau}}^2 
\geq \min_{1 \leq i \leq \dstar - 1 } 
\left\{ \frac{ \card{\lstar_i}\cdot\card{\lstar_{i+1}}}{ \card{\lstar_i} + \card{\lstar_{i+1}}} 
\cdot \hilbertnorm{\mustar_{\lstar_{i+1}}-\mustar_{\lstar_i}}^2\right\}
\, . 
\end{equation}
Then, 
\[
\norm{\mustar-\mustar_{\tau}}^2 
\geq \gammainf\cdot \deltainf^2 
\qquad \text{where} \qquad 
\gammainf  
= \left( n \max_{1 \leq i \leq \dstar - 1} \left\{ \frac{1}{\abs{\lstar_i}} 
+ \frac{1}{\abs{\lstar_{i+1}}} \right\} \right)^{-1}
\, . 
\]
Since we always have 
\begin{equation}
\notag 
\lambdainf_{\taustar} 
\geq 
\gammainf 
\geq 
\frac{1}{2} \lambdainf_{\taustar} 
\, , 
\end{equation}
Eq.~\eqref{eq:approx-error-minoration:optimal-bound} follows.

\paragraph{Proof of Eq.~\eqref*{eq:approx-error-minoration:stronger} by induction}
We show by strong induction on $\dstar$ that, for any $\dstar \geq 2$, 
for any $D \in \{ 1, \ldots,  \dstar - 1\} $, any $n \geq \dstar$ and any $\tau \in\T_n^D$, 
Eq.~\eqref*{eq:approx-error-minoration:stronger} holds true. 

First, if $\dstar = 2$, 
the result follows by Eq.~\eqref*{eq:approx-error-minoration:remark} 
in Lemma~\ref{le.approx-lower} 
since we then have $i=1$ and 
\[
\frac{\card{\lambda\cap\lstar_1}}{\card{\lstar_1}}  
= \frac{\card{\lambda\cap\lstar_{2}}}{\card{\lstar_{2}}}
= 1 
\, . 
\]

Suppose now that the result is proved for all $\dstar \in \{2, \ldots, p\}$ 
and consider a change-point problem $(\taustar, \mustar)$ 
with $D_{\taustar} = \dstar = p+1$ and $n \geq p+1$. 
Let $D < p+1$ and some segmentation $\tau \in \T_n^D$ be fixed. 
Then one of these two scenarios occurs: 
(i) there exists $\lstar_i$ with $2\leq i\leq \dstar - 1$ that does not contain any change-point of~$\tau$, or 
(ii) $\lstar_2$,...,$\lstar_{\dstar-1}$ all contain a change-point of~$\tau$. 

\paragraph{Case (i)} 
Suppose that there exists an inner segment~$\lstar_i$ of~$\taustar$, 
$2\leq i\leq \dstar-1$,  
that does not contain any change-point of~$\tau$ 
(see Figure~\ref{fig:approx-error-minoration:case-i}). 
Therefore, there exists  $k \in \{1, \ldots, D \}$ such that 
$\lstar_i \subsetneqq \lambda_{k}$. 
By definition, there are $i-1$ change-points of~$\taustar$ to the left of~$\lstar_i$ 
and~$k-1$ change-points of~$\tau$ to the left of~$\lstar_i$. 
Suppose that $k<i$. 
We define~$\tauzero$ as the segmentation obtained by adding~$\taustar_i$ to~$\tau$ 
(see Figure~\ref{fig:approx-error-minoration:case-i}). 
Then $\norm{\mustar-\mustar_{\tau}}^2 \geq \norm{\mustar-\mustar_{\tauzero}}^2$ because~$\tauzero$ is finer than~$\tau$.
Reducing $\tauzero$ to a segmentation $\redtauzero$ of $\{1,2,\ldots, \taustar_i \}$ 
in~$k$ segments  and~$\taustar$ to a segmentation $\redtaustar$ 
of $\{1,2,\ldots, \taustar_i \}$ in~$i$ segments 
and defining $\redmustar = (\mustar_1, \ldots, \mustar_{\taustar_i}) \in \hilbert^i$, 
we get back to a situation covered by the induction 
since $i\leq \dstar -1$ and $k <i$. 
So, 
\begin{align*}
\norm{\redmustar-\redmustar_{\redtauzero}}^2  
&\geq \inf_{1\leq j\leq i-1}\left\{\frac{\card{\lstar_j}\cdot\card{\lstar_{j+1}}}{\card{\lstar_j} + \card{\lstar_{j+1}}} \cdot \hilbertnorm{\redmustar_{\lstar_{j+1}}-\redmustar_{\lstar_j}}^2\right\} 
\\
&\geq \inf_{1\leq j\leq \dstar-1}\left\{\frac{\card{\lstar_j}\cdot\card{\lstar_{j+1}}}{\card{\lstar_j} + \card{\lstar_{j+1}}} \cdot \hilbertnorm{\mustar_{\lstar_{j+1}}-\mustar_{\lstar_j}}^2\right\}
\end{align*}
and we get the result since 
$ \norm{\mustar-\mustar_{\tauzero}}^2 
\geq \norm{\redmustar-\redmustar_{\redtauzero}}^2  $. 
A symmetric reasonning can be applied if $k\geq i$, 
considering change-points to the right of~$\lstar_i$ and using that 
$D - k + 1 < \dstar - i + 1$ since $D<\dstar$. 

\begin{figure}[ht!]
\centering
\begin{tikzpicture}
\draw (0,0) rectangle (10,1.5) ;
\draw (0,0.5) rectangle (10,1) ;
\draw (4,1) -- (4,2) ;
\draw (5,1.25) node {$\lstar_i$} ;
\draw (5,0.75) node {$\lambda_{k}$} ;
\draw (6,1) -- (6,1.5) ;
\draw (3,-0.5) -- (3,1) ;
\draw (7,1) -- (7,0) ;
\draw (6,0) -- (6,0.5) ;
\draw (-0.5,1.75) node {$\widetilde{\tau}^{\star}$} ;
\draw (-0.5,1.25) node {$\taustar$} ;
\draw (-0.5,0.75) node {$\tau$} ;
\draw (-0.5,0.25) node {$\tauzero$} ;
\draw (-0.5,-0.25) node {$\widetilde{\tau}^{\circ}$} ;
\draw (1.5,0.75) node {$\cdots$} ;
\draw (1.5,1.25) node {$\cdots$} ;
\draw (1.5,0.25) node {$\cdots$} ;
\draw (1.5,1.75) node {$\cdots$} ;
\draw (1.5,-0.25) node {$\cdots$} ;
\draw (8.5,0.75) node {$\cdots$} ;
\draw (8.5,1.25) node {$\cdots$} ;
\draw (8.5,0.25) node {$\cdots$} ;
\draw (0,-0.5) -- (0,0) ;
\draw (0,-0.5) -- (6,-0.5) ;
\draw (6,-0.5) -- (6,0) ;
\draw (0,2) -- (6,2) ;
\draw (0,1.5) -- (0,2) ;
\draw (6,1.5) -- (6,2) ;
\end{tikzpicture}
\caption{\label{fig:approx-error-minoration:case-i}
Proof of Lemma~\ref{lemma:approx-error-minoration:optimal-bound}, Case (i):~$\lstar_i$ is a segment of~$\taustar$ that is included in a segment of~$\tau$. The segmentation~$\tauzero$ is obtained by joining~$\taustar_i$ to the segmentation~$\tau$. 
}
\end{figure}

\paragraph{Case (ii)} 
Suppose that each inner segment of~$\taustar$ contains a change-point of~$\tau$. 
Since there are $\dstar-2$ inner segments of~$\taustar$ and $D-1\leq \dstar-2$ change-points of~$\tau$, 
there is at most (hence exactly) one change-point of~$\tau$ in each inner segment of~$\taustar$.
Then $D=\dstar - 1$ and we are in the situation depicted in Figure~\ref{fig:approx-error-minoration:case-ii}.

\begin{figure}[ht!]
\centering
\begin{tikzpicture}
\draw (0,0) rectangle (10,1) ;
\draw (0,0.5) -- (10,0.5) ;
\draw (-0.5,0.75) node {$\taustar$} ;
\draw (-0.5,0.25) node {$\tau$} ;
\draw (1.25,0.5) -- (1.25,1) ;
\draw (4.0,0.5) -- (4.0,1) ;
\draw (6.5,0.5) -- (6.5,1) ;
\draw (8.7,0.5) -- (8.7,1) ;
\draw (2.5,0.5) -- (2.5,0) ;
\draw (5.9,0.5) -- (5.9,0) ;
\draw (7.5,0.5) -- (7.5,0) ;
\draw (5.3,0.75) node {$\cdots$} ;
\draw (6.8,0.25) node {$\cdots$} ;
\draw (0.6,0.75) node {$\lstar_1$} ;
\draw (2.7,0.75) node {$\lstar_{2}$} ;
\draw (7.7,0.75) node {$\lstar_{D}$} ;
\draw (9.5,0.75) node {$\lstar_{D+1}$} ;
\draw (1.5,0.25) node {$\lambda_1$} ;
\draw (4.6,0.25) node {$\lambda_2$} ;
\draw (8.8,0.25) node {$\lambda_D$} ;
\draw[dotted] (1.25,1) -- (1.25,1.5) ;
\draw[dotted] (2.5,1) -- (2.5,1.5) ;
\draw[dotted] (4.0,1) -- (4.0,1.5) ;
\draw[dotted] (6.5,1) -- (6.5,1.5) ;
\draw[dotted] (7.5,1) -- (7.5,1.5) ;
\draw[dotted] (8.7,1) -- (8.7,1.5) ;
\draw[<->] (1.25,1.5) -- (2.5,1.5) ;
\draw[<->] (2.5,1.5) -- (4.0,1.5) ;
\draw[<->] (6.5,1.5) -- (7.5,1.5) ;
\draw[<->] (7.5,1.5) -- (8.7,1.5) ;
\draw (1.9,1.75) node {$\alpha_2 \card{\lstar_{2}}$} ;
\draw (3.4,1.75) node {$(1-\alpha_2) \card{\lstar_{2}}$} ;
\draw (6.9,1.75) node {$\alpha_{D} \card{\lstar_{D}}$} ;
\draw (8.5,1.75) node {$(1-\alpha_{D})\card{\lstar_{D}}$} ;
\end{tikzpicture}
\caption{\label{fig:approx-error-minoration:case-ii}
Proof of Lemma~\ref{lemma:approx-error-minoration:optimal-bound}, Case (ii): $D=\dstar-1$ and each inner segment of~$\taustar$ contains exactly one change-point of~$\tau$.
}
\end{figure}
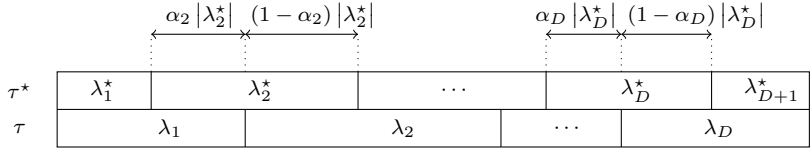

We can use Eq.~\eqref*{eq:approx-error-minoration:remark} 
in Lemma~\ref{le.approx-lower} to lower bound the contribution 
of each $\lambda \in \tau$ to $\norm{\mustar-\mustar_{\tau}}^2$. 
For $2\leq i\leq D=\dstar - 1$, define $\alpha_i\defeq \card{\lstar_i\cap\lambda_{i-1}} / \card{\lstar_i}$. 
Then, we have 
\begin{align*}
\norm{\mustar-\mustar_{\tau}}^2 
&\geq 
(1 \wedge \alpha_2) \frac{\card{\lstar_1}\cdot\card{\lstar_{2}}}{\card{\lstar_1} + \card{\lstar_{2}}} \cdot \hilbertnorm{\mustar_{\lstar_{2}}-\mustar_{\lstar_1}}^2 
\\
&\quad 
+ \sum_{j=2}^{D-1} \left( \bigl[ (1-\alpha_j)\wedge \alpha_{j+1} \bigr] 
\cdot \frac{\card{\lstar_j}\cdot\card{\lstar_{j+1}}}{\card{\lstar_j} + \card{\lstar_{j+1}}} \cdot \hilbertnorm{\mustar_{\lstar_{j+1}}-\mustar_{\lstar_j}}^2 \right) 
\\
&\quad 
+ \bigl[ (1-\alpha_D) \wedge 1 \bigr]
\frac{\card{\lstar_{D}}\cdot\card{\lstar_{D+1}}}{\card{\lstar_D} + \card{\lstar_{D+1}}} \cdot \hilbertnorm{\mustar_{\lstar_{D+1}}-\mustar_{\lstar_D}}^2 
\\
\geq 
\left[1\wedge \alpha_2 \right. & \left. +(1-\alpha_2)\wedge \alpha_3 + \cdots + (1-\alpha_{D-1})\wedge \alpha_D + (1-\alpha_{D})\wedge 1\right] 
\\
&\quad \times \inf_{1\leq j\leq \dstar-1}\left\{\frac{\card{\lstar_j}\cdot\card{\lstar_{j+1}}}{\card{\lstar_j} + \card{\lstar_{j+1}}} \cdot \hilbertnorm{\mustar_{\lstar_{j+1}}-\mustar_{\lstar_j}}^2\right\}
\, .
\end{align*}
Since $\alpha_i\geq 0$ for any $2\leq i\leq \dstar-1$, it is straightforward to show that
\[
\alpha_2 + (1-\alpha_2)\wedge \alpha_3 + \cdots + (1-\alpha_{D})\geq 1
\, , 
\]
which concludes the proof.
\qed

\subsection{Proof of Lemma~\ref{lemma:approx-error-minoration-2}} 
\label{sec:proof-approx-error-2}

Let us define 
$\delta := \min\bigl\{n\lambdainf_{\taustar},\dinf(\taustar,\tau)\bigr\}$. 
If $\delta = 0$, then Eq.~\eqref{eq:approx-error-minoration-2} holds true. 
We assume from now on that $\delta>0$.

Because $n\lambdainf_{\taustar}\geq \delta$, for any $1\leq i\leq \dstar-1$, we can write $\abs{\taustar_{i+1}-\taustar_i} \geq \delta$.
On the other hand, because $\dinf(\taustar,\tau) \geq \delta$, there exists $i\in \{1,\dots,\dstar-1\}$ such that, for any $j\in\{1,\dots, D-1\}$, $\abs{\taustar_i - \tau_j}\geq \delta$.
Since $\delta \leq n\lambdainf_{\taustar}$, this also holds true for $j=0$ and $j=D$.
Let us define, as illustrated by Figure~\ref{fig:construction-mzero-2}, 
\[\lzero \defeq \set{\taustar_i-\delta+1,\dots,\taustar_i,\taustar_i+1,\dots,\taustar_i +\delta}
\, .
\]
\begin{figure}[ht]
\centering
\begin{tikzpicture}
\def\offset{0}
\draw (0,0+\offset) rectangle (10,1.5+\offset) ;
\draw (0,1) -- (10,1) ;
\draw (0,0.5+\offset) -- (10,0.5+\offset) ;
\draw (5,1.5+\offset) -- (5,1+\offset) ;
\draw (8,1.5+\offset) -- (8,1+\offset) ;
\draw (2,1) -- (2,0.5) ;
\draw (1.5,0.73) node {$\bullet$} ;
\draw (2.5,0.73) node {$\times$} ;
\draw (9,0) -- (9,1) ;
\draw (8.5,0.73) node {$\bullet$} ;
\draw (9.5,0.73) node {$\times$} ; 
\draw (0.5,1.23+\offset) node {$\bullet$} ;
\draw (1.5,1.23+\offset) node {$\bullet$} ;
\draw (2.5,1.23+\offset) node {$\bullet$} ;
\draw (3.5,1.23+\offset) node {$\bullet$} ;
\draw (4.5,1.23+\offset) node {$\bullet$} ;
\draw (5.5,1.23+\offset) node {$\times$} ;
\draw (6.5,1.23+\offset) node {$\bullet$} ;
\draw (7.5,1.23+\offset) node {$\bullet$} ;
\draw (8.5,1.23+\offset) node {$\times$} ;
\draw (9.5,1.23+\offset) node {$\bullet$} ;
\draw (2,0+\offset) -- (2,0.5+\offset) ;
\draw (3,0+\offset) -- (3,0.5+\offset) ;
\draw (7,0+\offset) -- (7,0.5+\offset) ;
\draw (-0.5,1.25+\offset) node {$\taustar$} ;
\draw (-0.5,0.75+\offset) node {$\tau$} ;
\draw (5,0.22+\offset) node {$\lzero$} ;
\draw (4.5,1.72+\offset) node {$\taustar_i$} ;
\draw (7.5,1.72) node {$\taustar_{i+1}$} ;
\end{tikzpicture}
\caption{\label{fig:construction-mzero-2} Construction of~$\lzero$ in the proof of Lemma~\ref{lemma:approx-error-minoration-2}. 
In this case, $\delta=2$ since $\lambdainf_{\taustar}=2/10$ 
(the rightmost segment of~$\taustar$ is of size~$2$) 
and $\distinf(\taustar,\tau)=3$ (achieved in~$\taustar_i$).}
\end{figure}
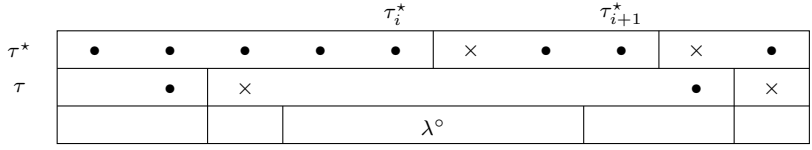

Since $\lzero$ is included in a segment of~$\tau$, 
\[
\norm{\mustar-\mustar_{\tau}}^2 
\geq \sum_{j\in \lzero}\hilbertnorm{ \mustar_j - (\mustar_{\tau})_j }^2 
\geq \sum_{j\in \lzero}\hilbertnorm{ \mustar_j - \mustar_{\lzero} }^2 
\, . 
\]
Because of the hypothesis we made,~$\lzero$ only intersects~$\lstar_i$ and~$\lstar_{i+1}$ 
among the segments of $\taustar$, 
so Eq.~\eqref*{eq:lemma:approx-error-minoration:claim-1}
in Lemma~\ref{le.approx-lower} shows that 
\begin{align*}
\sum_{j\in \lzero}\hilbertnorm{ \mustar_j - \mustar_{\lzero} }^2 
&= \frac{\card{\lzero\cap\lstar_i}\cdot \card{\lzero\cap\lstar_{i+1}}}{\card{\lzero\cap\lstar_i}+\card{\lzero\cap\lstar_{i+1}}} \hilbertnorm{ \mustar_{\lstar_{i+1}} - \mustar_{\lstar_i} }^2 \\
&= \frac{\delta}{2} \hilbertnorm{ \mustar_{\lstar_{i+1}} - \mustar_{\lstar_i} }^2 
\geq \frac{\delta}{2} \deltainf^2 
\, ,
\end{align*}
hence the result. 
\qed

\subsection{Proof of Lemma~\ref{lemma:majo-linear-part}}
\label{sec:proof-lemma-linear}

In this proof, since $\tau$ is fixed, we denote by 
$\lambda_1,\dots,\lambda_D$ the  segments of~$\tau$, 
that is,  $\lambda_i=\set{\tau_{i-1}+1,\dots,\tau_i}$.

First, notice that
\begin{equation}
\label{eq:decomp-linear}
L_{\tau} = \inner{\mustar-\mustar_{\tau}}{\varepsilon} = \sum_{i=1}^{\dstar}\hilbertinner{\mustar_{\lstar_i}}{\sum_{j\in\lstar_i}\varepsilon_j} - \sum_{i=1}^{D_{\tau}}\hilbertinner{\mustar_{\lambda_i}}{\sum_{j\in\lambda_i}\varepsilon_j}
\, .
\end{equation}
Now, if $D_{\tau} < \dstar$ we \emph{arbitrarily} define 
$\lambda_{D_{\tau}+1} = \cdots = \lambda_{\dstar} = \emptyset$, so that 
$\sum_{j\in\lambda_i}\varepsilon_j = 0$ for every $i \in \{ D_{\tau}+1 , \ldots, \dstar \}$. 
Similarly, if $\dstar < D_{\tau}$, we define 
$\lstar_{\dstar+1} = \cdots = \lambda_{D_{\tau}} = \emptyset$. 
We also define $\mustar_{\emptyset} = \mustar_n$ by convention. 
Then, defining $D^+ := \max\bigl\{ \dstar, D_{\tau}\bigr\}$, we can rewrite 
Eq.~\eqref*{eq:decomp-linear} as follows: 
\begin{align}
\notag 
L_{\tau} 
&= \sum_{i=1}^{D^+}\hilbertinner{\mustar_{\lstar_i}}{\sum_{j\in\lstar_i}\varepsilon_j} - \sum_{i=1}^{D^+}\hilbertinner{\mustar_{\lambda_i}}{\sum_{j\in\lambda_i}\varepsilon_j}
\\ \notag 
&= 
\sum_{i=1}^{D^+}\hilbertinner{\mustar_{\lstar_i} - \mustar_{\lambda_i}}{\sum_{j\in\lstar_i}\varepsilon_j}
+ \sum_{i=1}^{D^+}\hilbertinner{\mustar_{\lambda_i}}{\sum_{j\in\lstar_i}\varepsilon_j - \sum_{j\in\lambda_i}\varepsilon_j}
\\ \notag 
&= 
\sum_{i=1}^{D^+}\hilbertinner{\mustar_{\lstar_i} - \mustar_{\lambda_i}}{\sum_{j\in\lstar_i}\varepsilon_j}
+ \sum_{i=1}^{D^+}\hilbertinner{\mustar_{\lambda_i} - \mustar_n}{\sum_{j\in\lstar_i}\varepsilon_j - \sum_{j\in\lambda_i}\varepsilon_j}
\, ,
\end{align}
since 
\[
\sum_{i=1}^{D^+} \Bigl( \sum_{j\in\lstar_i}\varepsilon_j - \sum_{j\in\lambda_i}\varepsilon_j \Bigr) = 0 
\, . 
\]
Then, by the triangle inequality and Cauchy-Schwarz inequality, 
\begin{align*}
\lvert L_{\tau} \rvert 
&\leq 
\sum_{i=1}^{D^+} \hilbertnorm{\mustar_{\lstar_i} - \mustar_{\lambda_i}} \hilbertnorm{\sum_{j\in\lstar_i}\varepsilon_j}
+\; \sum_{i=1}^{D^+} \hilbertnorm{\mustar_{\lambda_i} - \mustar_n} \hilbertnorm{\sum_{j\in\lstar_i}\varepsilon_j - \sum_{j\in\lambda_i}\varepsilon_j}
\\
&\leq 
\diam\conv\set{\mustar_j \, / \, j\in\set{1,\ldots,n}}   \\
&\qquad\times \left[ 
\sum_{i=1}^{D^+} \hilbertnorm{\sum_{j\in\lstar_i}\varepsilon_j}
+ 
\;\sum_{i=1}^{D^+} \left( \hilbertnorm{\sum_{j\in\lstar_i}\varepsilon_j} 
+ \hilbertnorm{\sum_{j\in\lambda_i}\varepsilon_j} \right)
\right]
\\
&\leq 
3D^+ \diam\conv\set{\mustar_j \, / \, j\in\set{1,\ldots,n}} 
 \times \sup_{1 \leq a < b \leq n} \hilbertnorm{ \sum_{j=a}^b \varepsilon_j }
\end{align*}
where we used that $\mustar_{\lambda} \in \conv\set{\mustar_j \, / \, j\in\set{1,\ldots,n}}$ 
for any segment $\lambda$. 
Since the diameter of the convex hull of a finite set of points 
is equal to the diameter of the set, we have 
\begin{align*}
\diam\conv\set{\mustar_j \, / \, j\in\set{1,\ldots,n}} 
&= \diam\set{\mustar_j \, / \, j\in\set{1,\ldots,n}} \\
&\leq (\dstar -1) \deltasup < \dstar \deltasup 
\, .
\end{align*}
Using also Lemma~\ref{lemma:partial-sum}, we get the result. 
\qed

\subsection{Proof of Lemma~\ref{lemma:kolmogorov-moment}}\label{sec:proof:lemma:kolmogorov-moment}

Let us put $\zeta\defeq \hilbertnorm{\varepsilon_1+\cdots+\varepsilon_n}^2$. 
Since for any~$j \neq k$, $\expec{\hilbertinner{\varepsilon_j}{\varepsilon_{k}}}=0$ 
(see Remark~\ref{remark:ind}), by definition of~$v_j$,
\begin{align*}
\expec{\zeta} 
= \expec{\hilbertnorm{\varepsilon_1+\cdots+\varepsilon_n}^2}
= \sum_{j=1}^n v_j.
\end{align*}
We recognize the right-hand side of \eqref{eq:kolmogorov-moment} up to~$1/x^2$.
For any $r>1$, let us denote by~$A_r$ the event
\[
\forall 1\leq s < r, \qquad 
\hilbertnorm{\varepsilon_1+\cdots +\varepsilon_s} <x
\qquad\text{and}\qquad 
\hilbertnorm{\varepsilon_1+\cdots+\varepsilon_r}\geq x
\, ,
\]
and by~$A_1$ the event $\hilbertnorm{\varepsilon_1}\geq x$. 
These events are disjoints, thus we can write
\begin{equation} 
\label{lemma:kolmogorov-moment:pr.1} 
\proba{\max_{1\leq k\leq n} \hilbertnorm{\varepsilon_1+\cdots +\varepsilon_k} \geq x} 
= \proba{\bigcup_{r=1}^n A_r} 
= \sum_{r=1}^n\proba{A_r} 
\, . 
\end{equation}
The law of total expectation and the positiveness of~$\zeta$ yield
\[
\expec{\zeta}\geq \sum_{r=1}^n \condexpec{\zeta}{A_r}\proba{A_r}
\, .
\]
Finally, let~$\ell \leq r < k$ be integers.
Since $\varepsilon_{\ell}$ is independent from~$\varepsilon_k$ 
conditionally to~$\sigma(\varepsilon_1,\dots,\varepsilon_r)$, 
$\varepsilon_{\ell}$ is independent from~$\varepsilon_k$ conditionally to~$A_r$.
Furthermore,~$\varepsilon_k$ is independent from~$A_r$ and
\[
\condexpec{\hilbertinner{\varepsilon_k}{\varepsilon_{\ell}}}{A_r} 
= \hilbertinner{\expec{\varepsilon_k}}{\condexpec{\varepsilon_{\ell}}{A_r}} 
= 0
\, .
\]
Because of this relation and the positivity of the (real) conditional expectation, 
for any integers~$r\leq k\leq j$,
\[
\condexpec{ \zeta }{A_r} 
= \condexpec{\hilbertnorm{\varepsilon_1+\cdots+\varepsilon_n}^2}{A_r} 
\geq \condexpec{\hilbertnorm{\varepsilon_1+\cdots+\varepsilon_r}^2}{A_r} 
\geq x^2
\, . 
\]
Therefore, $\condexpec{\zeta}{A_r}\geq x^2$, which gives $\expec{\zeta} \geq x^2\sum \proba{A_r}$.
This concludes the proof, thanks to \eqref[name=Eq.~]{lemma:kolmogorov-moment:pr.1}. 
\qed

\begin{remark}
\label{remark:ind}
The independence between~$\varepsilon_j$ and~$\varepsilon_k$ for $j\neq k$ 
yields $\expec{\hilbertinner{\varepsilon_j}{\varepsilon_k}}=0$.
Indeed, we dispose of a conditional expectation on~$\hilbert$ \citep[chapter~5]{Die_Uhl:1977}, 
which satisfies the same properties than the conditional expectation with real random variables.
Hence we can write
\begin{align*}
\expec{\hilbertinner{\varepsilon_j}{\varepsilon_k}} 
= \expec{\condexpec{\hilbertinner{\varepsilon_j}{\varepsilon_k}}{\varepsilon_k}} 
&= \expec{\hilbertinner{\condexpec{\varepsilon_j}{\varepsilon_k}}{\varepsilon_k}} 
\\
&
= \expec{\hilbertinner{\expec{\varepsilon_j}}{\varepsilon_k}} = 0.
\end{align*}
Note that the~$\varepsilon_j$s expectation vanishes by hypothesis.
\end{remark}

\section*{Acknowledgments}

Damien Garreau PhD scholarship is financed by DGA / Inria.
Sylvain Arlot is also member of the Select project-team of Inria Saclay. 
At the beginning of this work, 
Sylvain Arlot was financed by CNRS and 
member of the Sierra team 
in the D\'epartement d'Informatique de l'\'Ecole normale sup\'erieure 
(CNRS / ENS / Inria UMR 8548), 45 rue d'Ulm, 75005 Paris, France. 
This work was also partly done while Sylvain Arlot 
was supported by Institut des Hautes \'Etudes Scientifiques 
(IHES, Le Bois-Marie, 35, route de Chartres, 91440 Bures-Sur-Yvette, France).
The authors thank Alain Celisse and Aymeric Dieuleveut for helpful discussions.

\bibliographystyle{imsart-nameyear}
\bibliography{learnker}

\begin{thebibliography}{56}

\bibitem[\protect\citeauthoryear{Abou-Elailah, Gouet-Brunet and
  Bloch}{2015}]{Abo_Gou_Blo:2015}
\begin{binproceedings}[author]
\bauthor{\bsnm{Abou-Elailah},~\bfnm{Abdalbassir}\binits{A.}},
  \bauthor{\bsnm{Gouet-Brunet},~\bfnm{Valerie}\binits{V.}} \AND
  \bauthor{\bsnm{Bloch},~\bfnm{Isabelle}\binits{I.}}
(\byear{2015}).
\btitle{Detection of Abrupt Changes in Spatial Relationships in Video
  Sequences}.
In \bbooktitle{International Conference on Pattern Recognition Applications and
  Methods}
\bpages{89--106}.
\bpublisher{Springer}.
\end{binproceedings}
\endbibitem

\bibitem[\protect\citeauthoryear{Abramovich et~al.}{2006}]{Abr_etal:2006}
\begin{barticle}[author]
\bauthor{\bsnm{Abramovich},~\bfnm{Felix}\binits{F.}},
  \bauthor{\bsnm{Benjamini},~\bfnm{Yoav}\binits{Y.}},
  \bauthor{\bsnm{Donoho},~\bfnm{David~L.}\binits{D.~L.}} \AND
  \bauthor{\bsnm{Johnstone},~\bfnm{Iain~M.}\binits{I.~M.}}
(\byear{2006}).
\btitle{Adapting to unknown sparsity by controlling the false discovery rate}.
\bjournal{Annals of Statistics}
\bvolume{34}
\bpages{584--653}.
\bmrnumber{MR2281879}
\end{barticle}
\endbibitem

\bibitem[\protect\citeauthoryear{Arlot and Celisse}{2011}]{Arl_Cel:2011}
\begin{barticle}[author]
\bauthor{\bsnm{Arlot},~\bfnm{Sylvain}\binits{S.}} \AND
  \bauthor{\bsnm{Celisse},~\bfnm{Alain}\binits{A.}}
(\byear{2011}).
\btitle{Segmentation of the mean of heteroscedastic data via cross-validation}.
\bjournal{Statistics and Computing}
\bvolume{21}
\bpages{613--632}.
\end{barticle}
\endbibitem

\bibitem[\protect\citeauthoryear{Arlot, Celisse and
  Harchaoui}{2012}]{Arl_Cel_Har:2012}
\begin{barticle}[author]
\bauthor{\bsnm{Arlot},~\bfnm{Sylvain}\binits{S.}},
  \bauthor{\bsnm{Celisse},~\bfnm{Alain}\binits{A.}} \AND
  \bauthor{\bsnm{Harchaoui},~\bfnm{Zaid}\binits{Z.}}
(\byear{2012}).
\btitle{A kernel multiple change-point algorithm via model selection}.
\bjournal{ArXiv e-prints}.
\bnote{Available at \url{https://arxiv.org/abs/1202.3878v2}}.
\end{barticle}
\endbibitem

\bibitem[\protect\citeauthoryear{Aronszajn}{1950}]{Aro:1950}
\begin{barticle}[author]
\bauthor{\bsnm{Aronszajn},~\bfnm{Nachman}\binits{N.}}
(\byear{1950}).
\btitle{Theory of reproducing kernels}.
\bjournal{Transactions of the American mathematical society}
\bpages{337--404}.
\end{barticle}
\endbibitem

\bibitem[\protect\citeauthoryear{Bai and Perron}{1998}]{Bai_Per:1998}
\begin{barticle}[author]
\bauthor{\bsnm{Bai},~\bfnm{Jushan}\binits{J.}} \AND
  \bauthor{\bsnm{Perron},~\bfnm{Pierre}\binits{P.}}
(\byear{1998}).
\btitle{Estimating and testing linear models with multiple structural changes}.
\bjournal{Econometrica}
\bpages{47--78}.
\end{barticle}
\endbibitem

\bibitem[\protect\citeauthoryear{Basseville and Nikiforov}{1993}]{Bas_Nik:1993}
\begin{bbook}[author]
\bauthor{\bsnm{Basseville},~\bfnm{Mich{\`e}le}\binits{M.}} \AND
  \bauthor{\bsnm{Nikiforov},~\bfnm{Igor~V.}\binits{I.~V.}}
(\byear{1993}).
\btitle{Detection of abrupt changes: theory and application}.
\bpublisher{Prentice Hall Englewood Cliffs}.
\end{bbook}
\endbibitem

\bibitem[\protect\citeauthoryear{Baudry, Maugis and
  Michel}{2012}]{Bau_Mau_Mic:2012}
\begin{barticle}[author]
\bauthor{\bsnm{Baudry},~\bfnm{Jean-Patrick}\binits{J.-P.}},
  \bauthor{\bsnm{Maugis},~\bfnm{Cathy}\binits{C.}} \AND
  \bauthor{\bsnm{Michel},~\bfnm{Bertrand}\binits{B.}}
(\byear{2012}).
\btitle{Slope heuristics: overview and implementation}.
\bjournal{Statistics and Computing}
\bvolume{22}
\bpages{455--470}.
\end{barticle}
\endbibitem

\bibitem[\protect\citeauthoryear{Bellman}{1961}]{Bel:1961}
\begin{barticle}[author]
\bauthor{\bsnm{Bellman},~\bfnm{Richard}\binits{R.}}
(\byear{1961}).
\btitle{On the approximation of curves by line segments using dynamic
  programming}.
\bjournal{Communications of the ACM}
\bvolume{4}
\bpages{284}.
\end{barticle}
\endbibitem

\bibitem[\protect\citeauthoryear{Birg{\'e} and Massart}{2001}]{Bir_Mas:2001}
\begin{barticle}[author]
\bauthor{\bsnm{Birg{\'e}},~\bfnm{Lucien}\binits{L.}} \AND
  \bauthor{\bsnm{Massart},~\bfnm{Pascal}\binits{P.}}
(\byear{2001}).
\btitle{Gaussian model selection}.
\bjournal{Journal of the European Mathematical Society}
\bvolume{3}
\bpages{203--268}.
\end{barticle}
\endbibitem

\bibitem[\protect\citeauthoryear{Birg{\'e} and Massart}{2007}]{Bir_Mas:2006}
\begin{barticle}[author]
\bauthor{\bsnm{Birg{\'e}},~\bfnm{Lucien}\binits{L.}} \AND
  \bauthor{\bsnm{Massart},~\bfnm{Pascal}\binits{P.}}
(\byear{2007}).
\btitle{{Minimal penalties for {G}aussian model selection}}.
\bjournal{Probability Theory and Related Fields}
\bvolume{138}
\bpages{33--73}.
\bmrnumber{MR2288064}
\end{barticle}
\endbibitem

\bibitem[\protect\citeauthoryear{Boysen et~al.}{2009}]{Boy_etal:2006}
\begin{barticle}[author]
\bauthor{\bsnm{Boysen},~\bfnm{Leif}\binits{L.}},
  \bauthor{\bsnm{Kempe},~\bfnm{Angela}\binits{A.}},
  \bauthor{\bsnm{Liebscher},~\bfnm{Volkmar}\binits{V.}},
  \bauthor{\bsnm{Munk},~\bfnm{Axel}\binits{A.}} \AND
  \bauthor{\bsnm{Wittich},~\bfnm{Olaf}\binits{O.}}
(\byear{2009}).
\btitle{Consistencies and rates of convergence of jump-penalized least squares
  estimators}.
\bjournal{Annals of Statistics}
\bvolume{37}
\bpages{157--183}.
\bdoi{10.1214/07-AOS558}
\bmrnumber{2488348 (2010c:62099)}
\end{barticle}
\endbibitem

\bibitem[\protect\citeauthoryear{Brodsky and Darkhovsky}{2013}]{Bro_Dar:2013}
\begin{bbook}[author]
\bauthor{\bsnm{Brodsky},~\bfnm{Boris~E.}\binits{B.~E.}} \AND
  \bauthor{\bsnm{Darkhovsky},~\bfnm{Boris~S.}\binits{B.~S.}}
(\byear{2013}).
\btitle{Nonparametric methods in change point problems}
\bvolume{243}.
\bpublisher{Springer Science \& Business Media}.
\end{bbook}
\endbibitem

\bibitem[\protect\citeauthoryear{Brunel}{2014}]{Bru:2014}
\begin{barticle}[author]
\bauthor{\bsnm{Brunel},~\bfnm{Victor-Emmanuel}\binits{V.-E.}}
(\byear{2014}).
\btitle{Convex set detection}.
\bjournal{ArXiv e-prints}.
\bnote{Available at \url{https://arxiv.org/abs/1404.6224}}.
\end{barticle}
\endbibitem

\bibitem[\protect\citeauthoryear{Carlstein}{1988}]{Car:1988}
\begin{barticle}[author]
\bauthor{\bsnm{Carlstein},~\bfnm{Edward}\binits{E.}}
(\byear{1988}).
\btitle{Nonparametric change-point estimation}.
\bjournal{Annals of Statistics}
\bpages{188--197}.
\end{barticle}
\endbibitem

\bibitem[\protect\citeauthoryear{Celisse et~al.}{2016}]{Cel_Mor_Mar_Rig:2016}
\begin{barticle}[author]
\bauthor{\bsnm{Celisse},~\bfnm{Alain}\binits{A.}},
  \bauthor{\bsnm{Marot},~\bfnm{Guillemette}\binits{G.}},
  \bauthor{\bsnm{Pierre-Jean},~\bfnm{Morgane}\binits{M.}} \AND
  \bauthor{\bsnm{Rigaill},~\bfnm{Guillem}\binits{G.}}
(\byear{2016}).
\btitle{{New efficient algorithms for multiple change-point detection with
  kernels}}.
\bjournal{Preprint}.
\bnote{Available at \url{https://hal.inria.fr/hal-01413230}}.
\end{barticle}
\endbibitem

\bibitem[\protect\citeauthoryear{Comte and Rozenholc}{2004}]{Com_Roz:2004}
\begin{barticle}[author]
\bauthor{\bsnm{Comte},~\bfnm{Fabienne}\binits{F.}} \AND
  \bauthor{\bsnm{Rozenholc},~\bfnm{Yves}\binits{Y.}}
(\byear{2004}).
\btitle{{A new algorithm for fixed design regression and denoising}}.
\bjournal{Annals of the Institute of Statistical Mathematics}
\bvolume{56}
\bpages{449--473}.
\bmrnumber{MR2095013 (2005e:62081)}
\end{barticle}
\endbibitem

\bibitem[\protect\citeauthoryear{Desobry, Davy and
  Doncarli}{2005}]{Des_Dav_Don:2005}
\begin{barticle}[author]
\bauthor{\bsnm{Desobry},~\bfnm{Fr{\'e}d{\'e}ric}\binits{F.}},
  \bauthor{\bsnm{Davy},~\bfnm{Manuel}\binits{M.}} \AND
  \bauthor{\bsnm{Doncarli},~\bfnm{Christian}\binits{C.}}
(\byear{2005}).
\btitle{An online kernel change detection algorithm}.
\bjournal{Signal Processing, IEEE Transactions on}
\bvolume{53}
\bpages{2961--2974}.
\end{barticle}
\endbibitem

\bibitem[\protect\citeauthoryear{Diestel and Uhl}{1977}]{Die_Uhl:1977}
\begin{bbook}[author]
\bauthor{\bsnm{Diestel},~\bfnm{Joseph}\binits{J.}} \AND
  \bauthor{\bsnm{Uhl},~\bfnm{John~Jerry}\binits{J.~J.}}
(\byear{1977}).
\btitle{Vector measures}
\bvolume{15}.
\bpublisher{American Mathematical Soc.}
\end{bbook}
\endbibitem

\bibitem[\protect\citeauthoryear{Dieuleveut and Bach}{2016}]{Die_Bac:2014}
\begin{barticle}[author]
\bauthor{\bsnm{Dieuleveut},~\bfnm{Aymeric}\binits{A.}} \AND
  \bauthor{\bsnm{Bach},~\bfnm{Francis}\binits{F.}}
(\byear{2016}).
\btitle{Nonparametric stochastic approximation with large step-sizes}.
\bjournal{Annals of Statistics}
\bvolume{44}
\bpages{1363--1399}.
\end{barticle}
\endbibitem

\bibitem[\protect\citeauthoryear{Fisher}{1958}]{Fis:1958}
\begin{barticle}[author]
\bauthor{\bsnm{Fisher},~\bfnm{Walter~D}\binits{W.~D.}}
(\byear{1958}).
\btitle{On grouping for maximum homogeneity}.
\bjournal{Journal of the American Statistical Association}
\bvolume{53}
\bpages{789--798}.
\end{barticle}
\endbibitem

\bibitem[\protect\citeauthoryear{Fryzlewicz}{2014}]{Fry:2014}
\begin{barticle}[author]
\bauthor{\bsnm{Fryzlewicz},~\bfnm{Piotr}\binits{P.}}
(\byear{2014}).
\btitle{Wild binary segmentation for multiple change-point detection}.
\bjournal{The Annals of Statistics}
\bvolume{42}
\bpages{2243--2281}.
\end{barticle}
\endbibitem

\bibitem[\protect\citeauthoryear{Fukumizu et~al.}{2008}]{Fuk_etal:2008}
\begin{bincollection}[author]
\bauthor{\bsnm{Fukumizu},~\bfnm{Kenji}\binits{K.}},
  \bauthor{\bsnm{Gretton},~\bfnm{Arthur}\binits{A.}},
  \bauthor{\bsnm{Sun},~\bfnm{Xiaohai}\binits{X.}} \AND
  \bauthor{\bsnm{Sch\"{o}lkopf},~\bfnm{Bernhard}\binits{B.}}
(\byear{2008}).
\btitle{Kernel Measures of Conditional Dependence}.
In \bbooktitle{Advances in Neural Information Processing Systems 20}
\bpages{489--496}.
\bpublisher{Curran Associates, Inc.}
\end{bincollection}
\endbibitem

\bibitem[\protect\citeauthoryear{Gretton et~al.}{2006}]{Gre_Bor_Ras:2006}
\begin{binproceedings}[author]
\bauthor{\bsnm{Gretton},~\bfnm{Arthur}\binits{A.}},
  \bauthor{\bsnm{Borgwardt},~\bfnm{Karsten~M.}\binits{K.~M.}},
  \bauthor{\bsnm{Rasch},~\bfnm{Malte}\binits{M.}},
  \bauthor{\bsnm{Sch{\"o}lkopf},~\bfnm{Bernhard}\binits{B.}} \AND
  \bauthor{\bsnm{Smola},~\bfnm{Alex~J.}\binits{A.~J.}}
(\byear{2006}).
\btitle{A kernel method for the two-sample-problem}.
In \bbooktitle{Advances in neural information processing systems}
\bpages{513--520}.
\end{binproceedings}
\endbibitem

\bibitem[\protect\citeauthoryear{Gretton et~al.}{2012}]{Gre_Sej_Str:2012}
\begin{binproceedings}[author]
\bauthor{\bsnm{Gretton},~\bfnm{Arthur}\binits{A.}},
  \bauthor{\bsnm{Sejdinovic},~\bfnm{Dino}\binits{D.}},
  \bauthor{\bsnm{Strathmann},~\bfnm{Heiko}\binits{H.}},
  \bauthor{\bsnm{Balakrishnan},~\bfnm{Sivaraman}\binits{S.}},
  \bauthor{\bsnm{Pontil},~\bfnm{Massimiliano}\binits{M.}},
  \bauthor{\bsnm{Fukumizu},~\bfnm{Kenji}\binits{K.}} \AND
  \bauthor{\bsnm{Sriperumbudur},~\bfnm{Bharath~K.}\binits{B.~K.}}
(\byear{2012}).
\btitle{Optimal kernel choice for large-scale two-sample tests}.
In \bbooktitle{Advances in Neural Information Processing Systems}
\bpages{1205--1213}.
\end{binproceedings}
\endbibitem

\bibitem[\protect\citeauthoryear{H{\'a}jek and R{\'e}nyi}{1955}]{Haj_Ren:1955}
\begin{barticle}[author]
\bauthor{\bsnm{H{\'a}jek},~\bfnm{Jaroslav}\binits{J.}} \AND
  \bauthor{\bsnm{R{\'e}nyi},~\bfnm{Alfr{\'e}d}\binits{A.}}
(\byear{1955}).
\btitle{Generalization of an inequality of {K}olmogorov}.
\bjournal{Acta Mathematica Hungarica}
\bvolume{6}
\bpages{281--283}.
\end{barticle}
\endbibitem

\bibitem[\protect\citeauthoryear{Harchaoui and Capp{\'e}}{2007}]{Har_Cap:2007}
\begin{binproceedings}[author]
\bauthor{\bsnm{Harchaoui},~\bfnm{Zaid}\binits{Z.}} \AND
  \bauthor{\bsnm{Capp{\'e}},~\bfnm{Olivier}\binits{O.}}
(\byear{2007}).
\btitle{Retrospective multiple change-point estimation with kernels}.
In \bbooktitle{IEEE Workshop on Statistical Signal Processing}
\bpages{768--772}.
\end{binproceedings}
\endbibitem

\bibitem[\protect\citeauthoryear{Harchaoui, Moulines and
  Bach}{2009}]{Har_Mou_Bac:2009}
\begin{binproceedings}[author]
\bauthor{\bsnm{Harchaoui},~\bfnm{Zaid}\binits{Z.}},
  \bauthor{\bsnm{Moulines},~\bfnm{Eric}\binits{E.}} \AND
  \bauthor{\bsnm{Bach},~\bfnm{Francis~R.}\binits{F.~R.}}
(\byear{2009}).
\btitle{Kernel change-point analysis}.
In \bbooktitle{Advances in neural information processing systems}
\bpages{609--616}.
\end{binproceedings}
\endbibitem

\bibitem[\protect\citeauthoryear{Hubert and Arabie}{1985}]{Hub_Ara:1985}
\begin{barticle}[author]
\bauthor{\bsnm{Hubert},~\bfnm{Lawrence}\binits{L.}} \AND
  \bauthor{\bsnm{Arabie},~\bfnm{Phipps}\binits{P.}}
(\byear{1985}).
\btitle{Comparing partitions}.
\bjournal{Journal of classification}
\bvolume{2}
\bpages{193--218}.
\end{barticle}
\endbibitem

\bibitem[\protect\citeauthoryear{Kim et~al.}{2009}]{Kim_Mar_Per:2009}
\begin{barticle}[author]
\bauthor{\bsnm{Kim},~\bfnm{Albert~Y.}\binits{A.~Y.}},
  \bauthor{\bsnm{Marzban},~\bfnm{Caren}\binits{C.}},
  \bauthor{\bsnm{Percival},~\bfnm{Donald~B.}\binits{D.~B.}} \AND
  \bauthor{\bsnm{Stuetzle},~\bfnm{Werner}\binits{W.}}
(\byear{2009}).
\btitle{Using labeled data to evaluate change detectors in a multivariate
  streaming environment}.
\bjournal{Signal Processing}
\bvolume{89}
\bpages{2529--2536}.
\end{barticle}
\endbibitem

\bibitem[\protect\citeauthoryear{Kolmogorov}{1928}]{Kol:1928}
\begin{barticle}[author]
\bauthor{\bsnm{Kolmogorov},~\bfnm{Andrei~N.}\binits{A.~N.}}
(\byear{1928}).
\btitle{{\"U}ber die {S}ummen durch den {Z}ufall bestimmten unabh{\"a}ngigen
  Gr{\"o}{\ss}en}.
\bjournal{Mathematische Annalen}
\bvolume{99}
\bpages{484--488}.
\end{barticle}
\endbibitem

\bibitem[\protect\citeauthoryear{Korostelev}{1988}]{Kor:1988}
\begin{barticle}[author]
\bauthor{\bsnm{Korostelev},~\bfnm{Aleksandr~P.}\binits{A.~P.}}
(\byear{1988}).
\btitle{On minimax estimation of a discontinuous signal}.
\bjournal{Theory of Probability \& Its Applications}
\bvolume{32}
\bpages{727--730}.
\end{barticle}
\endbibitem

\bibitem[\protect\citeauthoryear{Korostelev and Tsybakov}{2012}]{Kor_Tsy:2012}
\begin{bbook}[author]
\bauthor{\bsnm{Korostelev},~\bfnm{Aleksandr~P.}\binits{A.~P.}} \AND
  \bauthor{\bsnm{Tsybakov},~\bfnm{Alexandre~B.}\binits{A.~B.}}
(\byear{2012}).
\btitle{Minimax theory of image reconstruction}
\bvolume{82}.
\bpublisher{Springer Science \& Business Media}.
\end{bbook}
\endbibitem

\bibitem[\protect\citeauthoryear{Lai et~al.}{2005}]{Lai_Joh_Kuc:2005}
\begin{barticle}[author]
\bauthor{\bsnm{Lai},~\bfnm{Weil~R.}\binits{W.~R.}},
  \bauthor{\bsnm{Johnson},~\bfnm{Mark~D.}\binits{M.~D.}},
  \bauthor{\bsnm{Kucherlapati},~\bfnm{Raju}\binits{R.}} \AND
  \bauthor{\bsnm{Park},~\bfnm{Peter~J.}\binits{P.~J.}}
(\byear{2005}).
\btitle{Comparative analysis of algorithms for identifying amplifications and
  deletions in array {CGH} data}.
\bjournal{Bioinformatics}
\bvolume{21}
\bpages{3763--3770}.
\end{barticle}
\endbibitem

\bibitem[\protect\citeauthoryear{Lajugie, Arlot and
  Bach}{2014}]{Laj_Arl_Bac:2014}
\begin{binproceedings}[author]
\bauthor{\bsnm{Lajugie},~\bfnm{R{\'e}mi}\binits{R.}},
  \bauthor{\bsnm{Arlot},~\bfnm{Sylvain}\binits{S.}} \AND
  \bauthor{\bsnm{Bach},~\bfnm{Francis}\binits{F.}}
(\byear{2014}).
\btitle{Large-margin metric learning for constrained partitioning problems}.
In \bbooktitle{Proceedings of The 31st International Conference on Machine
  Learning}
\bpages{297--305}.
\end{binproceedings}
\endbibitem

\bibitem[\protect\citeauthoryear{Lavielle}{2005}]{Lav:2005}
\begin{barticle}[author]
\bauthor{\bsnm{Lavielle},~\bfnm{Marc}\binits{M.}}
(\byear{2005}).
\btitle{Using penalized contrasts for the change-point problem}.
\bjournal{Signal processing}
\bvolume{85}
\bpages{1501--1510}.
\end{barticle}
\endbibitem

\bibitem[\protect\citeauthoryear{Lavielle and Moulines}{2000}]{Lav_Mou:2000}
\begin{barticle}[author]
\bauthor{\bsnm{Lavielle},~\bfnm{Marc}\binits{M.}} \AND
  \bauthor{\bsnm{Moulines},~\bfnm{Eric}\binits{E.}}
(\byear{2000}).
\btitle{Least-squares Estimation of an Unknown Number of Shifts in a Time
  Series}.
\bjournal{Journal of time series analysis}
\bvolume{21}
\bpages{33--59}.
\end{barticle}
\endbibitem

\bibitem[\protect\citeauthoryear{Lavielle and Teyssiere}{2006}]{Lav_Tey:2006}
\begin{barticle}[author]
\bauthor{\bsnm{Lavielle},~\bfnm{Marc}\binits{M.}} \AND
  \bauthor{\bsnm{Teyssiere},~\bfnm{Gilles}\binits{G.}}
(\byear{2006}).
\btitle{Detection of multiple change-points in multivariate time series}.
\bjournal{Lithuanian Mathematical Journal}
\bvolume{46}
\bpages{287--306}.
\end{barticle}
\endbibitem

\bibitem[\protect\citeauthoryear{Lebarbier}{2005}]{Leb:2005}
\begin{barticle}[author]
\bauthor{\bsnm{Lebarbier},~\bfnm{{\'E}milie}\binits{{\'E}.}}
(\byear{2005}).
\btitle{Detecting multiple change-points in the mean of a Gaussian process by
  model selection}.
\bjournal{Signal Proces.}
\bvolume{85}
\bpages{717--736}.
\end{barticle}
\endbibitem

\bibitem[\protect\citeauthoryear{Ledoux and Talagrand}{2013}]{Led_Tal:2013}
\begin{bbook}[author]
\bauthor{\bsnm{Ledoux},~\bfnm{Michel}\binits{M.}} \AND
  \bauthor{\bsnm{Talagrand},~\bfnm{Michel}\binits{M.}}
(\byear{2013}).
\btitle{Probability in Banach Spaces: isoperimetry and processes}
\bvolume{23}.
\bpublisher{Springer Science \& Business Media}.
\end{bbook}
\endbibitem

\bibitem[\protect\citeauthoryear{Li et~al.}{2015}]{Li_Xie_Dai:2015}
\begin{barticle}[author]
\bauthor{\bsnm{Li},~\bfnm{Shuang}\binits{S.}},
  \bauthor{\bsnm{Xie},~\bfnm{Yao}\binits{Y.}},
  \bauthor{\bsnm{Dai},~\bfnm{Hanjun}\binits{H.}} \AND
  \bauthor{\bsnm{Song},~\bfnm{Le}\binits{L.}}
(\byear{2015}).
\btitle{{$M$}-Statistic for Kernel Change-Point Detection}.
\bjournal{Advances in Neural Information Processing Systems}
\bpages{3366--3374}.
\end{barticle}
\endbibitem

\bibitem[\protect\citeauthoryear{Liu, Wu and Zidek}{1997}]{Liu_Wu_Zid:1997}
\begin{barticle}[author]
\bauthor{\bsnm{Liu},~\bfnm{Jian}\binits{J.}},
  \bauthor{\bsnm{Wu},~\bfnm{Shiying}\binits{S.}} \AND
  \bauthor{\bsnm{Zidek},~\bfnm{James~V.}\binits{J.~V.}}
(\byear{1997}).
\btitle{On segmented multivariate regression}.
\bjournal{Statistica Sinica}
\bvolume{7}
\bpages{497--525}.
\end{barticle}
\endbibitem

\bibitem[\protect\citeauthoryear{Liu et~al.}{2017}]{Liu_Suz_Rel:2014}
\begin{barticle}[author]
\bauthor{\bsnm{Liu},~\bfnm{Song}\binits{S.}},
  \bauthor{\bsnm{Suzuki},~\bfnm{Taiji}\binits{T.}},
  \bauthor{\bsnm{Relator},~\bfnm{Raissa}\binits{R.}},
  \bauthor{\bsnm{Sese},~\bfnm{Jun}\binits{J.}},
  \bauthor{\bsnm{Sugiyama},~\bfnm{Masashi}\binits{M.}} \AND
  \bauthor{\bsnm{Fukumizu},~\bfnm{Kenji}\binits{K.}}
(\byear{2017}).
\btitle{Support consistency of direct sparse-change learning in Markov
  networks}.
\bjournal{Annals of Statistics (accepted)}.
\bnote{arXiv:1407.0581}.
\end{barticle}
\endbibitem

\bibitem[\protect\citeauthoryear{Page}{1955}]{Pag:1955}
\begin{barticle}[author]
\bauthor{\bsnm{Page},~\bfnm{E.~S.}\binits{E.~S.}}
(\byear{1955}).
\btitle{A test for a change in a parameter occurring at an unknown point}.
\bjournal{Biometrika}
\bpages{523--527}.
\end{barticle}
\endbibitem

\bibitem[\protect\citeauthoryear{Ritov, Raz and
  Bergman}{2002}]{Rit_Raz_Ber:2002}
\begin{barticle}[author]
\bauthor{\bsnm{Ritov},~\bfnm{Ya'acov}\binits{Y.}},
  \bauthor{\bsnm{Raz},~\bfnm{Aeyal}\binits{A.}} \AND
  \bauthor{\bsnm{Bergman},~\bfnm{Hagai}\binits{H.}}
(\byear{2002}).
\btitle{Detection of onset of neuronal activity by allowing for heterogeneity
  in the change points}.
\bjournal{Journal of neuroscience methods}
\bvolume{122}
\bpages{25--42}.
\end{barticle}
\endbibitem

\bibitem[\protect\citeauthoryear{Sch{\"o}lkopf and Smola}{2002}]{Sch_Smo:2002}
\begin{bbook}[author]
\bauthor{\bsnm{Sch{\"o}lkopf},~\bfnm{Bernhard}\binits{B.}} \AND
  \bauthor{\bsnm{Smola},~\bfnm{Alexander~J.}\binits{A.~J.}}
(\byear{2002}).
\btitle{Learning with kernels: Support vector machines, regularization,
  optimization, and beyond}.
\bpublisher{MIT press}.
\end{bbook}
\endbibitem

\bibitem[\protect\citeauthoryear{Shao}{1997}]{Sha:1997}
\begin{barticle}[author]
\bauthor{\bsnm{Shao},~\bfnm{Jun}\binits{J.}}
(\byear{1997}).
\btitle{An asymptotic theory for linear model selection}.
\bjournal{Statistica Sinica}
\bvolume{7}
\bpages{221--242}.
\end{barticle}
\endbibitem

\bibitem[\protect\citeauthoryear{Sharipov, Tewes and
  Wendler}{2016}]{Sha_Tew_Wen:2016}
\begin{barticle}[author]
\bauthor{\bsnm{Sharipov},~\bfnm{Olimjon}\binits{O.}},
  \bauthor{\bsnm{Tewes},~\bfnm{Johannes}\binits{J.}} \AND
  \bauthor{\bsnm{Wendler},~\bfnm{Martin}\binits{M.}}
(\byear{2016}).
\btitle{Sequential block bootstrap in a {H}ilbert space with application to
  change point analysis}.
\bjournal{The Canadian Journal of Statistics. La Revue Canadienne de
  Statistique}
\bvolume{44}
\bpages{300--322}.
\bdoi{10.1002/cjs.11293}
\bmrnumber{3536199}
\end{barticle}
\endbibitem

\bibitem[\protect\citeauthoryear{Spokoiny}{2009}]{Spo:2009}
\begin{barticle}[author]
\bauthor{\bsnm{Spokoiny},~\bfnm{Vladimir}\binits{V.}}
(\byear{2009}).
\btitle{Multiscale local change point detection with applications to
  value-at-risk}.
\bjournal{Annals of Statistics}
\bpages{1405--1436}.
\end{barticle}
\endbibitem

\bibitem[\protect\citeauthoryear{Sriperumbudur et~al.}{2009}]{Sri_Fuk_Gre:2009}
\begin{bincollection}[author]
\bauthor{\bsnm{Sriperumbudur},~\bfnm{Bharath~K.}\binits{B.~K.}},
  \bauthor{\bsnm{Fukumizu},~\bfnm{Kenji}\binits{K.}},
  \bauthor{\bsnm{Gretton},~\bfnm{Arthur}\binits{A.}},
  \bauthor{\bsnm{Lanckriet},~\bfnm{Gert R.~G.}\binits{G.~R.~G.}} \AND
  \bauthor{\bsnm{Sch{\"o}lkopf},~\bfnm{Bernhard}\binits{B.}}
(\byear{2009}).
\btitle{Kernel Choice and Classifiability for {RKHS} Embeddings of Probability
  Distributions}.
In \bbooktitle{Advances in Neural Information Processing Systems},
\bvolume{21}
\bpublisher{NIPS Foundation}.
\end{bincollection}
\endbibitem

\bibitem[\protect\citeauthoryear{Tartakovsky, Nikiforov and
  Basseville}{2014}]{Tar_Bas_Nik:2014}
\begin{bbook}[author]
\bauthor{\bsnm{Tartakovsky},~\bfnm{Alexander}\binits{A.}},
  \bauthor{\bsnm{Nikiforov},~\bfnm{Igor~V.}\binits{I.~V.}} \AND
  \bauthor{\bsnm{Basseville},~\bfnm{Mich{\`e}le}\binits{M.}}
(\byear{2014}).
\btitle{Sequential Analysis: Hypothesis Testing and Changepoint Detection}.
\bseries{Monographs on Statistics and Applied Probability}
\bvolume{136}.
\bpublisher{Chapman and Hall/CRC}, \baddress{Boca Raton, FL}.
\end{bbook}
\endbibitem

\bibitem[\protect\citeauthoryear{Vogt and Dette}{2015}]{Vog_Det:2015}
\begin{barticle}[author]
\bauthor{\bsnm{Vogt},~\bfnm{Michael}\binits{M.}} \AND
  \bauthor{\bsnm{Dette},~\bfnm{Holger}\binits{H.}}
(\byear{2015}).
\btitle{Detecting gradual changes in locally stationary processes}.
\bjournal{Annals of Statistics}
\bvolume{43}
\bpages{713--740}.
\end{barticle}
\endbibitem

\bibitem[\protect\citeauthoryear{Wang and Samworth}{2016}]{Wan_Sam:2016}
\begin{bmisc}[author]
\bauthor{\bsnm{Wang},~\bfnm{Tengyao}\binits{T.}} \AND
  \bauthor{\bsnm{Samworth},~\bfnm{Richard~J.}\binits{R.~J.}}
(\byear{2016}).
\btitle{High-dimensional changepoint estimation via sparse projection}.
\bnote{\url{https://arxiv.org/abs/1606.06246}}.
\end{bmisc}
\endbibitem

\bibitem[\protect\citeauthoryear{Yao}{1988}]{Yao:1988}
\begin{barticle}[author]
\bauthor{\bsnm{Yao},~\bfnm{Yi-Ching}\binits{Y.-C.}}
(\byear{1988}).
\btitle{Estimating the number of change-points via {S}chwarz' criterion}.
\bjournal{Statistics \& Probability Letters}
\bvolume{6}
\bpages{181--189}.
\end{barticle}
\endbibitem

\bibitem[\protect\citeauthoryear{Yao and Au}{1989}]{Yao_Au:1989}
\begin{barticle}[author]
\bauthor{\bsnm{Yao},~\bfnm{Yi-Ching}\binits{Y.-C.}} \AND
  \bauthor{\bsnm{Au},~\bfnm{Siu-Tong}\binits{S.-T.}}
(\byear{1989}).
\btitle{Least-squares estimation of a step function}.
\bjournal{Sankhy{\=a}: The Indian Journal of Statistics, Series A}
\bpages{370--381}.
\end{barticle}
\endbibitem

\bibitem[\protect\citeauthoryear{Zou et~al.}{2014}]{Zou_Yin_Fen:2014}
\begin{barticle}[author]
\bauthor{\bsnm{Zou},~\bfnm{Changliang}\binits{C.}},
  \bauthor{\bsnm{Yin},~\bfnm{Guosheng}\binits{G.}},
  \bauthor{\bsnm{Feng},~\bfnm{Long}\binits{L.}} \AND
  \bauthor{\bsnm{Wang},~\bfnm{Zhaojun}\binits{Z.}}
(\byear{2014}).
\btitle{Nonparametric maximum likelihood approach to multiple change-point
  problems}.
\bjournal{Annals of Statistics}
\bvolume{42}
\bpages{970--1002}.
\bdoi{10.1214/14-AOS1210}
\bmrnumber{3210993}
\end{barticle}
\endbibitem

\end{thebibliography}

\end{document}